\documentclass[11pt,a4paper]{amsart}
\pdfoutput=1 

\usepackage{amssymb}
\usepackage{amscd}
\usepackage{amsmath,amsfonts,stmaryrd, graphicx, graphics}
\usepackage[usenames,dvipsnames,svgnames]{xcolor}
\usepackage{centernot}
\usepackage{stackengine}
\usepackage[
backend=bibtex,
style=alphabetic,
sorting=nty,
giveninits=true
]{biblatex}
\usepackage{mathtools}

\usepackage{silence}
\usepackage[linktoc=all,colorlinks, urlcolor=Navy, linkcolor=Navy, citecolor=Navy]{hyperref}
\hypersetup{
    pdftitle={Combinatorics of quasi-hereditary structures},    
}
\usepackage{microtype}
\usepackage{rotating}
\usepackage{lscape}
\usepackage{afterpage}
\usepackage[capitalise]{cleveref}

\usepackage{adjustbox}
\usepackage{array}
\usepackage{booktabs,multirow,siunitx}
\newcolumntype{x}[1]{>{\centering\let\newline\\\arraybackslash\hspace{0pt}}m{#1}}

\usepackage[all]{xy}

\usepackage{tikz}
\usepackage{tikz-cd}
\usetikzlibrary{arrows,shapes,chains,matrix,positioning,scopes,cd,patterns}
\usetikzlibrary{decorations.pathmorphing}
\usetikzlibrary{decorations.pathreplacing,calligraphy}
\usetikzlibrary{backgrounds,calc,positioning}
\usetikzlibrary{arrows.meta}

\pgfdeclarelayer{bg}    
\pgfsetlayers{bg,main} 

\theoremstyle{plain}
\newtheorem{thm}{Theorem}[section]
\newtheorem{lem}[thm]{Lemma}

\newtheorem{pro}[thm]{Proposition}
\newtheorem{cor}[thm]{Corollary}
\newtheorem{thm*}{Theorem}
\theoremstyle{definition}
\newtheorem{dfn}[thm]{Definition}
\newtheorem{exa}[thm]{Example}
\newtheorem{rem}[thm]{Remark}
\newtheorem{conj}[thm]{Conjecture}
\newtheorem*{notation-convention}{Notation and conventions}

\numberwithin{equation}{section}

\DeclareMathOperator{\Hom}{Hom}

\DeclareMathOperator{\End}{End}
\DeclareMathOperator{\Ext}{Ext}

\DeclareMathOperator{\modu}{mod}

\newcommand{\qhstr}{\mathsf{qh}\text{.}\mathsf{str}}
\newcommand{\Dec}{\mathsf{Dec}}
\newcommand{\Inc}{\mathsf{Inc}}
\newcommand{\lhdp}{\lhd^{\prime}}
\newcommand{\Deltap}{\Delta^{\prime}}
\newcommand{\xto}[1]{\xrightarrow{#1}}

\newcommand\flowsfroma{\mathrel{\vcenter{\hbox{\rotatebox{180}{$\leadsto$}}}}}

\newcommand{\old}[1]{{\color{red} #1}}
\newcommand{\new}[1]{{\color{blue} #1}}
\newcommand{\p}{\prime}
\newcommand{\ideal}[1]{\left\langle #1 \right\rangle}

\newcommand{\kk}{\mathbf{k}}

\DeclareMathOperator{\id}{id}

\renewcommand{\bar}{\overline}

\newcommand{\op}{\mathrm{op}}

\newcommand{\hl}{\mathrel{\widehat{\lhd}}}
\DeclareMathOperator{\Ker}{ker}
\DeclareMathOperator{\coKer}{coker}
\DeclareMathOperator{\im}{im}

\DeclareMathOperator{\soc}{soc}
\DeclareMathOperator{\Top}{top}

\DeclareMathOperator{\add}{add}

\newcommand{\La}{\Lambda}

\newcommand{\A}{\mathbb{A}}

\newcommand{\D}{\mathbb{D}}
\newcommand{\E}{\mathbb{E}}

\newcommand{\Z}{\mathbb{Z}}

\newcommand{\mfq}{\mathfrak{q}}

\newcommand{\mcC}{\mathcal{C}}

\newcommand{\mcF}{\mathcal{F}}

\usepackage{todonotes}
\newcommand{\YKeditcomment}[1]{\mbox{}\todo[inline, color=green!10]{YK: #1}}
\newcommand{\BReditcomment}[1]{\mbox{}\todo[inline, color=green!10]{BR: #1}}

\renewcommand{\epsilon}{\varepsilon}
\renewcommand{\tilde}{\widetilde}

\begin{document}
\title{Combinatorics of quasi-hereditary structures}
\author{Manuel Flores}
\author{Yuta Kimura}
\author{Baptiste Rognerud}

\address{Manuel Flores: Fakult\"at f\"ur Mathematik, Universit\"at Bielefeld, 33501 Bielefeld, Germany}
\email{mflores@math.uni-bielefeld.de}

\address{Yuta Kimura: Fakult\"at f\"ur Mathematik, Universit\"at Bielefeld, 33501 Bielefeld, Germany}
\email{ykimura@math.uni-bielefeld.de} 

\address{Baptiste Rognerud: Universit\'e Paris Diderot, Institut de math\'ematiques de Jussieu - Paris Rive Gauche, 75013 Paris, France}
\email{baptiste.rognerud@imj-prg.fr}



\thanks{The first author has been supported by the Consejo Nacional de Ciencia y Tecnolog\'ia, M\'exico, (CONACYT), the Deutscher Akademischer Austauschdienst (DAAD), and the  Bielefeld University through the Bielefelder Nachwuchsfonds programme. The second author has been supported by the Alexander von Humboldt Foundation 
in the framework of an Alexander von Humboldt Professorship 
endowed by the German Federal Ministry of Education and Research.}


\begin{abstract}
A quasi-hereditary algebra is an Artin algebra together with a partial order on its set of isomorphism classes of simple modules which satisfies certain conditions. 
In this article we investigate all the possible choices that yield quasi-hereditary structures on a given algebra, in particular we introduce and study what we call the poset of quasi-hereditary structures. Our techniques involve certain quiver decompositions and idempotent reductions. For a path algebra of Dynkin type $\A$, we provide a full classification of its quasi-hereditary structures.
For types $\D$ and $\E$, we give a counting method for the number of quasi-hereditary structures. In the case of a hereditary incidence algebra, we present a necessary and sufficient condition for its poset of quasi-hereditary structures to be a lattice.

\end{abstract}
\maketitle
{\hypersetup{linkcolor=black}
\tableofcontents
}
\section{Introduction}
\noindent Quasi-hereditary algebras were introduced by Scott in \cite{scott} as an algebraic axiomatization of the theory of rational representations of semisimple algebraic groups. They were generalized to the concept of highest weight categories soon after in \cite{cline_parshall_scott}  allowing infinitely many simple objects. They form a huge class of algebras and categories  which appear in many areas of modern representation theory, including complex semisimple Lie algebras, Schur algebras, algebras of global dimension at most two and many more.

A quasi-hereditary algebra is a pair $(A,(I,\lhd))$ where $A$ is an Artin algebra and $(I,\lhd)$ is a partially ordered set indexing the isomorphism classes of simple $A$-modules which satisfies additional properties involving the filtrations of the projective modules by the so-called standard modules (see Definition \ref{dfn-qh-alg}). In particular, an algebra is \emph{not intrinsically} a quasi-hereditary algebra since the standard modules heavily depend on the partial order $(I,\lhd)$.

In the early examples of highest weight categories and quasi-hereditary algebras, the partial orders were easy to choose because they were related to the classical combinatorics of weights and roots and we already knew that these partial orderings were related to the representation theory of our examples. However, there are instances of quasi-hereditary algebras where there is no natural choice for the partial ordering and even if there is such a natural choice, one may wonder about the other possible orderings.

To our knowledge, there are two known results in this direction. The first, which is due to Dlab and Ringel, asserts that an Artin algebra is hereditary if and only if it is quasi-hereditary for any total ordering on $I$ (see \cite[Theorem~1]{Dlab-Ringel-89} for more details). The second, due to Coulembier, says that if an algebra has a simple preserving duality, then it has at most one `quasi-hereditary structure' (see \cite[Theorem~2.1.1]{C19} for more details, the notion of quasi-hereditary structure is explained below).

The main objective of this article is to continue the systematic study of all the possible choices of partially ordered sets that yield quasi-hereditary structures on a given Artin algebra. We start by known and easy remarks that will allow us to better define our objectives, and then give the foundations for our investigations. The organisation of this article can be summarised chronologically as follows.

The first main issue when one tries to classify all the partial orders $\lhd$ on $I$ such that $(A,(I,\lhd))$ is a quasi-hereditary algebra is that there are too many of them. However, as it has been proved by Conde (see \cite[Proposition ~1.4.12]{C16}), that if $(A,(I,\lhd))$ is a quasi-hereditary algebra, then the partial order $\lhd$ is an \emph{adapted poset to} $A$ in the sense of Dlab and Ringel (see Definition \ref{dfn_adapted}). The definition of adapted order is rather technical, but it reduces significantly the number of partial orders that we have to consider. 

The next step is to realize that we do not want to classify adapted partial orders, but we want to classify equivalence classes for an appropriate equivalence relation. Indeed, as it can be seen in Definition \ref{dfn-qh-alg}, the only place where the partial order $(I,\lhd)$ is involved is in the construction of the standard modules. As a consequence, if $(A,(I,\lhd_1))$ and $(A,(I,\lhd_2))$ are two quasi-hereditary algebras with same set of standard modules, it is then natural to say that the partial orders $\lhd_1$ and $\lhd_2$ are \emph{equivalent}, and we write $\lhd_1 \sim \lhd_2$. We call the equivalence classes of this relation the \emph{quasi-hereditary structures} on the algebra $A$, and these are exactly what we want to classify. We denote by $\qhstr(A)$ the set of quasi-hereditary structures on the algebra $A$. 

This equivalence relation was introduced by Dlab and Ringel in \cite{dlab-ringel} where they proved that any quasi-hereditary structure contains \emph{at least one total ordering}. This key fact is the reason why in most of the references the authors usually assume the ordering on $I$ to be total. However, for our purpose, we should not restrict ourselves to total orderings since one equivalence class may contain \emph{many different} total orderings. For example, when $A$ is a semisimple algebra with $n$ isomorphism classes of simple modules, it has a unique quasi-hereditary structure which contains $n!$ total orderings.

Let $(A,(I,\lhd))$ be a quasi-hereditary algebra with set of standard and costandard modules denoted by $\Delta$ and $\nabla$. We define subsets $\Dec(\lhd)$ and $\Inc(\lhd)$ of $I^2$ as follows:
\begin{align*}
\Dec(\lhd)\coloneqq\{ (i,j)\in I^2 \mid [\Delta(j):S(i)]\neq 0\}, \qquad 
\Inc(\lhd)\coloneqq \{(i,j)\in I^2 \mid [\nabla(j):S(i)]\neq 0\}.
\end{align*}
Then, we show in Lemma \ref{lem_Dec_Inc} that the transitive cover $\lhd_m$ of the relation $(\Dec(\lhd)\cup \Inc(\lhd))$ is equivalent to $\lhd$ and we prove in Proposition \ref{pro-mini-adapted} that it is the unique minimal element (with respect to the inclusion of relations) in the quasi-hereditary structure containing $\lhd$. These partial orders were already considered by Coulembier under the name of \emph{essential orderings} (see \cite[Definition~1.2.5]{C19}). Here, we call them \emph{minimal adapted orders}.  The relations $\Dec(\lhd)$ and $\Inc(\lhd)$ are called the \emph{decreasing} and \emph{increasing} relations of $\lhd$. These names have been chosen to highlight the relationship of our work and the work of Ch\^{a}tel, Pilaud and Pons on the weak order of integer posets (see \cite{CPP19}). It turns out that the combinatorics of the quasi-hereditary structures on an equioriented quiver of type $\A$ is the same as their combinatorics of the ``Tamari order element poset''(see Section \ref{section-An}). Moreover, in the case of a quiver algebra, these names are particularly relevant since an increasing relation follows the paths of the quiver, and a decreasing one goes against them.

In addition to standard and costandard modules, a quasi-hereditary algebra has a distinguished tilting module called the \emph{characteristic tilting module} (see Proposition \ref{pro-ch-tilt}). This leads to the following easy proposition.   

\begin{pro}[Lemma \ref{lem_two_qh_str_equiv}]
Let $\lhd_1$ and $\lhd_2$ be two partial orders on $I$ such that $(A, (I,\lhd_1))$ and $(A, (I,\lhd_2))$ are quasi-hereditary algebras.
Then the following statements are equivalent.
\begin{enumerate}
\item $\lhd_1 \sim \lhd_2$.
\item $(\lhd_1)_m = (\lhd_2)_m$. 
\item $T_1 \cong T_2$, where $T_i$ is the characteristic tilting module of $(A, (I,\lhd_i))$ for $i=1,2$. 
\end{enumerate}
\end{pro}

It is now clear that we should see the set of quasi-hereditary structures on $A$ as a subposet of its poset of tilting modules in the sense of Happel and Unger (see \cite{HU,RS} for more details). 

The poset of tilting modules is most of the time infinite, but the poset of quasi-hereditary structures is always finite since any quasi-hereditary structure contains at least one total order. In Example \ref{ex-complete-quiver} we show that the weak order on the set of permutations on $n$ letters can be realized as the poset of quasi-hereditary structures on the path algebra of a suitable orientation of a complete graph. The Tamari lattices are obtained by considering path algebras of an equioriented orientation of a type $\A$ Dynkin diagram. More precisely, we prove the following result in Section \ref{section-An}.

\begin{thm}[Theorem \ref{thm_Baptiste_bijection}]
Let $n \in \mathbb{N}$, and $A_n$ be an equioriented quiver of type $\A$  with $n$ vertices, and $\Lambda_n = \kk A_n$. Then, there are explicit bijections between
\begin{enumerate}
\item Minimal adapted posets to $\Lambda_n$,
\item Binary trees with $n$ vertices,
\item Isomorphism classes of tilting modules over $\Lambda_n$.
\end{enumerate}
Moreover, the poset of quasi-hereditary structures on $\Lambda_n$ is isomorphic to the Tamari lattice. 
\end{thm}
This implies that the number of quasi-hereditary structures is given by the $n$-th Catalan number $c_n =\frac{1}{n+1} { \binom{2n}{n} }$. Furthermore, any tilting module for this algebra is a characteristic tilting module. 

The understanding of the other orientations of a type $\A$ quiver is obtained as a consequence of a more general argument involving sources and sinks of a quiver.

A \emph{(iterated) deconcatenation} of a quiver $Q$ is a disjoint union $Q^1\sqcup \dots \sqcup Q^{\ell}$ of full subquivers $Q^i$ of $Q$ satisfying certain properties, see Subsection \ref{subsection-qh-sink-source} for details. The quivers $Q^i$ are obtained by deconcatenating $Q$ at sources or sinks. Then in Section \ref{section-dec} we prove the following result.

\begin{thm}[Theorem \ref{thm_qhstr_divide}]\label{source_sink_princ}
Let $Q^1\sqcup Q^2 \sqcup \dots \sqcup Q^{\ell}$ be an iterated deconcatenation of $Q$.
Let $A$ be a factor algebra of $\kk Q$ modulo some admissible ideal and $A^i\coloneqq  A / \langle e_u \mid u \in Q_0\setminus Q_0^{i} \rangle$.
Then we have an isomorphism of posets 
\begin{align*}
\qhstr(A) \longrightarrow \prod_{i=1}^{\ell}\qhstr(A^i).
\end{align*}
\end{thm}

Under the isomorphism, we have an explicit construction of the characteristic tilting $A$-modules in terms of the characteristic tilting $A^i$-modules, see Theorem \ref{thm-dec-ch-tilt}.

A quiver $Q$ whose underlying graph is of type $\A$ has an iterated deconcatenation $Q^1\sqcup Q^2 \sqcup \dots \sqcup Q^{\ell}$ such that each $Q^i$ is an equioriented quiver of type $\A$. Therefore as a corollary we obtain a classification of the quasi-hereditary structures on the path algebras of type $\A$.

\begin{thm}[Theorems \ref{thm-A-dec}, \ref{thm-A-ch-tilting}]
Let $Q$ be a quiver whose underlying graph is of type $\A$. Let $Q^1\sqcup Q^2 \sqcup \dots \sqcup Q^{\ell}$ be an iterated deconcatenation of $ Q $ such that each $ Q^i $ is an equioriented quiver of type $ \A_{n_{i}} $ for some $ n_{i}\in \Z_{\geq 1} $. Then we have an isomorphism of posets
\[
\qhstr(\kk Q) \longrightarrow  \prod_{i=1}^{\ell}\qhstr(\Lambda_{n_{i}}).  \]
Moreover, if $\lhd$ is a minimal adapted order to $\kk Q$ and $T=\bigoplus_{i\in I}T(i)$ is the characteristic tilting module associated to $\lhd$, then for vertices $i,j\in Q_0$, a simple $\kk Q$-module  associated to $j$ is a composition factor of $T(i)$ if and only if $j\lhd i$ holds.
\end{thm}

In \cref{section-DE-count} we count the number of quasi-hereditary structures on the path algebras of Dynkin types $\D$ and $\E$. For the quivers of type $\D$, by Theorem \ref{source_sink_princ} and the duality argument of Lemma \ref{lem-dual}, it is enough to consider only two orientations of $\D_n$, with $n\geq 4$, that is, the quivers $D_1$ and $D_2$ defined below.

\begin{thm}[Lemmas ~\ref{lem-counting-Q(n-3,1,1)}, \ref{lem-counting-Q(1,n-3,1)}]
Let $n\geq 3$. Then, 
\begin{enumerate}
    \item $|\qhstr(\kk D_1)| = 2c_{n} - 3c_{n-1}$,
    \item $|\qhstr(\kk D_2)| = 3c_{n-1} - c_{n-2}$,
\end{enumerate}
where $c_n$ is the $n$-th Catalan number and
\begin{align*}
\centering 
\begin{tikzpicture}[baseline=-.625ex, scale=0.8]
\node at (0.15,0) {{$ D_1 =$}};
\node(al)at(1,0){$1$};
\node(al-1)at(2,0){$2$};
\node(acdots)at(3.15,0)[inner sep=1pt]{$\,\cdots$};
\node(0)at(4.75,0){$n-2$};
\node(b1)at(6.75,0.5){$n-1\phantom{n}$};
\node(c1)at(6.75,-0.5){$n\phantom{n-1}$};
\draw[thick, ->] (al)--(al-1);
\draw[thick, ->] (al-1)--(acdots);
\draw[thick, ->] (acdots)--(0);
\draw[thick, ->] (0)--(b1.west);
\draw[thick, ->] (0)--(c1.west);
\end{tikzpicture}\hspace{-0.2cm},\quad
\begin{tikzpicture}[baseline=-.625ex,scale=0.8]
\node at (0.15,0) {$ D_2 =$};
\node(al)at(1,0){$1$};
\node(al-1)at(2,0){$2$};
\node(acdots)at(3.15,0)[inner sep=1pt]{$\,\cdots$};
\node(0)at(4.75,0){$n-2$};
\node(b1)at(6.75,0.5){$n-1\phantom{n}$};
\node(c1)at(6.75,-0.5){$n\phantom{n-1}$};
\draw[thick, ->] (al)--(al-1);
\draw[thick, ->] (al-1)--(acdots);
\draw[thick, ->] (acdots)--(0);
\draw[thick, ->] (0)--(b1.west);
\draw[thick, ->] (c1.west)--(0);
\end{tikzpicture}\hspace{-0.2cm}.
\end{align*}
\end{thm}
The type $\E$ case is explained in Example \ref{typeE}.

So far, in all our examples, we have found a lattice of quasi-hereditary structures: in the complete graph case, we obtained the weak order and for quivers of type $\A$ we obtained products of Tamari lattices. As it can be seen in Proposition \ref{pro-an-zz}, this is no longer true in the case of affine quivers of type $\A$. The question of understanding which Artin algebras have a lattice of quasi-hereditary structures is both natural and very intriguing. We have the feeling that this will only happen in a few cases but our progress is rather modest: in Section \ref{section-lattice-qhstr} we prove the following result. 

\begin{thm}[Theorem \ref{tree-qh-lattice}]\label{intro-tree-qh-lattice}
Let $Q$ be a finite acyclic quiver whose underlying graph is a tree.
Then $\qhstr(\kk Q)$  is a lattice if and only if $Q$ does not have the following quiver as a subquiver for any $n\geq 4$:
\begin{align*}
\begin{tikzpicture}[baseline=-.625ex,xscale=.85]
\node(01+)at(0,0.5){$1$};
\node(01-)at(0,-0.5){$2$};
\node(al)at(1,0){$3$};
\node(al-1)at(2,0){$4$};
\node(acdots)at(3.07,0)[inner sep=.2pt]{$\,\cdots$};
\node(0)at(4.5,0){$n-1$};
\node(b1)at(6.5,0.5){$n\phantom{n+1}$};
\node(c1)at(6.5,-0.5){$n+1\phantom{n}$};
\draw[thick, ->] (01+.east)--(al);
\draw[thick, ->] (01-.east)--(al);
\draw[thick, ->] (al)--(al-1);
\draw[thick, ->] (al-1)--(acdots);
\draw[thick, ->] (acdots)--(0);
\draw[thick, ->] (0)--(b1.west);
\draw[thick, ->] (0)--(c1.west);
\end{tikzpicture}\hspace{-0.2cm}.
\end{align*}
\end{thm}

When it is a lattice, we give an explicit description of the partial orders which represent a meet or a join of given two quasi-hereditary structures, see Subsection \ref{subsection-lattice-tree}. 

As an immediate corollary of this theorem, we have the following one.
\begin{cor}
For a Dynkin quiver $Q$, $\qhstr(\kk Q)$ is a lattice.
\end{cor}

Curiously, Theorem \ref{intro-tree-qh-lattice} does not seem to have a nice generalization to the setting of finite acyclic quivers. However, it seems to generalize to the setting of incidence algebras of a finite poset. We denote by $Z_n$ the partially ordered set whose Hasse diagram is a zig-zag orientation of an affine quiver of type $\A_n$ (see diagram \ref{zigzag}) and we propose the following conjecture.  

\begin{conj}\label{intro-conj-qh-lattice-poset}
Let $(P,\leq)$ be a finite poset. Then, the poset of quasi-hereditary structures on the incidence algebra of $(P,\leq)$ is a lattice if and only if $Z_n$ is not isomorphic to a full subposet of $(P,\leq)$ for any $n\geq 4$.
\end{conj}

Actually, the setting of incidence algebras is not only a good setting for a generalization of our result, it also simplifies our arguments.
So we prove Theorem \ref{intro-tree-qh-lattice} as a corollary of the following theorem. 

\begin{thm}[Theorem \ref{thm-qh-lattice-poset}]\label{intro-thm-qh-lattice-poset}
Let $(P,\leq)$ be a finite poset. We assume that the incidence algebra $A(P)$ of $(P,\leq)$ is hereditary. Then, the poset of quasi-hereditary structures on $A(P)$ is a lattice if and only if $Z_n$ is not isomorphic to a full subposet of $(P,\leq)$ for any $n\geq 4$.
\end{thm}
The relationship between Theorem \ref{intro-tree-qh-lattice} and Theorem \ref{intro-thm-qh-lattice-poset} is explained in Remark \ref{comp-thms}.

Experiments were carried out using the GAP-package QPA \cite{qpa} and SageMath \cite{sage}. The code is available online \cite{code}.

\begin{notation-convention}
Throughout this paper, 
$\kk$ denotes a field of arbitrary characteristic.
A subcategory always means a full subcategory which is closed under isomorphisms.
An Artin algebra is an $R$-algebra $A$ such that $R$ is a commutative Artinian ring and $A$ is finitely generated as an $R$-module.
For an Artin algebra, we usually deal with finitely generated right modules.
We denote by $\modu A$ the category of finitely generated right $A$-modules.
For an element or subset $X$ of $A$, we denote by $\ideal{X}$ the two-sided ideal of $A$ generated by $X$.
For a quiver $Q$, let $\kk Q$ be the path algebra of $Q$.
For two arrows $\alpha, \beta$ of $Q$, if the terminating vertex of $\alpha$ equals the starting vertex of $\beta$, then $\alpha\beta$ indicates the composite of $\alpha$ with $\beta$. An order means a partially ordered set.
\end{notation-convention}

\section{Preliminaries}
In this section, we give some definitions and preliminary results which we use throughout this paper.

\subsection{Adapted posets}
We first consider the so-called \emph{adapted partial orders} introduced by Dlab and Ringel in \cite{dlab-ringel} on the set of isomorphism classes of simple modules over an Artin algebra $A$.

Let $(I,\lhd)$ be a finite partially ordered set with $\{S(i)\}_{i\in I}$ being a complete set of representatives of isomorphism classes of simple $A$-modules. For $i\in I$, we denote by $P(i)$ a projective cover of $S(i)$ and by $I(i)$ an injective envelope of $S(i)$.
If $M$ is an $A$-module, we denote by $[M:S(i)]$ the Jordan-H\"{o}lder multiplicity of $S(i)$ in $M$.

Let $\Theta$ be a class of $A$-modules which is closed under isomorphisms. We denote by $\mathcal{F}(\Theta)$ the subcategory of $\modu A$ consisting of all $A$-modules which have a $\Theta$-filtration, that is, an $A$-module $M$ for which there exists a chain of submodules
\[
0=M_n \subset M_{n-1}\subset \cdots \subset M_1 \subset M_0 = M
\]
such that $M_{i}/M_{i+1}$ belongs to $\Theta$.
Let $\Delta(i)$ be the largest quotient of $P(i)$ whose composition factors $S(j)$ are such that $j\lhd i$. Similarly let $\nabla(i)$ be the largest submodule of $I(i)$ whose composition factors $S(j)$ are such that $j\lhd i$. Set $\Delta=\{\Delta(i)\}_{i\in I}$ and $\nabla=\{\nabla(i)\}_{i\in I}$. The module $\Delta(i)$ (resp.\ $\nabla(i)$) is called the \emph{standard} (resp.\  \emph{costandard}) module with \emph{weight} $i$.

If $M\in \mathcal{F}(\Delta)$, the number of times that $\Delta(i)$ appears in a $\Delta$-filtration of $M$ does not depend on the choice of the filtration. We denote it by $\big(M:\Delta(i) \big)$. Similarly, the number of times that $\nabla(i)$ appears in $N\in \mathcal{F}(\nabla)$ is independent of the choice of the filtration and we denote it by $\big(N:\nabla(i)\big)$.

\begin{dfn}[Dlab-Ringel]\label{dfn_adapted}
Let $A$ be an Artin algebra with set of isomorphism classes of simple modules $\{S(i)\}_{i\in I}$. A partial order $\lhd$ on $I$ is \emph{adapted} to $A$ if for every $A$-module $M$ with top $S(i)$ and socle $S(j)$, where $i$ and $j$ are incomparable, there is $k \in I$ with $i \lhd k$ \underline{and} $j \lhd k$ and $[M:S(k)]\neq 0$. 

\end{dfn}
\begin{lem}[Dlab-Ringel]\label{weak-adapted}
Let $A$ be an Artin algebra with set of isomorphism classes of simple modules $\{S(i)\}_{i\in I}$. A partial order $\lhd$ on $I$ is adapted to $A$ if and only if for every $A$-module $M$ with top $S(i)$ and socle $S(j)$, where $i$ and $j$ are incomparable, there is $k \in I$ with $i \lhd k$ \underline{or} $j \lhd k$ and $[M:S(k)]\neq 0$. 
\end{lem}
\begin{proof}
See \cite{dlab-ringel} at the bottom of page $3$.
\end{proof}
As a nice property of adapted posets, we have the following result.
\begin{lem}[Dlab-Ringel]\label{extension_adapted}
Let $(I,\lhd_1)$ be an adapted poset to $A$. Let $\lhd_2$ be a refinement of $\lhd_1$. For $l=1,2$ let $\Delta_l(i)$ be the standard module with weight $i$ for the poset $\lhd_l$. Then 
\[
\Delta_1(i) = \Delta_2(i) \ \forall i\in I.
\]
\end{lem}
\begin{proof}
This is stated without proof at the top of page 4 of \cite{dlab-ringel}. We sketch the proof for the convenience of the reader. It is clear that there is a surjective map $\psi_j \colon \Delta_2(j)\to \Delta_1(j)$. If this is not an isomorphism, we consider $i$ such that $S(i)$ is a composition factor of $\Delta_2(j)$ at the top of the kernel of $\psi_j$. Then, there is a module $M$ which is a non-split extension of $\Delta_1(j)$ and $S(i)$. Since $\lhd_1$ is adapted to $A$ one of the following holds.
\begin{enumerate}
    \item $i$ and $j$ are comparable in $(I, \lhd_1)$.
    \item They are not comparable and there is a composition factor $S(k)$ of $\Delta_1(j)$ such that $j\lhd_1 k$ and $i\lhd_1 k$.
\end{enumerate}
Since the composition factors in $\Delta_1(j)$ are smaller than $j$ with respect to $\lhd_1$, the second condition cannot occur. For the first condition, since $i\vartriangleleft_2 j$ and $\lhd_2$ is a refinement of $\lhd_1$ the only possibility is that $i\lhd_1 j$ and this contradicts the fact that $S(i)$ is not a composition factor of $\Delta_1(j)$.
\end{proof} 
By duality we have a similar result for costandard modules.

\begin{dfn}\label{dfn-qh-algebra}
Let $A$ be an Artin algebra with set of isomorphism classes of simple modules $\{S(i)\}_{i\in I}$. Let $\lhd_1$ and $\lhd_2$ be two partial orders on $I$. Then $\lhd_1$ is \emph{equivalent} to $\lhd_2$, denoted by $\lhd_1 \sim \lhd_2$, if $\Delta_1 = \Delta_2$ and $\nabla_1 = \nabla_2$.
\end{dfn}

This equivalence relation is compatible with the notion of adapted poset in the following sense. 

\begin{lem}\label{equiv_adapted}
Let $A$ be an Artin algebra and $I$ be a set indexing the isomorphism classes of simple $A$-modules. Let $\lhd_1$ be an adapted poset to $A$. If $\lhd$ is equivalent to $\lhd_1$, then $\lhd$ is adapted to $A$. 
\end{lem}

\begin{proof}
Let $M$ be an indecomposable module with simple top $S(i)$ and simple socle $S(j)$ with $i$ and $j$ incomparable with respect to $\lhd$.
Since $M$ has simple top $S(i)$ it is a quotient of $P(i)$.
We denote by $U(i)$ the kernel of the projection from $P(i)$ to $\Delta(i)=\Delta_1(i)$.
Since $i$ and $j$ are incomparable with respect to $\lhd$, the module $S(j)$ is a composition factor of $U(i)$.
Then, there is a composition factor $S(k)$ which is at the top of $U(i)$ and which is also a composition factor of $M$.
We denote by $N$ a non-split extension of $S(k)$ and $\Delta_1(i)$.
There is factor module $N'$ of $N$ with simple top $S(i)$ and simple socle $S(k)$.
In fact, the socle of $N$ is of the form $S(k)\oplus S'$ for some semisimple module $S'$.
Then a socle of a factor module $N/S'$ has $S(k)$ as a direct summand.
By the induction on the length of $N$, we have the desired factor module $N'$ of $N$.
Note that any composition factor of $N'/S(k)$ is a composition factor of $\Delta_1(i)$.

We show that $i \lhd_1 k$ holds.
If $i$ and $k$ are incomparable with respect to $\lhd_1$, then there exists $\ell\in I$ such that $i \lhd_1 \ell$ and $k \lhd_1 \ell$ with $[N' : S(\ell)]\neq 0$ since $\lhd_1$ is an adapted poset.
Then, as mentioned above, $S(\ell)$ is a composition factor of $\Delta_1(i)$ or $\ell=k$.
Both of them induce contradictions.
If $k \lhd_1 i$ holds, then since $S(k)$ is at the top of $U(i)$, $\Delta_1(i)$ should have a composition factor $S(k)$, which is a contradiction.
Therefore $i \lhd_1 k$ holds.

By transitivity, if $S(\ell)$ is a composition factor of $\Delta_1(i)$, then $\ell\lhd_1 k$ holds.
In particular, any composition factor $S(\ell)$ of $N'$ satisfies $\ell \lhd_1 k$.
This implies that, by taking an injective envelop $N' \to I(k)$, $N'$ is contained in $\nabla(k)=\nabla_1(k)$.
In particular, $i \lhd k$ holds.
Thus we complete the proof by Lemma \ref{weak-adapted}.
\end{proof}

Partial orders on a finite set $I$ can be ordered by inclusion of their sets of relations. This gives a `poset of posets' over $I$ where the minimal element is the equality relation on $I$ and the maximal elements are the total orders. The poset obtained by taking the intersection of the relations of two given posets $\lhd_1$ and $\lhd_2$ is called the \emph{intersection} of $\lhd_1$ and $\lhd_2$.

\begin{lem}\label{intersection_adapted}
Let $A$ be an Artin algebra with $I$ a set indexing a complete set of isomorphism classes of simple $A$-modules. Let $\lhd_1$ and $\lhd_2$ be two adapted orders in the same equivalence class.
Then the following statements hold.
\begin{enumerate}
    \item The intersection of $\lhd_1$ and $\lhd_2$ is an adapted order in the same equivalence class. 
    \item In each equivalence class of adapted posets to $A$ there is a unique minimal poset. 
\end{enumerate}
\end{lem}
\begin{proof}
It is clear that (1) implies (2). Let $\lhd_1$ and $\lhd_2$ be two posets in the same equivalence class. We denote by $\Delta=\Delta_1=\Delta_2$ the corresponding set of standard modules. We let $\lhd_{int}$ be the intersection of $\lhd_1$ and $\lhd_2$ and we denote by $\Delta_{int}$ the corresponding set of standard modules. 

Let $i\in I$. By definition $\Delta_{int}(i)$ is the largest quotient of $P(i)$ whose composition factors are $S(j)$ such that $j\lhd_1 i$ and $j\lhd_2 i$. So $\Delta_1(i)$ surjects onto $\Delta_{int}(i)$. If they are not isomorphic, at the top of the kernel there is a simple module $S(j)$ such that $j\lhd_1 i$ but $j$ is not smaller that $i$ for $\lhd_2$. This contradicts $\Delta_1(i)=\Delta_2(i)$. So, we have $\Delta_{int} = \Delta$ and by a dual argument, we see that $\nabla_{int}=\nabla$ and the poset $\lhd_{int}$ is equivalent to the posets $\lhd_1$ and $\lhd_2$. The result follows from Lemma \ref{equiv_adapted}
\end{proof}

\begin{rem}
Lemma \ref{intersection_adapted} is used in the codes of the first and last authors \cite{code}.
\end{rem}

For a partial order $\lhd$ on $I$, let $\Delta$ be the set of standard $A$-modules and $\nabla$ the set of costandard $A$-modules associated to $\lhd$.
We define subsets $\Dec(\lhd)$ and $\Inc(\lhd)$ of $I^2$ as follows:
\begin{align*}
\Dec(\lhd)\coloneqq \{ (i,j)\in I^2 \mid [\Delta(j):S(i)]\neq 0\}, \qquad 
\Inc(\lhd)\coloneqq \{(i,j)\in I^2 \mid [\nabla(j):S(i)]\neq 0\}.
\end{align*}
Clearly, $\Dec(\lhd)$ and $\Inc(\lhd)$ depend on only the equivalence class of $\lhd$.
For $i, j \in I$, we write $i \lhd^D j$ if $(i,j)\in\Dec(\lhd)$ and write $i\lhd^I j$ if $(i,j)\in\Inc(\lhd)$.
For $i\in I$, we have $i\lhd^D i$ and $i \lhd^I i$.

For a subset $R$ of $I^2$, we denote by $R^{\sf tc}$ the transitive closure of $R$.
Then the following lemma is easy to prove.
\begin{lem}\label{lem_Dec_Inc}
Let $\lhd_m=\left( \Dec(\lhd) \cup \Inc(\lhd) \right)^{\sf tc}$.
\begin{enumerate}
\item
If $i \lhd_m j$, then $i \lhd j$ holds.
\item
$\lhd_m$ is a partial order on $I$.
\end{enumerate}
\end{lem}
\begin{proof}
By definition, $i \lhd^D j$ implies $i\lhd j$, and $i \lhd^I j$ implies $i \lhd j$.
Thus the assertions hold.
\end{proof}
Then we have the following proposition.
\begin{pro}\label{pro-mini-adapted}
Let $A$ be an Artin algebra and $I$ be a set indexing the isomorphism classes of simple $A$-modules. For an adapted partial order $\lhd$ on $I$, let $\lhd_m=\left( \Dec(\lhd) \cup \Inc(\lhd) \right)^{\sf tc}$.
\begin{enumerate}
\item
The partial orders $\lhd$ and $\lhd_m$ are equivalent.
\item
$\lhd_m$ is the unique minimal partial order among partial orders $\lhdp$ on $I$ with $\lhd\sim\lhdp$.
\end{enumerate}
\end{pro}
\begin{proof}
(1)
We denote by $\Delta_m$ the set of standard $A$-modules associated to $\lhd_m$.
Fix $i\in I$.
Since all composition factors $S(k)$ of $\Delta_m(i)$ satisfy $k\lhd_m i$, $k\lhd i$ holds by Lemma \ref{lem_Dec_Inc}.
We see that $\Delta_m(i)$ is a quotient of $\Delta(i)$. Conversely if $S(k)$ is a composition factor of $\Delta(i)$, then by construction we have $k\lhd_m i$. Since this holds for all the composition factors of $\Delta(i)$, we see that $\Delta(i)$ is a quotient of $\Delta_m(i)$. 

By a dual argument, we see that the set of costandard modules for $\lhd$ and $\lhd_m$ are equal. 

(2)
Since $\Dec(\lhd)=\Dec(\lhdp)$ and $\Inc(\lhd)=\Inc(\lhdp)$ hold, the assertion follows from Lemma \ref{lem_Dec_Inc}.
\end{proof}
\begin{dfn}
Let $A$ be an Artin algebra. A partial order of the form $\lhd_m$ for an adapted order  $\lhd$ to $A$ is called a \emph{minimal adapted order}. 
\end{dfn}
\subsection{Quasi-hereditary algebras}
Let $A$ be an Artin algebra, and $(I,\lhd)$ be a finite partially ordered set indexing the set $\{S(i)\}_{i\in I}$ of isomorphism classes of simple $A$-modules.

\begin{dfn}[Cline-Parshall-Scott \cite{cline_parshall_scott}]\label{dfn-qh-alg}
The pair $(A,(I,\lhd))$ is a \emph{quasi-hereditary algebra} if:
\begin{enumerate}
    \item $[\Delta(i):S(i)] = 1$.
    \item $P(i)\in \mathcal{F}(\Delta)$.
    \item $\big(P(i) : \Delta(i)\big) = 1$, and $\big(P(i): \Delta(j)\big)\neq 0$ implies $i\lhd j$.
\end{enumerate}
\end{dfn}
In \cite{dlab-ringel}, Dlab and Ringel restrict themselves to adapted posets in order to work with total orders (since an adapted poset is always equivalent to a total order). Using the next result  by Conde, we see that this is not a restriction.

\begin{lem}[{\cite{C16}}]
If $(A,(I,\lhd))$ is a quasi-hereditary algebra, then $(I,\lhd)$ is adapted to $A$.
\end{lem}
\begin{proof}
The proof is a part of \cite[Proposition 1.4.12 ]{C16}. We sketch it for the convenience of the reader. Let $M$ be a module with  simple top $S(i)$ and  simple socle $S(j)$.
Since $M$ is a quotient of $P(i)$, $S(j)$ is a composition factor of $P(i)$.
Because the algebra is quasi-hereditary $P(i)$ has a $\Delta$-filtration.
So $S(j)$ must be a composition factor of a standard module $\Delta(k)$ with $i \lhd k$ by Definition \ref{dfn-qh-alg} (3).
If $k=i$ then $j\lhd i$. If $k\neq i$ and $S(j)$ is at the top of $\Delta(k)$ we have $i\lhd j$.
Finally, if $S(j)$ is not at the top of $\Delta(k)$ we have $i\lhd k$ and $j\lhd k$. 
\end{proof}

We can know considerably simplify the definition of quasi-hereditary algebras.

\begin{pro}\label{simple-def-qh}
Let $A$ be an Artin algebra with a poset $(I,\lhd)$ indexing the isomorphism classes of simple $A$-modules. Then $(A,(I,\lhd))$ is a quasi-hereditary algebra if and only if
\begin{enumerate}
\item $(I,\lhd)$ is adapted to $A$.
\item For every $i\in I$, $[\Delta(i):S(i)]=1$.
\item $P(i)\in \mathcal{F}(\Delta)$ for $i\in I$.
\end{enumerate}
\end{pro}
\begin{proof}
See \cite[Theorem~1]{dlab-ringel}.
\end{proof}

If an Artin algebra $A$ is hereditary, then any adapted order automatically defines a quasi-hereditary algebra.

\begin{pro}\label{pro-hered-adapted}
If $A$ is hereditary, then for any adapted order $(I,\lhd)$ to $A$, the pair $(A,(I,\lhd))$ is a quasi-hereditary algebra.
\end{pro}
\begin{proof}
By Lemma \ref{extension_adapted}, we may assume that $\lhd$ is a total order.
Then the assertion follows from \cite[Theorem 1]{Dlab-Ringel-89}.
\end{proof}

We have the following well known property of quasi-hereditary algebras.

\begin{lem}\label{lem-qh-maximal}
Let $(A,(I,\lhd))$ be a quasi-hereditary algebra, and assume that $i$ is a maximal element of $(I,\lhd)$.
We denote by $\lhd|_{I\setminus\{i\}}$ the restriction of $\lhd$ on $I\setminus\{i\}$.
Then $(A/\ideal{e_i}, \lhd|_{I\setminus\{i\}})$ is a quasi-hereditary algebra such that the standard $A/\ideal{e_i}$-module with weight $k\in I\setminus\{i\}$ corresponds to the standard $A$-module with weight $k$.
\end{lem}
\begin{proof}
Since $i$ is maximal, it is clear that the poset $(I\setminus \{i\},\lhd|_{I\setminus\{i\}})$ is an order ideal. The  subcategory of $\modu A$ consisting of all $A$-modules with composition factors $S(j)$ with $j\in I\setminus \{i\}$ coincides with $\modu (A/\ideal{e_i})$. Then the Lemma is a consequence of \cite[Theorem 3.5]{cline_parshall_scott}. 
\end{proof}

For a subcategory $\mcC$ of $\modu A$, we denote by $\mcC^{\perp}$ (${}^{\perp}\mcC$, respectively) the full subcategory of $\modu A$ consisting of modules $M$ with $\Ext_A^i(C,M)=0$ ($\Ext_A^i(M,C)=0$, respectively) for any $i>0$ and $C\in\mcC$.
For a module $M$, we denote by $\add M$ the full subcategory of $\modu A$ consisting of direct summands of the direct sum of finitely many copies of $M$.
For simplicity, write $(\add M)^{\perp}=M^{\perp}$ and so on.

We recall some properties of quasi-hereditary algebras which we use in this paper.
For a quasi-hereditary algebra $(A, (I,\lhd))$, we have the following equalities, see \cite[Theorems 4, $4^{\ast}$]{Ringel91}:
\begin{align}\label{eq-F-perp}
\mcF(\Delta) = {}^{\perp}\mcF(\nabla),
\qquad
\mcF(\nabla) = \mcF(\Delta)^{\perp}. 
\end{align}

An $A$-module $T$ is called a \emph{tilting module} if it satisfies the following three conditions:
\begin{itemize}
\item[(i)] the projective dimension of $T$ is finite,
\item[(ii)] $\Ext_A^i(T,T)=0$ for all $i>0$,
\item[(iii)] there exists an exact sequence $0\to A \to T_0 \to T_1 \to \dots \to T_{\ell} \to 0$, with $T_i\in\add T$.
\end{itemize}

Let $M$ be an $A$-module and $\mcC$ be a subcategory of $\modu A$.
A \emph{left $\mcC$-approximation of $M$} is a morphism $f \colon M \to C$ such that $C\in\mcC$ and the map $\Hom_A(f,C^\p) \colon \Hom_A(C,C^\p) \to \Hom_A(M,C^\p)$ is surjective for any $C^\p\in\mcC$.
A \emph{right $\mcC$-approximation of $M$} is defined dually.
 
\begin{pro}[\cite{Ringel91}]\label{pro-ch-tilt}
Let $(A,(I,\lhd))$ be a quasi-hereditary algebra.
For each $i\in I$, there exists an indecomposable $A$-module $T(i)$ and short exact sequences
\[
0\to \Delta(i) \xto{f} T(i) \to X(i) \to 0,
\qquad
0\to Y(i) \to T(i) \xto{g} \nabla(i) \to 0,
\]
where $X(i)$ belongs to $\mcF(\Delta(j) \mid j \lhd i, j\neq i)$ and $Y(i)$ belongs to $\mcF(\nabla(j) \mid j \lhd i, j\neq i )$ such that
\begin{enumerate}
\item
$f$ is a left $\mcF(\nabla)$-approximation of $\Delta(i)$.
\item
$g$ is a right $\mcF(\Delta)$-approximation of $\nabla(i)$.
\item
$T=\bigoplus_{i\in I}T(i)$ is a tilting $A$-module satisfying $\add T=\mcF(\Delta)\cap\mcF(\nabla)$.
\item
$\mcF(\Delta) = {}^{\perp}T$ and $\mcF(\nabla) = T^{\perp}$ hold.
\end{enumerate}
\end{pro}

A tilting module $T=\bigoplus_{i\in I}T(i)$ in Proposition \ref{pro-ch-tilt} is called a \emph{characteristic tilting module} for a quasi-hereditary algebra $(A,(I,\lhd))$.

If the standard modules and the costandard modules have small projective or injective dimension, then categories $\mcF(\Delta)$ and $\mcF(\nabla)$ have good properties.
We only refer a statement about $\mcF(\Delta)$, the statement for $\mcF(\nabla)$ is dual. 

\begin{lem}[{\cite[Appendix]{Ringel10}}]\label{lem_hered_F}
Let $(A,(I,\lhd))$ be a quasi-hereditary algebra.
Then $\mcF(\Delta)$ is closed under submodules if and only if the projective dimension of any standard module is at most one.
\end{lem}

\subsection{Poset of quasi-hereditary structures}

Let $I$ be a set indexing the set of isomorphism classes of simple $A$-modules. Let $\lhd_1$ and $\lhd_2$ be two partial orders on $I$. If $(A,(I,\lhd_1))$ is a quasi-hereditary algebra and $\lhd_2 \sim \lhd_1$ then $(A,(I,\lhd_2))$ is also a quasi-hereditary algebra since the definition of quasi-hereditary algebra only depends on the set of standard modules. In this case, we can give various characterizations of this equivalence relation.

\begin{lem}\label{lem_two_qh_str_equiv}
Let $\lhd_1$ and $\lhd_2$ be two partial orders on $I$ such that $(A, (I,\lhd_1))$ and $(A, (I,\lhd_2))$ are two quasi-hereditary algebras.
Then the following statements are equivalent.
\begin{enumerate}
\item $\lhd_1 \sim \lhd_2$.
\item $\Delta_1 = \Delta_2$.
\item $\nabla_1 = \nabla_2$.
\item $\mcF(\Delta_1)=\mcF(\Delta_2)$.
\item $\mcF(\nabla_1)=\mcF(\nabla_2)$.
\item $T_1 \cong  T_2$ where $T_i$ is the characteristic tilting module of $(A, (I,\lhd_i))$ for $i=1,2$. 
\end{enumerate}
\end{lem}
\begin{proof}
We show (4) implies (2).
For each $i\in I$, let $K(i)$ be the sum of the kernels of non-zero surjective maps $P(i) \to X$ with $X\in\mcF(\Delta_1)$.
Then $\Delta_1(i)=P(i)/K(i)$ holds by the proof of \cite[Corollary~4]{Ringel91}.
We have $\Delta_1(i) = P(i)/K(i) = \Delta_2(i)$ by the assumption.
Dually, (5) implies (3).

The other equivalences are induced from equation (\ref{eq-F-perp}) and Proposition \ref{pro-ch-tilt}.
\end{proof}

\begin{dfn}\label{dfn-minimal-adapted-order}
Let $A$ be an Artin algebra with an adapted poset $(I,\lhd)$ indexing the isomorphism classes of simple $A$-modules.
The equivalence class of $\lhd$ with respect to $\sim$ is a \emph{quasi-hereditary structure} on $A$ provided $(A,(I,\lhd))$ is a quasi-hereditary algebra.
\end{dfn}

We denote by $\qhstr(A)$ the set of all quasi-hereditary structures on $A$ and denote by $[\lhd]$ a quasi-hereditary structure represented by $\lhd$.

By Lemma \ref{lem_two_qh_str_equiv}, the equivalence class of $\lhd$ only depends on its characteristic tilting module. It is then natural to order the quasi-hereditary structures in the following way. Let $[\lhd_1]$ and $[\lhd_2]$ be two quasi-hereditary structures on $A$ with respective set of standard modules $\Delta_1$ and $\Delta_2$, then we set $[\lhd_1] \preceq [\lhd_2]$ if $\mathcal{F}(\Delta_2) \subseteq \mathcal{F}(\Delta_1)$.
By this ordering, we regard $(\qhstr(A), \preceq)$ as a poset.

As we see in the next lemma, this ordering is induced from a partial ordering on tilting modules \cite{HU, RS}.

\begin{lem}\label{lem-qhstr-poset}
Let $[\lhd_1]$ and $[\lhd_2]$ be two quasi-hereditary structures on $A$ and $T_i$ be the characteristic tilting module of $(A,\lhd_i)$ for $i=1,2$.
The following statements are equivalent.
\begin{enumerate}
    \item $[\lhd_1] \preceq [\lhd_2]$.
    \item $\mathcal{F}(\nabla_1) \subseteq \mathcal{F}(\nabla_2)$.
    \item $T_1^{\perp} \subseteq T_2^{\perp}$.
\end{enumerate}
\end{lem}
\begin{proof}
The assertions follow from $\mcF(\nabla)=\mcF(\Delta)^{\perp}$, $\mcF(\Delta)={}^{\perp}\mcF(\nabla)$, $\mcF(\nabla)=T^{\perp}$ and $\mcF(\Delta)={}^{\perp}T$ for a characteristic tilting module $T$.
\end{proof}

\begin{lem}\label{lem-surj-Dec}
Let $[\lhd_1]$ and $[\lhd_2]$ be two quasi-hereditary structures on $A$.
\begin{enumerate}
\item
The following are equivalent.
\begin{enumerate}
    \item There is a surjective morphism  $\Delta_{2}(i)\to \Delta_1(i)$ for all $i \in I$.
    \item $\Dec(\lhd_1)\subseteq\Dec(\lhd_2)$.
\end{enumerate}
\item
The following are equivalent.
\begin{enumerate}
    \item[(c)] There is an injective morphism  $\nabla_2(i)\to \nabla_1(i)$ for all $i \in I$.
    \item[(d)] $\Inc(\lhd_2)\subseteq\Inc(\lhd_1)$.
\end{enumerate}
\item
If $[\lhd_1] \preceq [\lhd_2]$ holds, then the statements in (1) and (2) hold.
Conversely, if $A$ is a hereditary algebra, then statements in (1) or (2) imply $[\lhd_1] \preceq [\lhd_2]$.
\end{enumerate}
\end{lem}
\begin{proof}
(1), (2)
By the definitions of standard and costandard modules, (a) is equivalent to (b), and (c) is equivalent to (d).

(3)
It is clear that $\mcF(\Delta_2)\subseteq\mcF(\Delta_1)$ implies (a).
Assume that $A$ is hereditary and (a) holds.
We prove by induction on $(I,\lhd_2)$ that all standard modules in $\Delta_2$ are $\Delta_1$-filtered.
If $i$ is minimal in $(I,\lhd_2)$, then $\Delta_2(i)$ is isomorphic to $S(i)$.
Since there is a surjection from $\Delta_2(i)$ to $\Delta_1(i)$ and $\Delta_1(i)$ is non-zero module, the surjection is an isomorphism.
For $i\in I$, assume that $\Delta_2(j)$ is $\Delta_1$-filtered for any $j$ which is strictly smaller than $i$ with respect to $\lhd_2$.
By assumption there is an exact sequence $0 \to X \to \Delta_2(i)\to \Delta_1(i)\to 0$ for some module $X$.
Since $A$ is hereditary, the category $\mcF(\Delta_2)$ is closed under submodules by Lemma \ref{lem_hered_F}.
Thus $X$ is $\{\Delta_2(j)\}$-filtered, where $j$ is strictly smaller than $i$ with respect to $\lhd_2$.
By induction hypothesis, $X$ is $\Delta_1$-filtered and so is $\Delta_2(i)$.
\end{proof}

Note that if $I$ is a set indexing the simple $A$-modules, then $I$ also indexes simple $A^{\rm op}$-modules.
Since $A$ is an Artin $R$-algebra, there exists a duality between $\modu A$ and $\modu A^{\op}$ induced from $D=\Hom_R(-,E)$, where $E$ is an injective envelop of the direct sum of representatives of all simple $R$-modules up to isomorphisms.
For a partial order $\lhd$ on $I$, we denote by $\Delta^{\op}$ (by $\nabla^{\op}$, respectively) the set of standard (costandard, respectively) $A^{\op}$-modules associated to the partial order $\lhd$.
This induces the following lemma, which is explained in Section 1 of \cite{Dlab-Ringel-89} without proof.
\begin{lem}\label{lem-dual}
Let $\lhd$ be a partial order on $I$.
Then the following statements hold.
\begin{enumerate}
\item
We have $\Delta^{\rm op}(i) \cong  D\nabla(i)$ and $\nabla^{\rm op}(i) \cong  D\Delta(i)$ for any $i\in I$.
In particular, $(A,(I,\lhd))$ is a quasi-hereditary algebra if and only if $(A^{\op},(I,\lhd))$ is a quasi-hereditary algebra.
\item
An assignment $\lhd \mapsto \lhd$ induces an anti-isomorphism between $\qhstr(A)$ and $\qhstr(A^{\rm op})$.
\end{enumerate}
\end{lem}
\begin{proof}
A dual of a simple $A$-module indexed by $i$ is a simple $A^{\op}$-module indexed by $i$.
Thus the isomorphisms in (1) are obtained from the definition of (co)standard modules.
By (1), the assignment induces a bijection from $\qhstr(A)$ to $\qhstr(A^{\rm op})$.
By the isomorphisms in (1), we have $D\mcF(\Delta)=\mcF(\nabla^{\op})$.
This implies that the map reverses the ordering on quasi-hereditary structures.
\end{proof}
For a path algebra over a field, the following proposition says that we may assume that the quiver has no multiple arrows if we want to study its poset of quasi-hereditary structures.

Let $Q$ be a finite acyclic quiver.
We say that $Q$ \emph{has multiple arrows} if there exist vertices $i, j$ of $Q_0$ such that the number of arrows between $i$ and $j$ is greater than one.
In this case, the arrows between $i$ and $j$ are said to be \emph{multiple arrows}.
If $Q$ has multiple arrows, we denote by $Q^{\rm mf}$ a finite acyclic quiver such that $Q^{\rm mf}_0=Q_0$ and draw exactly one arrow from $i$ to $j$ in $Q^{\rm mf}$ if and only if there is an arrow from $i$ to $j$ in $Q$.
In particular, $Q^{\rm mf}$ has no multiple arrows.
\begin{pro}\label{prop-multiple-free}
Let $Q$ be a finite acyclic quiver. Then we have 
\[\qhstr(\kk Q) = \qhstr(\kk Q^{\rm mf}).\]


\end{pro}
\begin{proof}

If $Q$ has no multiple arrows, then $Q=Q^{\rm mf}$ holds, so there is nothing to show.
Assume that there exist multiple arrows form $i'$ to $j'$ in $Q$.
We denote by $Q'$ a quiver obtained by removing one arrow $\alpha$ between $i'$ and $j'$ in $Q$.
We show that $\qhstr(\kk Q) = \qhstr(\kk Q')$ holds.
Since $\kk Q'=\kk Q/\ideal{\alpha}$, a canonical surjection from $\kk Q$ to $\kk Q'$ induces a fully faithfull functor from $\modu\kk Q'$ to $\modu\kk Q$.
By this functor, we regard $\kk Q'$-modules as $\kk Q$-modules.

Let $\lhd$ be a partial order on $I=Q_0=Q'_0$.
We claim that $(I, \lhd)$ is adapted to $\kk Q$ if and only if it is adapted to $\kk Q'$.
Assume that $(I, \lhd)$ is adapted to $\kk Q$.
Let $M$ be a $\kk Q'$-module with top $S(i)$ and socle $S(j)$.
Since $Q'$ is acyclic, there is a path from $i$ to $j$ in $Q'$.
By regarding $M$ as a $\kk Q$-module, there is a vertex $k$ such that $i\lhd k$, $j\lhd k$ and $[M : S(k)]\neq 0$.
Therefore $(I, \lhd)$ is adapted to $\kk Q'$.
Conversely, assume that $(I, \lhd)$ is adapted to $\kk Q'$.
Let $N$ be a $\kk Q$-module with top $S(i)$ and socle $S(j)$.
Since $Q$ is acyclic, there is a path $p$ from $i$ to $j$ in $Q$ such that $S(\ell)$ is a composition factor of $N$ for each vertex $\ell$ where $p$ passes.
By the construction of $Q'$, we can take such path in $Q'$.
There is a $\kk Q'$-module $X(p)$ along with the path $p$, that is, $X(p)$ is a uniserial $\kk Q'$-module with top $S(i)$, socle $S(j)$ and $S(\ell)$ is a composition factor of $X(p)$ if and only if $p$ passes through $\ell$.
Since $(I, \lhd)$ adapted to $\kk Q'$, there is a vertex $k$ where $p$ passes  such that $i\lhd k$, $j\lhd k$ and $[X(p) : S(k)]\neq 0$.
Since $S(k)$ is also a composition factor of $N$, we have that $(I, \lhd)$ is adapted to $\kk Q$.

Let $\lhd_1$ and $\lhd_2$ be partial order of $I$ which are adapted to $\kk Q$ and $\kk Q'$.
By Proposition \ref{pro-hered-adapted}, both of then define quasi-hereditary structures of $\kk Q$ and $\kk Q'$.
Since we need to distinguish quasi-hereditary structures of $\kk Q$ or $\kk Q'$, we write $\lhd_1$ and $\lhd_2$ when we regard them as partial orders of $Q_0$, and write $\lhd_1'$ and $\lhd_2'$ when we regard them as partial orders of $Q_0'$.

First we show that $\Dec(\lhd_{\ell}) = \Dec(\lhd_{\ell}')$ for $\ell = 1,2$. Indeed, $(i,j)\in\Dec(\lhd_{\ell})$ if and only if there is a path $p$ in $Q$ from $j$ to $i$ such that $k\lhd_{\ell} j$ holds for each vertex $k$ where passes $p$. If $p$ contains $\alpha$, then by replacing by another arrow from $i'$ to $j'$, we have that a path in $Q'$, so we get $(i, j)\in\Dec(\lhd_{\ell}')$. The converse is trivial because a path in $Q'$ is also a path in $Q$. 

By Lemma \ref{lem-surj-Dec}, two adapted orders are equivalent if and only if they have the same sets of decreasing relations. Hence, two adapted orders for $kQ$ are equivalent if and only if they are equivalent for $kQ'$. Moreover, since both algebras are hereditary, by Lemma \ref{lem-surj-Dec}, the partial ordering of the quasi-hereditary structures only involves the set of decreasing relations. Hence, $[\lhd_1] \preceq [\lhd_2]$ in $\qhstr(\kk Q)$ if and only if $[\lhd_1'] \preceq [\lhd_2']$ in $\qhstr(\kk Q')$. Therefore, we have $\qhstr(\kk Q)=\qhstr(\kk Q')$.
\end{proof}




\begin{rem}
Proposition \ref{prop-multiple-free} is elementary from the point of view of equivalence classes of adapted orders, but it is more surprising from the point of view of characteristic tilting modules. For example, if $Q$ is a generalized Kronecker quiver with $m$ arrows. The tilting theory of $Q$ is very simple if $m=1$: there are only two tilting modules and they both are characteristic tilting modules for some quasi-hereditary structures. However, when $m>1$, there are infinitely many tilting modules but only two characteristic tilting modules. 
\end{rem}
\begin{exa}
Let $Q$ be a finite acyclic quiver and $\kk Q$ be the path algebra of $Q$.
Then $\qhstr(\kk Q)$ always admits both a unique maximal element and a unique minimal element.
In fact, it is easy to construct a total order $\lhd$ on $Q_0$ such that all corresponding standard modules are projective by Lemma \ref{lem-qh-maximal}.
Then $[\lhd]$ is a unique maximal element in $\qhstr(\kk Q)$.
\end{exa}
\begin{exa}\label{ex-complete-quiver}
Let $K_n$ be a quiver such that the set of vertices is $I=\{1,2,\dots,n\}$ and there is a unique arrow from $i$ to $j$ whenever $i>j$.
In particular, the underlying graph of $K_n$ is a complete graph.
It is easy to see that any adapted order to $\kk K_n$ is a total order on $I$, and two distinct total orders on $I$ induce different quasi-hereditary structures on $\kk K_n$.

Let $S_n$ be the permutation group on $I$.
For $v,w\in S_n$, we write $v \leq_S w$ if every pair $i,j$ in $I$ such that $i< j$ and $w^{-1}(i)> w^{-1}(j)$ also satisfies $v^{-1}(i)> v^{-1}(j)$.
It is well known that this $\leq_S$ induces a partial order on $S_n$ called the weak (Bruhat) order (we refer to \cite[Proposition 3.1.3]{coxeter_combinatorics} for the proof and to \cite[Chapter 3]{coxeter_combinatorics} for historical comments on the weak order.) 

We have a bijection from $S_n$ to $\qhstr(\kk K_n)$ by $w \mapsto w(1)\lhd_w w(2) \lhd_w\dots\lhd_w w(n)$.
This bijection gives an isomorphism of posets:
\[
(S_n, \leq_S) \cong  (\qhstr(\kk K_n), \preceq).
\]
In fact, by Lemma \ref{lem-surj-Dec}, it is enough to show that $v \leq_S w$ holds if and only if $\Dec(\lhd_v) \subseteq \Dec(\lhd_w)$ holds.
This holds because, by the construction of $K_n$, a pair $(i, j)\in I^2$ belongs to $\Dec(\lhd)$ if and only if $i\leq j$ and $i\lhd j$ hold for a total order $\lhd$ on $I$.
\end{exa}
\section{Quasi-hereditary structures and deconcatenations}\label{section-dec}
\subsection{Deconcatenations at a sink or a source}\label{subsection-qh-sink-source}
Throughout this section, let $Q$ be a finite connected quiver and $v$ be a sink or a source of $Q$.
All algebras are assumed to be finite dimensional over a field.
\begin{dfn}\label{dfn-deconcatenation}
A \emph{deconcatenation} of $Q$ at a sink or a source $v$ is a disjoint union $Q^1 \sqcup Q^2 \sqcup \dots \sqcup Q^{\ell}$ of proper full subquivers $Q^i$ of $Q$ satisfying the following properties:
\begin{enumerate}
\item
each $Q^i$ is a connected full subquiver of $Q$ having a vertex $v$,
\item
$Q_0=\left(Q^1_0\setminus\{v\}\right) \sqcup \dots \sqcup \left(Q^{\ell}_0\setminus\{v\}\right) \sqcup \{v\}$ and $Q^i_0 \cap Q^j_0 = \{v\}$ hold, and
\item
there are no arrows between $u$ and $w$ in $Q$, where $u \in Q^i_0\setminus\{v\}$ and $w\in Q^j_0\setminus\{v\}$ for $1\leq i \neq j \leq \ell$.
\end{enumerate}
\end{dfn}

\begin{exa}\label{exa-dec}
Let $Q=1 \to 2 \leftarrow 3 \leftarrow 4 \to 5$.
Then we have two deconcatenations of $Q$.
\[
(1 \to 2) \sqcup (2 \leftarrow 3 \leftarrow 4 \to 5),
\qquad 
(1 \to 2 \leftarrow 3 \leftarrow 4) \sqcup (4 \to 5).
\]
Moreover, the former has a deconcatenation at $4$ and the latter has a deconcatenation at $2$, and the resulting quiver is the same as follows:
\[
(1\to 2)\sqcup(2\leftarrow 3 \leftarrow 4)\sqcup(4\to 5).
\]
\end{exa}

In this section, for a given deconcatenation $Q^1 \sqcup Q^2 \sqcup \dots \sqcup Q^{\ell}$ of $Q$, we compare quasi-hereditary structures on algebras whose Gabriel quivers are $Q$ and $Q^i$.
It is easy to see that if $Q^1 \sqcup Q^2$ is a deconcatenation of $Q$ at a vertex $v$, and $Q^3 \sqcup Q^4$ is a deconcatenation of $Q^2$ at the vertex $v$, then $Q^1 \sqcup Q^3 \sqcup Q^4$ is a deconcatenation of $Q$.
Therefore, we consider a deconcatenation which is the disjoint union of two full subquivers.

Let $Q^1 \sqcup Q^2$ be a deconcatenation of $Q$ at a sink or a source $v$.
Let $A$ be a factor algebra of $\kk Q$ modulo some admissible ideal.
For each $\ell=1,2$, let
$$
A^{\ell} \coloneqq  \frac{A}{\langle e_u \mid u \in Q_0\setminus Q_0^{\ell} \rangle}.
$$
Thus we have a surjective morphism of algebras $A \to A^{\ell}$, and this induces a fully faithful exact functor $\modu A^{\ell} \to \modu A$.
By this functor, we regard $\modu A^{\ell}$ as a full subcategory of $\modu A$.
Therefore, an $A$-module $M$ is an $A^{\ell}$-module if and only if $Me_u=0$ for any $u\in Q_0\setminus Q_0^{\ell}$.
For a vertex $i\in Q_0^{\ell}$, let $P^{\ell}(i)$, $I^{\ell}(i)$, $S^{\ell}(i)$ be the indecomposable projective, indecomposable injective and the simple $A^{\ell}$-module associated to the vertex $i$, respectively.
The following lemma is easy and we omit the proof.
\begin{lem}\label{lem_sink_module}
Let $Q^1 \sqcup Q^2$ be a deconcatenation of $Q$ at a sink or a source $v$.
Let ${\ell}=1,2$.
\begin{enumerate}
\item
For any vertex $i\in Q_0^{\ell}$, we have $S(i) \cong  S^{\ell}(i)$.
\item
For any $i\in Q_0^{\ell}\setminus\{v\}$, we have $P(i) \cong  P^{\ell}(i)$ and $I(i)\cong  I^{\ell}(i)$.
\item
If $v$ is a sink, then we have $S(v) \cong  P(v) \cong  P^{\ell}(v)$ for $\ell=1,2$.
\item
If $v$ is a source, then we have $S(v)\cong  I(v) \cong  I^{\ell}(v)$ for $\ell=1,2$.
\item
Let $M$ be a non-zero $A$-module.
If both of the top and the socle of $M$ are simple, then one of $M\in\modu A^1$ or $M\in \modu A^2$ holds.
\item
Let $M\in\modu A^{\ell}$ and $i\in Q_0$.
If $[M : S(i)]\neq 0$, then $i\in Q_0^{\ell}$ holds.
\end{enumerate}
\end{lem}
Let $\lhd$ be a partial order on $Q_0$.
By restricting this order, we have a partial order $\lhd|_{Q_0^{\ell}}$ on $Q_0^{\ell}$ for $\ell=1,2$.
We first compare standard and costandard modules associated to these orders.
\begin{lem}\label{lem_Q_Ql}
Let $Q^1 \sqcup Q^2$ be a deconcatenation of $Q$ at a sink or a source $v$.
Let $\lhd$ be a partial order on $Q_0$ and $\Delta$ (\,$\nabla$, respectively) the set of standard (costandard, respectively) $A$-modules associated to $\lhd$.
Let $\ell=1,2$.
We denote by $\Delta^{\ell}$ (\,$\nabla^{\ell}$, respectively) the set of standard (costandard, respectively) $A^{\ell}$-modules associated to $\lhd|_{Q_0^{\ell}}$.
Then we have the following statements.
\begin{enumerate}
\item
For any $i\in Q_0^{\ell}\setminus\{v\}$, we have $\Delta(i)\cong  \Delta^{\ell}(i)$ and $\nabla(i)\cong  \nabla^{\ell}(i)$.
\item
If $v$ is a sink, then we have $S(v) \cong  \Delta(v)\cong  \Delta^{\ell}(v)$.
\item
If $v$ is a source, then we have $S(v) \cong  \nabla(v) \cong  \nabla^{\ell}(v)$.
\item
If $\lhd$ defines a quasi-hereditary structure on $A$, then $\lhd|_{Q_0^{\ell}}$ defines a quasi-hereditary structure on $A^{\ell}$ for each ${\ell}=1,2$.
\end{enumerate}
\end{lem}
\begin{proof}
(1)
Since $P(i) \cong  P^{\ell}(i)$ for $i\in Q_0^{\ell}\setminus\{v\}$, a composition factor $S(j)$ of $P(i)$ satisfies $j\in Q_0^{\ell}$.
Since $\lhd^{\ell}$ is a restriction of $\lhd$, for any $j\in Q_0^{\ell}$, $j\lhd i$ if and only if $j\lhd^{\ell} i$.
Thins implies that $\Delta^{\ell}(i) \cong  \Delta(i)$.
Similarly, we have $\nabla^{\ell}(i)\cong  \nabla(i)$.
(2) (3) The assertions are clear.
(4)
Assume that $v$ is a sink.
Let $i\in Q_0^{\ell}$.
By (1), $\Delta(i) \cong  \Delta^{\ell}(i)$ holds.
Thus we have $[\Delta^{\ell}(i):S^{\ell}(i)]=1$ by Lemma \ref{lem_sink_module}.
Since any composition factor of $P(i)$ is a simple $A^{\ell}$-module, a $\Delta$-filtration of $P(i)$ in $\modu A$ gives a $\Delta^{\ell}$-filtration of $P^{\ell}(i)$ in $\modu A^{\ell}$.
Clearly, this filtration satisfies the axiom (3) of Definition \ref{dfn-qh-alg}.
Thus $\lhd|_{Q_0^{\ell}}$ defines a quasi-hereditary structure on $A^{\ell}$.
If $v$ is a source, then by Lemma \ref{lem-dual}, the assertion holds.
\end{proof}
Set $\overline{1}=2$ and $\overline{2}=1$.
Next we construct a partial order on $Q_0$ from partial orders on $Q_0^{\ell}$.
Let $\lhd^{\ell}$ be partial orders on $Q_0^{\ell}$ for $\ell=1, 2$.
Then we have a partial order $\lhd = \lhd(\lhd^1, \lhd^2)$ on $Q_0$ as follows: for $i, j\in Q_0$, $i\lhd j$ if one of the following two statements holds:
\begin{enumerate}
\item
$i,j\in Q_0^{\ell}$ and $i\lhd^{\ell} j$ holds for some $\ell$,
\item
$i\in Q_0^{\ell}$, $j\in Q_0^{\overline{\ell}}$, $i \lhd^{\ell} v$ and $v \lhd^{\overline{\ell}} j$ hold.
\end{enumerate}
\begin{lem}\label{lem_Ql_Q}
Let $Q^1 \sqcup Q^2$ be a deconcatenation of $Q$ at a sink or a source $v$.
Let $\lhd^{\ell}$ be a partial order on $Q_0^{\ell}$ and $\Delta^{\ell}$ (\,$\nabla^{\ell}$, respectively) the set of standard (costandard, respectively) $A^{\ell}$-modules associated to $\lhd^{\ell}$ for $\ell=1,2$.
We denote by $\Delta$ (\,$\nabla$, respectively) the set of standard (costandard, respectively) $A$-modules associated to $\lhd=\lhd(\lhd^1, \lhd^2)$.
Let $\ell=1,2$.
Then we have the following statements.
\begin{enumerate}
\item
For any $i\in Q_0^{\ell}\setminus\{v\}$, we have $\Delta(i)\cong  \Delta^{\ell}(i)$ and $\nabla(i)\cong  \nabla^{\ell}(i)$.
\item
If $v$ is a sink, then we have $S(v) \cong  \Delta(v)\cong  \Delta^{\ell}(v)$.
\item
If $v$ is a source, then we have $S(v) \cong  \nabla(v) \cong  \nabla^{\ell}(v)$.
\item
If $\lhd^{\ell}$ defines a quasi-hereditary structure on $A^{\ell}$ for both $\ell=1,2$, then $\lhd=\lhd(\lhd_1, \lhd_2)$ defines a quasi-hereditary structure on $A$.
\end{enumerate}
\end{lem}
\begin{proof}
(1)
Since $P(i) \cong  P^{\ell}(i)$ for $i\in Q_0^{\ell}\setminus\{v\}$, a composition factor $S(j)$ of $P(i)$ satisfies $j\in Q_0^{\ell}$.
Then by the definition of $\lhd$, we have that, for a composition factor $S(j)$ of $P(i)$, $j\lhd i$ holds if and only if $j \lhd^{\ell} i$ holds.
This implies the assertion.
This implies that $\Delta(i)\cong  \Delta^{\ell}(i)$.
(2) (3) The assertions are clear.
(4)
Assume that $v$ is a sink.
Let $i\in Q_0^{\ell}$.
By (1), we have $\Delta(i) \cong  \Delta^{\ell}(i)$.
In particular, we have $[\Delta(i): S(i)]=1$.
By Lemma \ref{lem_sink_module}, $P(i)\cong  P^{\ell}(i)$ holds for $i\in Q_0^{\ell}$.
Since $\lhd^{\ell}$ defines a quasi-hereditary structure on $A^{\ell}$, $P(i)$ has a $\Delta^{\ell}$-filtration.
Since the number $(P(i):\Delta(j))$ does not depend on the choice of the filtration, the axiom (3) of Definition \ref{dfn-qh-alg} is satisfied.
If $v$ is a source, then by Lemma \ref{lem-dual}, the assertion holds.
\end{proof}

Let $Q^1 \sqcup Q^2$ be a deconcatenation of $Q$ at a sink or a source $v$.
Let $A$ be a factor algebra of $\kk Q$ modulo some admissible ideal.
By Lemmas \ref{lem_Q_Ql} and \ref{lem_Ql_Q}, we have the following map.
\begin{align*}
\Phi \colon \qhstr(A) \longrightarrow \qhstr(A^1) \times \qhstr(A^2), \qquad [\lhd] \mapsto \left( [\lhd|_{Q_0^1}], [\lhd|_{Q_0^2}] \right).
\end{align*}
We also have an inverse map
\begin{align*}
\Psi \colon \qhstr(A^1) \times \qhstr(A^2) \longrightarrow \qhstr(A), \qquad \left( [\lhd^1], [\lhd^2] \right) \mapsto [\lhd(\lhd^1, \lhd^2)].
\end{align*}
For two posets $(A,\leq_A), (B,\leq_B)$ and $(a_1,b_1),(a_2,b_2)\in A\times B$, we write $(a_1,b_1)\leq (a_2,b_2)$ if $a_1\leq_A a_2$ and $b_1\leq_B b_2$ hold.
Then $(A\times B, \leq)$ is a poset, called the \emph{product poset}.
\begin{pro}\label{pro_qhstr_divide}
The map $\Phi$ is an isomorphism of posets and its inverse is $\Psi$.
\end{pro}
\begin{proof}
Consider the product poset on $ \qhstr(A^1) \times \qhstr(A^2)$.
Then the assertion directly follows from Lemmas \ref{lem_two_qh_str_equiv}, \ref{lem_Q_Ql} and \ref{lem_Ql_Q}, since (co)standard modules are preserved by the maps.
\end{proof}
Let $Q^1\sqcup Q^2 \sqcup \dots \sqcup Q^{\ell}$ be a deconcatenation of $Q$ at a vertex $v$.
If $Q^{\ell+1}\sqcup \dots \sqcup Q^{m}$ is a deconcatenation of $Q^{\ell}$ at a vertex $u$,
then we have a disjoint union $Q^1\sqcup Q^2 \sqcup \dots \sqcup Q^{\ell-1} \sqcup Q^{\ell+1}\sqcup \dots \sqcup Q^{m}$ of full subquivers of $Q$, and so on for each connected quiver $Q^i$.
We call a disjoint union $Q^1\sqcup Q^2 \sqcup \dots \sqcup Q^{\ell^{\prime}}$ of full subquivers of $Q$ obtained by iterated operations as above an \emph{iterated deconcatenation} of $Q$.

Then we have the following theorem.
\begin{thm}\label{thm_qhstr_divide}
Let $Q^1\sqcup Q^2 \sqcup \dots \sqcup Q^{\ell}$ be an iterated deconcatenation of $Q$ at sink or source vertices.
Let $A$ be a factor algebra of $\kk Q$ modulo some admissible ideal and $A^i\coloneqq  A / \langle e_u \mid u \in Q_0\setminus Q_0^i \rangle$.
Then we have an isomorphism of posets 
\begin{align*}
\qhstr(A) \longrightarrow \prod_{i=1}^{\ell}\qhstr(A^i),
\end{align*}
which is given by $[\lhd] \mapsto \left( [\lhd|_{Q_0^i}] \right)_{i=1}^{\ell}$.
\end{thm}
\begin{proof}
By applying Proposition \ref{pro_qhstr_divide} iteratively, we have the assertion.
\end{proof}
\begin{exa}
Let $Q=1 \to 2 \leftarrow 3 \leftarrow 4 \to 5$.
We have an isomorphism of posets
\[
\qhstr(\kk Q) \longrightarrow \qhstr(\La_2)\times\qhstr(\La_3)\times\qhstr(\La_2),
\]
where $\La_n$ is the path algebra of an equioriented quiver of type $\A_n$.
We study $\qhstr(\La_n)$ precisely in Section \ref{section-An}.
\end{exa}
We use the following lemma later.
\begin{lem}\label{lem-sink-source-minimal}
Let $Q^1 \sqcup Q^2$ be a deconcatenation of $Q$ at a sink or a source $v$.
Let $A$ be a factor algebra of $\kk Q$ modulo some admissible ideal.
Let $\lhd$ ($\lhd^1$, $\lhd^2$ respectively) be a partial order on $Q_0$ ($Q_0^1$, $Q_0^2$, respectively) defining a quasi-hereditary structure on $A$ ($A^1$, $A^2$, respectively).
Then the following statements hold.
\begin{enumerate}
\item
If $\lhd$ is a minimal adapted order, then both $\lhd|_{Q_0^1}$ and $\lhd|_{Q_0^2}$ are minimal adapted orders.
\item
If $\lhd^1$ and $\lhd^2$ are minimal adapted orders, then $\lhd(\lhd^1, \lhd^2)$ is a minimal adapted order.
\end{enumerate}
In particular, if $\lhd$ is minimal, then $\lhd=\lhd(\lhd|_{Q_0^1}, \lhd|_{Q_0^2})$ holds.
\end{lem}
\begin{proof}
We show only (1) and the last assertion.
The assertion (2) is shown similarly.
We show that $\lhd_1 \coloneqq \lhd|_{Q_0^1}$ is a minimal adapted order.
Let $\lhd_1^{\prime}$ be a partial order on $Q_0^1$ such that $\lhd_1^{\prime} \sim \lhd_1$.
Let $i,j\in Q_0^1$ and assume that $i \lhd_1 j$ holds.
This implies that $i \lhd j$.
By Proposition \ref{pro_qhstr_divide}, we have $\lhd^{\prime}\coloneqq  \lhd(\lhd^{\prime}_1, \lhd|_{Q_0^2})\sim \lhd$.
Since $\lhd$ is minimal, $i\lhd^{\prime} j$ holds.
By the definition of $\lhd^{\prime}$, we have $i \lhd_1^{\prime}j$.
The last assertion follows from a uniqueness of a minimal adapted order.
\end{proof}
\subsection{Characteristic tilting modules and deconcatenations}
In this subsection, we compare characteristic tilting modules via deconcatenations.
Throughout this subsection let $Q^1 \sqcup Q^2$ be a deconcatenation of $Q$ at a sink $v$.
We denote by $A$ a factor algebra of $\kk Q$ modulo some admissible ideal.
Let $A^{\ell}=A/\ideal{e_u \mid u\in Q_0\setminus Q_0^{\ell}}$ for $\ell=1,2$.
Since $A^1$ and $A^2$ are factor algebras of $A$, we regard $A^1$-modules and $A^2$-modules as $A$-modules.

Let $\lhd$ be a partial order on $Q_0$ defining a quasi-hereditary structure on $A$ and $\lhd^{\ell}\coloneqq \lhd|_{Q^{\ell}_0}$ for $\ell=1,2$.
By Lemma \ref{lem_Q_Ql}, $\lhd^{\ell}$ defines a quasi-hereditary structure on $A^{\ell}$.
We denote by $\Delta^{\ast}(i)$, $\nabla^{\ast}(i)$, $T^{\ast}(i)$ a standard module, a costandard module and an indecomposable direct summand of the characteristic tilting module of a quasi-hereditary algebra $(A^{\ast},(Q_0^{\ast}, \lhd^{\ast}))$ for $\ast=\emptyset, 1, 2$.
Note that $\nabla(v)$ is obtained by the push-out of $\nabla^1(v) \leftarrow S(v) \to \nabla^2(v)$.
By Lemma \ref{lem_Q_Ql}, we have that $\mcF(\Delta^{\ell})$ is contained in $\mcF(\Delta)$ and $\mcF(\nabla^{\ell}(u) \mid u\neq v)$ is contained in $\mcF(\nabla)$ for $\ell=1,2$.
Since $v$ is a sink, $\Delta^1(v)=\Delta^2(v)=S(v)$ is a simple projective $A^{\ast}$-module for $\ast=\emptyset, 1, 2$.

Let $e^{\ell}=\sum_{u\in Q_0^{\ell}} e_u$.
Then it is easy to see that $A^{\ell}\cong  e^{\ell}Ae^{\ell}$ as algebras.
Thus we have an exact functor $\modu A \to \modu A^{\ell}$ defined by $M \mapsto Me^{\ell}$.
Moreover, we have $\Delta(i)e^{\ell} \cong  \Delta^{\ell}(i)$ if $i \in Q_0^{\ell}$ and $\Delta(i)e^{\ell}=0$ if $i\notin Q_0^{\ell}$.
This restricts to functors $\mcF(\Delta) \to \mcF(\Delta^{\ell})$ and $\mcF(\nabla) \to \mcF(\nabla^{\ell})$, since the functor is exact and $\nabla(v)e^{\ell}=\nabla^{\ell}(v)$ holds.
Note that by regarding $Me^{\ell}$ as an $A$-module via a map $A \to A^{\ell}\cong  e^{\ell}Ae^{\ell}$, we have an injective morphism $Me^{\ell} \to M$ of $A$-modules.

Throughout this subsection, we fix one $i\in Q_0^1$.
Let $T^1(i)$ be an indecomposable direct summand of the characteristic tilting $A^1$-module.
We denote by $m$ the length of an \emph{$S(v)$-socle} $\soc_v(T^1(i))$ of $T^1(i)$, that is, the maximal direct summand of $\soc(T^1(i))$ which is a direct sum of copies of $S(v)$.
So there exists an injective morphism $s\colon  S(v)^{\oplus m} \to T^1(i)$.
Recall that by Proposition \ref{pro-ch-tilt} we have a short exact sequence
\[
0\to S(v) \xto{f} T^2(v) \to X^2(v) \to 0,
\]
where $X^2(v)$ belongs to $\mcF(\Delta^2(j) \mid j \lhd v, j\neq v)$.
Consider the following push-out diagram.
\begin{align}\label{diagram-push-out-U}
\begin{tikzpicture}[baseline=-30,xscale=.8]
\node(S)at(0,0){$S(v)^{\oplus m}$};
\node(Y)at(3,0){$T^1(i)$};
\node(T)at(0,-2){$T^{2}(v)^{\oplus m}$};
\node(Q)at(3,-2){$U(i)$};
\draw[thick, ->] (S)--(T) node[midway, left]{$f^{\oplus m}$};
\draw[thick, ->] (S)--(Y) node[midway, above]{$s$};
\draw[thick, ->] (Y)--(Q) node[midway, right]{$\beta$};
\draw[thick, ->] (T)--(Q); 
\end{tikzpicture}
\end{align}
In Theorem \ref{thm-dec-ch-tilt} we show that $U(i)$ is isomorphic to $T(i)$.
Since the cokernel of $\beta$ is isomorphic to $X^2(v)^{\oplus m}$, $U(i)$ belongs to $\mcF(\Delta)$.
We have the following proposition.
\begin{pro}\label{pro-push-out-int}
The push-out $U(i)$ belongs to $\mcF(\Delta)\cap\mcF(\nabla)$.
Moreover, we have $U(i)e^1 \cong  T^1(i)$ and $U(i)e^2 \cong  T^2(v)^{\oplus m}$.
\end{pro}
\begin{proof}
The diagram induces a short exact sequence $0 \to S(v)^{\oplus m}\to T^2(v)^{\oplus m}\oplus T^1(i) \to U(i) \to 0$.
Note that multiplication by an idempotent yields an exact functor.
Thus by multiplying $U(i)$ on the right by $e^1$ and $e^2$, we have $U(i)e^1 \cong  T^1(i)$ and $U(i)e^2 \cong  T^2(v)^{\oplus m}$.

We show that $U(i)$ belongs to $\mcF(\nabla)$.
By \cite[Theorem 1]{dlab-ringel}, $\mcF(\nabla)=\{Y\in\modu A \mid \Ext_A^1(\Delta, Y)=0\}$ holds.
Therefore we show that $\Ext_A^1(\Delta(j), U(i))=0$ for any $j\in Q_0$.
Indeed, if $j=v$, then $\Ext_A^1(\Delta(v),U(i))=0$ holds, since $\Delta(v)$ is a projective $A$-module.
Take a short exact sequence 
\begin{align}\label{ses-push-out}
0 \to U(i) \to X \xto{g} \Delta(j) \to 0
\end{align}
in $\modu A$. 
Assume that $j\in Q_0^1\setminus\{v\}$.
We have a short exact sequence $0 \to U(i)e^1 \to Xe^1 \xto{g e^1} \Delta(j)e^1 \to 0$ in $\modu(e^1A e^1)$.
Since $U(i)e^1\cong  T^1(i)$ and $\Delta(j)e^1 \cong  \Delta^1(j)$ and $T^1(i)$ belongs to $\mcF(\nabla^1)=\{Y\in\mod A^1 \mid \Ext_{A^1}^1(\Delta^1, Y)=0\}$, then this sequence splits.
Thus there exists an $(e^1 A e^1)$-morphism $h \colon \Delta(j)e^1 \to Xe^1$ such that $g e^1\circ h = \id$.
By Lemma \ref{lem_Q_Ql}, $\Delta(j) = \Delta(j)e^1$ holds.
So a map $h \colon \Delta(j) \to X$ given by $d \mapsto h(d)$ is well-defined.
We show that $h$ is a morphism of $A$-modules.
Let $a\in A$ and $d\in \Delta(j)$.
Note that $da=de^1ae^1$ holds, since $\Delta(j) = \Delta(j)e^1$ holds.
Then we have that $h(da)=h(d)e^1ae^1$.
On the other hand, since $\im(h)\subset Xe^1$, we have $h(d)a=h(d)e^1ae^1 + h(d)e^1a(1-e^1)$.
However, $e^1 A (1-e^1)=0$ holds, since $Q^1\sqcup Q^2$ is a deconcatenation at a sink.
It is easy to see that $g\circ h=\id$, that is, the short exact sequence (\ref{ses-push-out}) splits.
If $j\in Q_0^2\setminus\{v\}$, then by using the fact that $U(i)e^2 \cong  T^2(v)^{\oplus m}$, $\Delta(j)=\Delta(j)e^2 \cong  \Delta^2(j)$ and $e^2 A (1-e^2)=0$, we can show that $g$ is a split morphism in a similar way as before.
\end{proof}
\begin{lem}\label{lem-push-out-appro}
In the diagram (\ref{diagram-push-out-U}), the morphism $\beta$ is a left $\mcF(\nabla)$-approximation of $T^1(i)$.
\end{lem}
\begin{proof}
By Proposition \ref{pro-push-out-int}, $U(i)\in\mcF(\nabla)$.
Let $M\in\mcF(\nabla)$ and $g \colon T^1(i) \to M$ be a morphism.
Since the image of $gs$ is a direct sum of copies of $S(v)$, the image is contained in $Me^2$.
Since $Me^2\in\mcF(\nabla^2)$ and $S(v)\to T^2(v)$ is a left $\mcF(\nabla^2)$-approximation of $S(v)$, there exists a morphism $h \colon T^2(v)^{\oplus m} \to M$ such that $hf^{\oplus m} = gs$.
Because the diagram is a push-out diagram, there exists a morphism $g^\p \colon U(i) \to M$ such that $g^{\p}\beta=g$.
\end{proof}
Recall that $\Delta(i)=\Delta^1(i)$ holds for $i\in Q^1_0$.
\begin{lem}\label{lem-delta-hatT-appro}
Let $f_1 \colon \Delta(i)\to T^1(i)$ be a morphism obtained by Proposition \ref{pro-ch-tilt}.
Then the composite $\Delta(i) \xto{f_1} T^1(i) \xto{\beta} U(i)$ is a left $\mcF(\nabla)$-approximation of $\Delta(i)$.
\end{lem}
\begin{proof}
We denote by $\iota \colon Me^1 \to M$ the inclusion map.
Let $M\in\mcF(\nabla)$ and $h \colon \Delta(i) \to M$ be a morphism.
Since $\Delta^1(i)=\Delta(i)$ holds, there exists a morphism $h' \colon \Delta(i) \to Me^1$ with $h=\iota h^\p$.
Since $f_1$ is a left $\mcF(\nabla^1)$-approximation of $\Delta(i)$ and $Me^1\in\mcF(\nabla^1)$, there exists a morphism $f_1^\p \colon T^1(i) \to Me^1$ with $h^{\p}=f_1^{\p}f_1$.
By Lemma \ref{lem-push-out-appro}, there exists a morphism $f^{\p\p}_1 \colon  U(i) \to M$ with $\iota f_1^\p=f_1^{\p\p}\beta$.
Then we have $h=\iota h^\p=\iota f_1^{\p}f_1 = f_1^{\p\p}\beta f_1$.
\[
\begin{tikzcd}[arrow style=tikz, arrows={thick},row sep=12mm,baseline=-35]
& \Delta(i) \arrow[r, "f_1"] \arrow[dl, "h^{\p}"', densely dashed, end anchor={[xshift=-1.2ex]}] \arrow[d, "h" pos=.73] & T^1(i) \arrow[r, "\beta"] \arrow[dll, "{f_1^\p}" description, pos=.22, densely dashed, preaction={draw=white, -,line width=4pt}] & U(i) \arrow[dll, "f_1^{\p\p}", densely dashed] \\
Me^1 \arrow[r, "\iota"'] & M 
\end{tikzcd}\qedhere
\]
\end{proof}
Let $R$ be a ring with a surjective ring morphism $R \to R/I$.
It is easy to see that if $R/I$ is a local ring and $I$ is contained in the radical of $R$, then $R$ is a local ring.
The following lemma implies that $U(i)$ is indecomposable.
\begin{lem}\label{lem-push-out-ind}
Let $A$ be a finite dimensional algebra and $S$ be a simple $A$-module.
Let $X,Y$ be $A$-modules such that the $S$-socles are $S^{\oplus m}$ and $S$, respectively, with $m\geq 1$.
Consider the following push-out diagram of $A$-modules
\[
\begin{tikzcd}[arrow style=tikz, arrows={thick}]
S^{\oplus m} \arrow[r, "a"] \arrow[d, "b"] & X \arrow[d, "\beta"] \\
Y^{\oplus m} \arrow[r, "\alpha"] & Z
\end{tikzcd}
\]
such that $a$ and $b$ are inclusion morphisms associated to $S$-socles.
Assume that the following conditions hold:
\begin{enumerate}
\item there exist no non-zero morphisms from $X$ to the cokernel of $b$,
\item there exist no non-zero morphisms from $Y$ to the cokernel of $a$,
\item $\beta$ is a left $(\add Z)$-approximation of $X$,
\item $X$ and $Y$ are indecomposable.
\end{enumerate}
Then $Z$ is indecomposable.
\end{lem}
\begin{proof}
Note that since we have a push-out diagram, $\coKer(a) \simeq \coKer(\alpha)$ and $\coKer(b) \simeq \coKer(\beta)$.
Moreover, as $a$ and $b$ are injective, $\alpha$ and $\beta$ are also injective.
By this injective morphisms, we regard $Y^{\oplus m}$ and $X$ as submodules of $Z$.

We show that the endomorphism algebra of $Z$ is a local algebra.
Let $\pi : Z \to \coKer(\beta)$ be a canonical morphism.
By the assumption (1), for each $\phi\in\End_A(Z)$, $\pi\phi\beta=0$.
Thus we have a morphism of algebras $R_X \colon  \End_A(Z) \to \End_A(X)$ given by $R_X(\phi)= \phi|_{X}$.
By the assumption (3), this $R_X$ is surjective.
We have an exact sequence
\[
0 \to \Ker(R_X) \to \End_A(Z) \to \End_A(X) \to 0.
\]
By the assumption (4), $\End_A(X)$ is a local algebra.
Thus it is enough to show that $\Ker(R_X)$ is contained in the radical of $\End_A(Z)$.

Let $\phi$ be an endomorphism of $Z$ with $R_X(\phi)=0$.
We show that $\phi$ belongs to the radical of $\End_A(Z)$.
By the assumption (2), in an analogous way as the definition of $R_X$, we have a morphism of algebras $R_{Y^{\oplus m}} \colon  \End_A(Z) \to \End_A(Y^{\oplus m})$ given by $R_{Y^{\oplus m}}(\psi)=\psi|_{Y^{\oplus m}}$.
The endomorphism $R_{Y^{\oplus m}}(\phi)$ is presented by a $(m\times m)$-matrix with entries in $\End_A(Y)$.
Since the $S$-socle of $Y$ is $S$ and $\phi|_{S^{\oplus m}}=(\phi|_X)|_{S^{\oplus m}}=0$, all the entries of a matrix presentation of $R_{Y^{\oplus m}}(\phi)$ are not isomorphisms and so these morphisms belongs to the radical of a local algebra $\End_A(Y)$.
This implies that $R_{Y^{\oplus m}}(\phi)$ belongs to the radical of $\End_A(Y^{\oplus m})$.
In particular, $R_{Y^{\oplus m}}(\phi)$ is nilpotent.
Let $f\in\End_A(Z)$ be any morphism. Then, as before, $(f\phi)|_{S^{\oplus m}}=0$ holds since $R_X(\phi)=0$.
Thus by the same argument, $R_{Y^{\oplus m}}(f\phi)$ is nilpotent.
By the universality of the push-out diagram, $f\phi$ is nilpotent for any $f\in \End_A(Z)$.
This implies that $\id_Z-f\phi$ is an isomorphism, and so $\phi$ belongs to the radical of $\End_A(Z)$.
\end{proof}
Then we have the following result.
\begin{thm}\label{thm-dec-ch-tilt}
For $i\in Q_0^1$, let $m$ be the length of an $S(v)$-socle of $T^1(i)$.
Then the push-out $U(i)$ of $T^2(v)^{\oplus m} \leftarrow S(v)^{\oplus m} \to T^1(i)$ as (\ref{diagram-push-out-U}) is isomorphic to $T(i)$.
\end{thm}
\begin{proof}
By Lemmas \ref{lem-push-out-appro}, \ref{lem-delta-hatT-appro} and \ref{lem-push-out-ind}, $U(i)$ is indecomposable.
By Lemma \ref{lem-delta-hatT-appro}, $U(i)$ gives a left $\mcF(\nabla)$-approximation of $\Delta(i)$.
Therefore, $U(i)$ is isomorphic to $T(i)$.
\end{proof}
Note that if $v$ is a source, then an analogous results can be derived with $T(i)$ being obtained through a pull-back.

\if() 
Throughout this subsection let $Q^1 \sqcup Q^2$ be a deconcatenation of $Q$ at a sink $v$.
We denote by $A$ a factor algebra of $\kk Q$ modulo some admissible ideal, and denote by $A^1$ and $A^2$ the corresponding  factor algebras of $A$ associated to the deconcatenation, see Subsection \ref{subsection-qh-sink-source}.
Since $A^1$ and $A^2$ are factor algebras of $A$, we regard $A^1$-modules and $A^2$-modules as $A$-modules.

Let $\lhd$ be a partial order on $Q_0$ defining a quasi-hereditary structure on $A$ and $\lhd^{\ell}\coloneqq \lhd|_{Q^{\ell}_0}$ for $\ell=1,2$.
By Lemma \ref{lem_Q_Ql}, $\lhd^{\ell}$ defines a quasi-hereditary structure on $A^{\ell}$ for $\ell=1,2$.
We denote by $\Delta^{\ast}$, $\nabla^{\ast}$, $T^{\ast}(i)$ standard modules, costandard modules and indecomposable direct summands of the characteristic tilting module of a quasi-hereditary algebra $(A^{\ast},(Q_0^{\ast}, \lhd^{\ast}))$ for $\ast=\emptyset, 1, 2$.
Note that $\nabla(v)$ is obtained by the push-out of $\nabla^1(v) \leftarrow S(v) \to \nabla^2(v)$.
It is easy to see that $\mcF(\Delta^{\ell})$ is contained in $\mcF(\Delta)$ and $\mcF(\nabla^{\ell}(u) \mid u\neq v)$ is contained in $\mcF(\nabla)$ for $\ell=1,2$.
Since $v$ is a sink, $\Delta^1(v)=\Delta^2(v)=S(v)$ is a simple projective $A^{\ast}$-module for $\ast=\emptyset, 1, 2$.

In this subsection, we construct a characteristic tilting module of $(A,(Q_0,\lhd))$ from those of $(A^1,(Q_0^1,\lhd^1))$ and $(A^2,(Q_0^2,\lhd^2))$, see Proposition \ref{pro-ch-tilt-push-out} and Theorem \ref{thm-dec-ch-tilt}.
We begin with the following lemma.
\begin{lem}\label{lem-push-out}
\new{In $\modu A$,} consider the following \old{left} commutative diagram
\begin{align*}
\begin{tikzpicture}
\node(Sl)at(0,0){$S$};
\node(Sr)at(4,0){$S$};
\node(X)at(1,-1){$X$};
\node(Y)at(5,-1){$Y$};
\node(Tl)at(0,-2){$T$};
\node(Tr)at(4,-2){$T$};
\node(P)at(1,-3){$P$};
\node(Q)at(5,-3){$Q$};
\draw[thick, ->] (Sl)--(Sr) node[midway, above]{$\id$};
\draw[thick, ->] (Sl)--(X);
\draw[thick, ->] (Sr)--(Y);
\draw[thick, ->] (Sl)--(Tl);
\draw[thick, ->] (Sr)--(Tr);
\draw[thick, ->] (Tl)--(Tr) node[midway, above]{$\id$};
\draw[thick, ->, preaction={draw=white,-,line width=4pt}] (X)--(P);
\draw[thick, ->] (Tl)--(P);
\draw[thick, ->] (Y)--(Q);
\draw[thick, ->] (P)--(Q);
\draw[thick, ->, preaction={draw=white,-,line width=4pt}] (X)--(Y);
\draw[thick, ->] (Tr)--(Q);
\end{tikzpicture}
\end{align*}
where $P$ is the push-out of $T \leftarrow S \to X$.
Then $Q$ is the push-out of $P \leftarrow X \to Y$ if and only if $Q$ is the push-out of $T \leftarrow S \to Y$.
Moreover, if $S\to T$ is injective and $Q$ is a push-out, then the front square is a pull-back diagram.
\old{and $Q$ is the push-out of $P \leftarrow X \to Y$.
Let $f\colon  T \to Q$ be the composite of $T \to P \to Q$.
Then the right side square is commutative and is a push-out diagram.}
\end{lem}
\begin{proof}
The claim is shown by a diagram chasing.
\end{proof}
By using Lemma \ref{lem-push-out}, we have the following two lemmas.
Recall that by Proposition \ref{pro-ch-tilt} we have two short exact sequences
\[
0\to S(v) \to T^2(v) \to X^2(v) \to 0,
\qquad
0\to Y^2(v) \to T^2(v) \to \nabla^2(v) \to 0,
\]
where $X^2(v)$ belongs to $\mcF(\Delta^2(j) \mid j \lhd v, j\neq v)$ and $Y^2(v)$ belongs to $\mcF(\nabla^2(j) \mid j \lhd v, j\neq v )$.
\begin{lem}\label{lem-T-nabla}
Let $T_{\nabla}$ be the push-out of $T^2(v) \leftarrow S(v) \to \nabla^1(v)$.
Then $T_{\nabla}$ belongs to $\mcF(\nabla)$.
\end{lem}
\begin{proof}
Recall that $S(v)$ is a projective $A^2$-module.
Thus by Lemma \ref{lem-push-out}, we have the following commutative diagram
\begin{align*}
\begin{tikzpicture}
\node(Sl)at(0,0){$S(v)$};
\node(Sr)at(4,0){$S(v)$};
\node(X)at(1,-1){$T^2(v)$};
\node(Y)at(5,-1){$\nabla^2(v)$};
\node(Tl)at(0,-2){$\nabla^1(v)$};
\node(Tr)at(4,-2){$\nabla^1(v)$};
\node(P)at(1,-3){$T_{\nabla}$};
\node(Q)at(5,-3){$\nabla(v)$};
\node(Yu)at(-3,-1){$Y^2(v)$};
\node(Yd)at(-3,-3){$Y^2(v)$};
\draw[thick, ->] (Sl)--(X);
\draw[thick, ->] (Sr)--(Y);
\draw[thick, ->] (Sr)--(Tr);
\draw[thick, ->, preaction={draw=white,-,line width=4pt}] (X)--(Y);
\draw[thick, ->] (Sl)--(Sr) node[midway, above]{$\id$};
\draw[thick, ->] (Sl)--(Tl);
\draw[thick, ->] (Tl)--(Tr) node[midway, above]{$\id$};
\draw[thick, ->, preaction={draw=white,-,line width=4pt}] (X)--(P);
\draw[thick, ->] (Tl)--(P);
\draw[thick, ->] (Y)--(Q);
\draw[thick, ->] (P)--(Q);
\draw[thick, ->] (Tr)--(Q);
\draw[thick, ->, preaction={draw=white,-,line width=4pt}] (Yu)--(X);
\draw[thick, ->] (Yu)--(Yd) node[midway, left]{$\id$};
\draw[thick, ->] (Yd)--(P);
\end{tikzpicture}
\end{align*}
where horizontal sequences in the front side are short exact sequences.
The upper short exact sequence in the front side is given by Proposition \ref{pro-ch-tilt}. 
Since $Y^2(v)$ belongs to $\mcF(\nabla^2(u) \mid u \lhd v, u \neq v)$, the bottom short exact sequence implies that $T_{\nabla}\in\mcF(\nabla)$.
\end{proof}
The following lemma is a key lemma.
\begin{lem}\label{lem-push-out-XYZ}
Let $Z$ be an $A^1$-module and let $0 \subset X \subset Y \subset Z$ be a sequence of submodules  of $Z$.
Assume that $Y/X$ is isomorphic to $\nabla^1(v)$.
Since $S(v)$ is projective, a non-zero morphism $S(v) \to Z/X$ lifts to $S(v) \to Z$.
Then the push-out $Z^\p$ of $T^2(v) \leftarrow S(v) \to Z$ satisfies the following properties.
\begin{itemize}
\item[(i)]
There exists a short exact sequence $0 \to Z \xto{\beta} Z^\p \to X^2(v) \to 0$ of $A$-modules.
\item[(ii)]
There exists a sequence of submodules $0 \subset X^\p \subset Y^\p \subset Z^\p$ of $Z^\p$ such that $X^\p=\beta(X)$, $Y^\p/X^\p \cong  T_{\nabla}$ and $Z^\p/Y^\p \cong  Z/Y$.
\end{itemize}
\end{lem}
\begin{proof}
The factor module $Z/X$ has a submodule $\nabla^1(v)$, so there exists a non-zero morphism $S(v) \to Z/X$.
Since $S(v)$ is projective, the morphism factors through a canonical surjection $Z\to Z/X$.
By Lemma \ref{lem-push-out}, we have the following commutative diagram
\begin{align*}
\begin{tikzpicture}
\node(Sl)at(0,0){$S(v)$};
\node(Sr)at(4,0){$S(v)$};
\node(X)at(1,-1){$Z$};
\node(Y)at(5,-1){$Z/X$};
\node(Tl)at(0,-2){$T^2(v)$};
\node(Tr)at(4,-2){$T^2(v)$};
\node(P)at(1,-3){$Z^\p$};
\node(Q)at(5,-3){$Z^\p/X$};
\node(Yu)at(-3,-1){$X$};
\node(Yd)at(-3,-3){$X$};
\draw[thick, ->] (Sl)--(Sr) node[midway, above]{$\id$};
\draw[thick, ->] (Sl)--(X);
\draw[thick, ->] (Sr)--(Y);
\draw[thick, ->] (Sr)--(Tr);
\draw[thick, ->, preaction={draw=white,-,line width=4pt}] (X)--(Y);
\draw[thick, ->] (Sl)--(Tl);
\draw[thick, ->] (Tl)--(Tr) node[midway, above]{$\id$};
\draw[thick, ->, preaction={draw=white,-,line width=4pt}] (X)--(P) node[near start ,left]{$\beta$};
\draw[thick, ->] (Tl)--(P);
\draw[thick, ->] (Y)--(Q);
\draw[thick, ->] (P)--(Q);
\draw[thick, ->] (Tr)--(Q);
\draw[thick, ->, preaction={draw=white,-,line width=4pt}] (Yu)--(X);
\draw[thick, ->] (Yu)--(Yd) node[midway, left]{$\id$};
\draw[thick, ->] (Yd)--(P);
\end{tikzpicture}
\end{align*}
where horizontal sequences are short exact sequences.
Since $S(v) \to T^2(v)$ is injective, so is $Z\to Z^\p$, and cokernels of these morphisms are isomorphic to each other, which are isomorphic to $X^2(v)$.
Therefore, we have an short exact sequence $0 \to Z \to Z^\p \to X^2(v) \to 0$.

On the other hand, by Lemma \ref{lem-push-out}, we have the following commutative diagram
\begin{align*}
\begin{tikzpicture}
\node(Sl)at(-4,0){$S(v)$};
\node(Sr)at(0,0){$S(v)$};
\node(X)at(1,-1){$Z/X$};
\node(Y)at(5,-1){$Z/Y$};
\node(Tl)at(-4,-2){$T^2(v)$};
\node(Tr)at(0,-2){$T^2(v)$};
\node(P)at(1,-3){$Z^\p/X$};
\node(Q)at(5,-3){$Z/Y$};
\node(Yu)at(-3,-1){$Y/X$};
\node(Yd)at(-3,-3){$T_{\nabla}$};
\draw[thick, ->] (Sl)--(Sr) node[midway, above]{$\id$};
\draw[thick, ->] (Sl)--(Yu);
\draw[thick, ->] (Sr)--(X);
\draw[thick, ->] (X)--(Y);
\draw[thick, ->] (Sl)--(Tl);
\draw[thick, ->] (Sr)--(Tr);
\draw[thick, ->] (Tl)--(Tr) node[midway, above]{$\id$};
\draw[thick, ->] (X)--(P);
\draw[thick, ->] (Tl)--(Yd);
\draw[thick, ->] (Y)--(Q) node[midway, left]{$\id$};
\draw[thick, ->] (P)--(Q);
\draw[thick, ->] (Tr)--(P);
\draw[thick, ->, preaction={draw=white,-,line width=4pt}] (Yu)--(X);
\draw[thick, ->, preaction={draw=white,-,line width=4pt}] (Yu)--(Yd);
\draw[thick, ->] (Yd)--(P);
\end{tikzpicture}
\end{align*}
where horizontal sequences are short exact sequences.
Thus $T_{\nabla}$ is a submodule of $Z^\p/X$.
Let $Y^\p=\pi^{-1}(T_{\nabla})$, where $\pi \colon  Z^\p \to Z^\p/X$ is a canonical morphism.
Then we have $Z^\p/Y^\p \cong  Z/Y$ and $Y^\p/X \cong  T_{\nabla}$.
\end{proof}
We need the following fundamental observation.
For $i=1,2$, let $f_i\colon X_i \to Y$ and $g_i\colon X_i\to Z_i$ be morphisms of modules over a ring.
We have the following two push-out diagrams
\begin{align}\label{push-out-P}
\begin{tikzpicture}[baseline=-30]
\node(S)at(0,0){$X_1$};
\node(Y)at(2,0){$Y$};
\node(T)at(0,-2){$Z_1$};
\node(Q)at(2,-2){$P$};
\draw[thick, ->] (S)--(Y) node[midway, above]{$f_1$};
\draw[thick, ->] (S)--(T) node[midway, left]{$g_1$};
\draw[thick, ->] (Y)--(Q) node[midway, right]{$\beta_1$};
\draw[thick, ->] (T)--(Q) node[midway, above]{$\alpha_1$};
\end{tikzpicture}
\end{align}
\begin{align}\label{push-out-Q}
\begin{tikzpicture}[baseline=-30]
\node(S)at(0,0){$X_1\oplus X_2$};
\node(Y)at(3,0){$Y$};
\node(T)at(0,-2){$Z_1\oplus Z_2$};
\node(Q)at(3,-2){$Q$};
\draw[thick, ->] (S)--(Y) node[midway, above]{$(f_1\,\, f_2)$};
\draw[thick, ->] (S)--(T) node[midway, left]{$g_1\oplus g_2$};
\draw[thick, ->] (Y)--(Q) node[midway, right]{$\beta_2$};
\draw[thick, ->] (T)--(Q) node[midway, above]{$(\alpha_2^1 \,\,\alpha_2^2)$};
\end{tikzpicture}
\end{align}
Since $P$ is a push-out, there exists a morphism $\gamma\colon  P \to Q$ such that $\gamma\alpha_1=\alpha_2^1$ and $\gamma\beta_1=\beta_2$.
Thus we have the following commutative diagram
\begin{align}\label{push-out-PQ}
\begin{tikzpicture}[baseline=-30]
\node(S)at(0,0){$X_2$};
\node(Y)at(2,0){$P$};
\node(T)at(0,-2){$Z_2$};
\node(Q)at(2,-2){$Q$};
\draw[thick, ->] (S)--(Y) node[midway, above]{$\beta_1f_2$};
\draw[thick, ->] (S)--(T) node[midway, left]{$g_2$};
\draw[thick, ->] (Y)--(Q) node[midway, right]{$\gamma$};
\draw[thick, ->] (T)--(Q) node[midway, above]{$\alpha_2^2$};
\end{tikzpicture}
\end{align}
Then our claim is the following.
\begin{lem}\label{lem-push-out-ass}
Under the above notation, the diagram (\ref{push-out-PQ}) is a push-out diagram.
\end{lem}
\begin{proof}
Let $W$ be a module and $a\colon Z_2 \to W$ and $b\colon P\to W$ be two morphisms satisfying $ag_2=b\beta_1f_2$.
We show that there exists a unique morphism $c\colon Q\to W$ satisfying $c\alpha_2^2=a$ and $c\gamma=b$.

First we show the uniqueness.
Assume that there exists a morphism $c\colon Q\to W$ satisfying $c\alpha_2^2=a$ and $c\gamma=b$.
Then this $c$ satisfies $c\beta_2=b\beta_1$ and $c(\alpha_2^1\,\,\alpha_2^2)=(b\alpha_1\,\,a)$.
In fact, $c\alpha_2^2=a$ is clear, and $c\beta_2=c\gamma\beta_1=b\beta_1$ and $c\alpha_2^1=c\gamma\alpha_1=b\alpha_1$ hold.
By the uniqueness property of the push-out diagram (\ref{push-out-Q}), $c$ is a unique morphism under the given assumption.

Next we show the existence property.
We have an equality $(b\alpha_1\,\,a)(g_1\oplus g_2) = b\beta_1 (f_1\,\,f_2)$.
Since $(\ref{push-out-Q})$ is a push-out diagram, there exists a morphism $c\colon Q\to W$ satisfying $c(\alpha_2^1\,\,\alpha_2^2)=(b\alpha_1\,\,a)$ and  $c\beta_2=b\beta_1$, see the following diagram.
\begin{align*}
\begin{tikzpicture}[baseline=-30]
\node(S)at(0,0){$X_1\oplus X_2$};
\node(Y)at(3,0){$Y$};
\node(T)at(0,-2){$Z_1\oplus Z_2$};
\node(Q)at(3,-2){$Q$};
\node(W)at(4,-3){$W$};
\draw[thick, ->] (S)--(Y) node[midway, above]{$(f_1\,\, f_2)$};
\draw[thick, ->] (S)--(T) node[midway, left]{$g_1\oplus g_2$};
\draw[thick, ->] (Y)--(Q) node[midway, left]{$\beta_2$};
\draw[thick, ->] (T)--(Q) node[midway, above]{$(\alpha_2^1 \,\,\alpha_2^2)$};
\draw[thick, ->] (T)--(W) node[midway, xshift=-7mm,yshift=-3.5mm]{$(b\alpha_1\,\,a)$};
\draw[thick, ->] (Y)--(W) node[midway, right]{$b\beta_1$};
\draw[thick, ->] (Q)--(W) node[midway, above]{$c$};
\end{tikzpicture}
\end{align*}
Then we show that this $c$ satisfies $c\alpha_2^2=a$ and $c\gamma=b$.
By comparing the second term of $c(\alpha_2^1\,\,\alpha_2^2)=(b\alpha_1\,\,a)$, we have $c\alpha_2^2=a$.
We have that $c\gamma\alpha_1=c\alpha_2^1=b\alpha_1$ and $c\gamma\beta_1=c\beta_2=b\beta_1$.
Then by the uniqueness property of the push-out diagram (\ref{push-out-P}), we have $c\gamma=b$.
\end{proof}
For $i\in Q_0^1$, let $T^1(i)$ be an indecomposable direct summand of the characteristic tilting $A^1$-module.
Assume that $(T^1(i):\nabla^1(v))=\ell$, that is, $\nabla^1(v)$ appears $\ell$ times in a $\nabla^1$-filtration of $T^1(i)$.
We have the following sequence of submodules of $T^1(i)$
\begin{align*}
0 \subset M_{2\ell} \subset M_{2\ell-1} \subset M_{2\ell-2} \subset \dots \subset M_{2} \subset M_{1} \subset M_{0} = T^1(i)
\end{align*}
such that each $M_{j-1}/M_j$ is isomorphic to $\nabla^1(v)$ if $j$ is even, and belongs to $\mcF(\nabla^1(u)\mid u\neq v)$ if $j$ is odd.

For each $k=1,2,\dots,\ell$, there exists a non-zero morphism $S(v) \to M_{2k-1}/M_{2k}$ and this morphism lift to a morphism $f_{2k} \colon  S(v) \to T^1(i)$, since $S(v)$ is projective.
This implies that $\ell$ corresponds to the length of an $S(v)$-socle $\soc_v(T^1(i))$ of $T^1(i)$, that is, the maximal direct summand of $\soc(T^1(i))$ which is a direct sum of copies of $S(v)$.
For each $k=1,2,\dots,\ell$, we have a map $g_k=(f_2\,\,f_4\,\,\cdots f_{2k}) \colon  S(v)^{\oplus k} \to T^1(i)$.
We denote by $T^{(k)}(i)$ the module induced from the following push-out diagram
\begin{align}\label{push-out-hatT}
\begin{tikzpicture}[baseline=-30,xscale=1.2]
\node(S)at(0,0){$S(v)^{\oplus k}$};
\node(Y)at(2,0){$T^1(i)$};
\node(T)at(0,-2){$T^2(v)^{\oplus k}$};
\node(Q)at(2,-2){$T^{(k)}(i)$};
\draw[thick, ->] (S)--(Y) node[midway, above]{$g_k$};
\draw[thick, ->] (S)--(T);
\draw[thick, ->] (Y)--(Q) node[midway, right]{$\beta_k$};
\draw[thick, ->] (T)--(Q);
\end{tikzpicture}
\end{align}
where $\Delta^2(v)=S(v)\to T^2(v)$ is the injective morphism in Proposition \ref{pro-ch-tilt}.

Let $T^{(0)}(i)\coloneqq T^1(i)$.
Then by Lemma \ref{lem-push-out-ass}, $T^{(k+1)}(i)$ fits into the following push-out diagram
\begin{align*}
\begin{tikzpicture}[baseline=-30]
\node(S)at(0,0){$S(v)$};
\node(Y)at(3,0){$T^{(k)}(i)$};
\node(T)at(0,-2){$T^2(v)$};
\node(Q)at(3,-2){$T^{(k+1)}(i)$};
\draw[thick, ->] (S)--(Y) node[midway, above]{$\beta_kf_{2k+2}$};
\draw[thick, ->] (S)--(T);
\draw[thick, ->] (Y)--(Q) node[midway, right]{$\gamma_k$};
\draw[thick, ->] (T)--(Q);
\end{tikzpicture}
\end{align*}
and we have $\beta_k=\gamma_{k-1}\gamma_{k-2}\dots\gamma_{0}$.
\begin{pro}\label{pro-ch-tilt-push-out}
Under the above notation, $T^{(\ell)}(i)$ belongs to $\mcF(\Delta)\cap\mcF(\nabla)$.
\end{pro}
\begin{proof}
Since the cokernel of $\beta_{\ell}$ is isomorphic to $X^2(v)^{\oplus\ell}$, we have $T^{(\ell)}(i) \in \mcF(\Delta)$.

We apply Lemma \ref{lem-push-out-XYZ} to $0\subset M_{2}\subset M_{1} \subset M_{0}=T^{(0)}(i)$.
We have the following push-out diagram 
\begin{align*}
\begin{tikzpicture}[baseline=-30]
\node(S)at(0,0){$S(v)$};
\node(Y)at(2,0){$T^{(0)}(i)$};
\node(T)at(0,-2){$T^2(v)$};
\node(Q)at(2,-2){$T^{(1)}(i)$};
\draw[thick, ->] (S)--(Y) node[midway, above]{$f_{2}$};
\draw[thick, ->] (S)--(T);
\draw[thick, ->] (Y)--(Q) node[midway, right]{$\beta_1=\gamma_0$};
\draw[thick, ->] (T)--(Q); 
\end{tikzpicture}
\end{align*}
where $\beta_1$ is injective.
By Lemma \ref{lem-push-out-XYZ}, $M_2$ can be regarded as a submodule of $T^{(1)}(i)$ by the composite $M_2\subset T^{(0)}(i) \xto{\beta_1} T^{(1)}(i)$.
Then there exists a sequence of submodules $0 \subset M_{2} \subset M^\p_{1}\subset T^{(1)}(i)$ with $M^\p_{1}/M_{2}\cong  T_{\nabla}$, and $T^{(1)}(i)/M^\p_{1} \cong  M_{0}/M_{1}$.
Since $T_{\nabla}$ belongs to $\mcF(\nabla)$, $T^{(1)}(i)/M_{3}$ also belongs to $\mcF(\nabla)$.

For $2 \leq k \leq \ell$, assume that there exists a sequence of submodules
\begin{align*}
0 \subset M_{2\ell} \subset M_{2\ell-1} \subset \dots \subset M_{2k} \subset M_{2k-1} \subset T^{(k-1)}(i)
\end{align*}
such that $T^{(k-1)}(i)/M_{2k-1}\in\mcF(\nabla)$, where the inclusion is given by the composite $M_{2k-1} \subset T^1(i) \xto{\beta_{k-1}} T^{(k-1)}(i)$.
We apply Lemma \ref{lem-push-out-XYZ} to $0 \subset M_{2k} \subset M_{2k-1} \subset T^{(k-1)}(i)$.
We have the following push-out diagram 
\begin{align*}
\begin{tikzpicture}[baseline=-30,xscale=.8]
\node(S)at(0,0){$S(v)$};
\node(Y)at(4,0){$T^{(k-1)}(i)$};
\node(T)at(0,-2){$T^2(v)$};
\node(Q)at(4,-2){$T^{(k)}(i)$};
\draw[thick, ->] (S)--(Y) node[midway, above]{$\beta_{k-1}f_{2k}$};
\draw[thick, ->] (S)--(T);
\draw[thick, ->] (Y)--(Q) node[midway, right]{$\gamma_{k-1}$};
\draw[thick, ->] (T)--(Q); 
\end{tikzpicture}
\end{align*}
Then by the same argument as above, $M_{2k}$ can be regarded as a submodule of $T^{(k)}(i)$ by $\gamma_{k-1}\beta_{k-1}=\beta_{k}  \colon T^{(0)}(i) \to T^{(k)}(i)$, and $T^{(k)}(i)/M_{2k+1}$ belongs to $\mcF(\nabla)$.
Thus finally, we have $T^{(\ell)}(i)\in\mcF(\nabla)$.
\end{proof}
Since $Q^1\sqcup Q^2$ is a deconcatenation of $Q$ at a sink $v$, we have an isomorphism of algebras $A^1 \cong  e^1 A e^1$, where $e^1=\sum_{i\in Q^1_0}e_i$.
Thus we have an exact functor $\modu A \to \modu A^1$ defined by $M \mapsto Me^1$.
This restricts to functors $\mcF(\Delta) \to \mcF(\Delta^1)$ and $\mcF(\nabla) \to \mcF(\nabla^1)$, since the functor $M\mapsto Me^1$ is exact and $\nabla(v)e^1=\nabla^1(v)$ holds.
Note that by regarding $Me^1$ as an $A$-module by a canonical map $A \to A^1\cong  e^1Ae^1$, we have an injective morphism $Me^1 \to M$ of $A$-modules.
By a similar construction, that is, by using $e^2=\sum_{i\in Q_0^2}e_i$, we have the functor $\modu A \to \modu A^2$.
\begin{lem}\label{lem-push-out-appro}
In the diagram (\ref{push-out-hatT}) for $k=\ell$, 
\begin{align*}
\begin{tikzpicture}[baseline=-30]
\node(S)at(0,0){$S(v)^{\oplus \ell}$};
\node(Y)at(2,0){$T^1(i)$};
\node(T)at(0,-2){$T^2(v)^{\oplus \ell}$};
\node(Q)at(2,-2){$T^{(\ell)}(i)$};
\draw[thick, ->] (S)--(Y) node[midway, above]{$g_{\ell}$};
\draw[thick, ->] (S)--(T) node[midway, left]{$\iota$};
\draw[thick, ->] (Y)--(Q) node[midway, right]{$\beta_{\ell}$};
\draw[thick, ->] (T)--(Q);
\end{tikzpicture}
\end{align*}
the morphism $\beta_{\ell}$ is a right $\mcF(\nabla)$-approximation of $T^1(i)$.
\end{lem}
\begin{proof}
Let $M\in\mcF(\nabla)$ and $f\colon T^1(i) \to M$ be a morphism.
Since the image of $fg_{\ell}$ is a direct sum of copies of $S(v)$, the image is contained in $Me^2$.
Since $Me^2\in\mcF(\nabla^2)$ and $S(v)\to T^2(v)$ is a left $\mcF(\nabla^2)$-approximation of $S(v)$, there exists a morphism $h \colon  T^2(v)^{\oplus \ell} \to M$ such that $h\iota = fg_{\ell}$.
Because the diagram is a push-out diagram, there exists a morphism $f^\p \colon  T^{(\ell)}(i) \to M$ such that $f^{\p}\beta_{\ell}=f$.
\end{proof}
Recall that $\Delta(i)=\Delta^1(i)$ holds.
\begin{lem}\label{lem-delta-hatT-appro}
Let $f\colon \Delta(i)\to T^1(i)$ be a morphism obtained by Proposition \ref{pro-ch-tilt}.
Then the composite $\Delta(i) \xto{f} T^1(i) \xto{\beta_{\ell}} T^{(\ell)}(i)$ is a left $\mcF(\nabla)$-approximation of $\Delta(i)$.
\end{lem}
\begin{proof}
We denote by $\iota \colon  Me^1 \to M$ an inclusion map.
Let $M\in\mcF(\nabla)$ and $h \colon  \Delta(i) \to M$ be a morphism.
Since $\Delta^1(i)=\Delta(i)$ holds, there exists a morphism $\iota : \Delta(i) \to Me^1$ with $h=\iota h^\p$.
Since $f$ is a left $\mcF(\nabla^1)$-approximation of $\Delta(i)$ and $Me^1\in\mcF(\nabla^1)$, there exists a morphism $f^\p \colon  T^1(i) \to Me^1$ with $h^{\p}=f^{\p}f$.
By Lemma \ref{lem-push-out-appro}, there exists a morphism $f^{\p\p} \colon  T^{(\ell)}(i) \to M$ with $\iota f^\p=f^{\p\p}\beta_{\ell}$.
Then we have $h=\iota h^\p=\iota f^{\p}f = f^{\p\p}\beta_{\ell}$.
\[
\begin{tikzcd}[arrow style=tikz, arrows={thick},row sep=12mm]
& \Delta(i) \arrow[r, "f"] \arrow[dl, "h^{\p}"', densely dashed, end anchor={[xshift=-1.2ex]}] \arrow[d, "h" pos=.73] & T^1(i) \arrow[r, "\beta_{\ell}"] \arrow[dll, "{f^\p}" description, pos=.22, densely dashed, preaction={draw=white, -,line width=4pt}] & T^{(\ell)}(i) \arrow[dll, "f^{\p\p}", densely dashed] \\
Me^1 \arrow[r, "\iota"'] & M 
\end{tikzcd}
\]
\end{proof}
Let $R$ be a ring with a surjective ring morphism $R \to R/I$.
It is easy to see that if $R/I$ is a local ring and $I$ is contained in the radical of $R$, then $R$ is a local ring.
The following lemma implies that $T^{(\ell)}(i)$ is indecomposable.
\begin{lem}\label{lem-push-out-ind}
Let $A$ be a finite dimensional algebra and $S$ be a simple $A$-module.
Let $X,Y$ be $A$-modules such that the $S$-socles are $S^{\oplus m}$ and $S$, respectively, with $m\geq 1$.
Consider the following push-out diagram of $A$-modules
\[
\begin{tikzcd}[arrow style=tikz, arrows={thick}]
S^{\oplus m} \arrow[r, "a"] \arrow[d, "b"] & X \arrow[d, "\beta"] \\
Y^{\oplus m} \arrow[r, "\alpha"] & Z
\end{tikzcd}
\]
such that $a$ and $b$ are inclusion morphisms associated to $S$-socles.
Assume that the following conditions hold:
\begin{enumerate}
\item there exist no morphisms from $X$ to the cokernel of $b$,
\item there exist no morphisms from $Y$ to the cokernel of $a$,
\item $\beta$ is a left $(\add Z)$-approximation of $X$,
\item $X$ and $Y$ are indecomposable.
\end{enumerate}
Then $Z$ is indecomposable.
\end{lem}
\begin{proof}
We show that the endomorphism algebra of $Z$ is a local algebra.
By the assumption (1), we have a morphism of algebras $R_X \colon  \End_A(Z) \to \End_A(X)$ given by $R_X(\phi)= \phi|_{X}$.
By the assumption (3), this $R_X$ is surjective.
We have an exact sequence
\[
0 \to \Ker(R_X) \to \End_A(Z) \to \End_A(X) \to 0.
\]
By the assumption (4), $\End_A(X)$ is a local algebra.
Thus it is enough to show that $\Ker(R_X)$ is contained in the radical of $\End_A(Z)$.

Let $\phi$ be an endomorphism of $Z$ with $R_X(\phi)=0$.
We show that $\phi$ belongs to the radical of $\End_A(Z)$.
By the assumption (2), we have a morphism of algebras $R_{Y^{\oplus m}} \colon  \End_A(Z) \to \End_A(Y^{\oplus m})$ given by $R_{Y^{\oplus m}}(\phi)=\phi|_{Y^{\oplus m}}$.
The endomorphism $R_{Y^{\oplus m}}(\phi)$ is presented by a $(m\times m)$-matrix with entries in $\End_A(Y)$.
Since the $S$-socle of $Y$ is $S$ and $\phi|_{S^{\oplus m}}=(\phi|_X)|_{S^{\oplus m}}=0$, all entries of a matrix presentation of $R_{Y^{\oplus m}}(\phi)$ are not isomorphisms and so these morphisms belongs to the radical of a local algebra $\End_A(Y)$.
This implies that $R_{Y^{\oplus m}}(\phi)$ belongs to the radical of $\End_A(Y^{\oplus m})$.
In particular, $R_{Y^{\oplus m}}(\phi)$ is nilpotent.
Let $f\in\End_A(Z)$ be any morphism. Then $(f\phi)|_{S^{\oplus m}}=0$ holds since $R_X(\phi)=0$.
Thus by the same argument, $R_{Y^{\oplus m}}(f\phi)$ is nilpotent.
By the universality of the push-out diagram, $f\phi$ is nilpotent for any $f\in \End_A(Z)$.
This implies that $\id_Z-f\phi$ is an isomorphism, and so $\phi$ belongs to the radical of $\End_A(Z)$.
\end{proof}
Then we have the following result.
\begin{thm}\label{thm-dec-ch-tilt}
For $i\in Q_0^1$, let $\ell$ be the length of an $S(v)$-socle of $T^1(i)$.
Then the push-out $T^{(\ell)}(i)$ of $T^2(v)^{\oplus \ell} \leftarrow S(v)^{\oplus\ell} \to T^1(i)$ is isomorphic to $T(i)$.
\end{thm}
\begin{proof}
By Lemmas \ref{lem-push-out-appro}, \ref{lem-delta-hatT-appro} and \ref{lem-push-out-ind}, $T^{(\ell)}(i)$ is indecomposable.
By Lemma \ref{lem-delta-hatT-appro}, $T^{(\ell)}(i)$ gives a left $\mcF(\nabla)$-approximation of $\Delta(i)$.
Therefore, $T^{(\ell)}(i)$ is isomorphic to $T(i)$.
\end{proof}
\begin{cor}\label{cor-dec-ch-tilt}
Using the same notion as Theorem \ref{thm-dec-ch-tilt} and let $e^j=\sum_{i\in Q^j_0}e_i$ for $j=1,2$
Then we have $T(i)e^1 \cong  T^1(i)$ and $T(i)e^2 \cong  T^2(v)^{\oplus\ell}$.
\end{cor}
\begin{proof}
There exists a short exact sequence $0 \to S(v)^{\oplus\ell}\to T^2(v)^{\oplus \ell}\oplus T^1(i) \to T(i) \to 0$.
By multiplying $e^1$ or $e^2$, we have the assertion.
\end{proof}
\fi 
\section{Path algebras of type \texorpdfstring{$\A$}{TEXT}}\label{section-An}
\subsection{Path algebra of an equioriented quiver of type $\A$}

Let $A_n = 1\to 2 \to \dots \to n$ be an equioriented quiver of type $\A$. It was first noticed by Gabriel that tilting modules over $\Lambda_{n}\coloneqq \kk A_n $ are counted by the $n$-th Catalan number $c_n = \frac{1}{n+1} { \binom{2n}{n} }$ \cite{G81}. In this subsection we show that the number of different quasi-hereditary structures on the path algebra $\Lambda_n$ coincides with $c_n$. 

Recall that \emph{binary trees} can be defined inductively as follows. A \emph{binary tree} $T$ is either the empty set or a tuple $(r,L,R)$ where $r$ is a singleton set, called the root of $T$, and $L$ and $R$ are two binary trees. The empty set has no vertex but has one leaf. The set of leaves of $T = (r,L,R)$ is the disjoint union of the set of leaves of $L$ and $R$. The \emph{size} of the tree is its number of vertices (equivalently the number of leaves minus $1$). It is classical that binary trees are counted by the Catalan numbers. 

A  \emph{binary search tree} is a binary tree labeled by integers such that if a vertex $x$ is labeled by $k$, then the vertices of the left subtree (resp. right subtree) of $x$ are labeled by integers less than (resp. superior to) $k$.
If $T$ is a binary tree with $n$ vertices, there is a unique labeling of the vertices by each of the integers $1,2,\dots, n$ that makes it a binary search tree.
This procedure is sometimes called the \emph{in-order traversal} of the tree or simply as the in-order algorithm (recursively visit left subtree, root and right subtree).
The first vertex visited by the algorithm is labeled by $1$, the second by $2$ and so on, see  \cref{fig:bin-search-tree}.

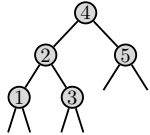
\begin{figure}[h]
\centering
\scalebox{.7}{\begin{tikzpicture}
  [ level distance=8mm,
   level 1/.style={sibling distance=15mm},
   level 2/.style={sibling distance=10mm},
   level 3/.style={sibling distance=5mm},
   inner/.style={circle,draw=black,fill=black!15,inner sep=0pt,minimum size=4mm,line width=1pt},
   leaf/.style={},                      
   edge from parent/.style={draw,line width=1pt}]
  \node [inner] {4}
     child {node [inner] {2}
       child {node [inner] {1}
         child {node [leaf] {}}
         child {node [leaf] {}}
       }
       child {node [inner] {3}
         child {node [leaf] {}}
         child {node [leaf] {}}
       }
     }
     child {node [inner] {5}
       child {node [leaf] {}
		}
       child {node [leaf] {}}
     };
\end{tikzpicture}}
\caption{Binary search tree of size 5 labeled by the in-order algorithm.}
\label{fig:bin-search-tree}
\end{figure}

Let $T$ be a binary tree of size $n$ viewed as a binary search tree. Then $T$ induces a poset $\lhd_T$ on $\{1,2,\dots,n\}$ by setting $i\lhd_T j$ if $i$ labels a vertex in the subtree of the vertex labeled by $j$.
For example, $\lhd_T$ of the above binary search tree is the transitive closure of $\{1 \lhd_T 2,\, 3 \lhd_T 2,\, 2 \lhd_T 4,\, 5 \lhd_T 4\}$.

\begin{pro}\label{char_tree_poset}
Let $\lhd$ be a partial order on $\{1,2,\dots, n\}$. Then there is a binary tree $T$ such that $\lhd = \lhd_T$ if and only if the following two conditions hold.
\begin{enumerate}
    \item For every $i<j$ incomparable with respect to $\lhd$, there exists $k$ such that $i<k<j$ and $i\lhd k$ and $j\lhd k$.
    \item For every $i<j<k$, if $i\lhd k$ then $j\lhd k$ and if $k\lhd i$ then $j\lhd i$. 
\end{enumerate}
\end{pro}
\begin{proof}
See \cite[Proposition 2.21]{CPP19}.
\end{proof}
\begin{rem}
Condition $(1)$ is equivalent to the following weaker condition: for every $i<j$ incomparable there exists $k$ such that $i<k<j$ and $i\lhd k$ or $j\lhd k$.
\end{rem}

In the proof of the following lemma, we denote by $[i, j]$ the interval in $\{1, 2, \dots, n\}$ in numerical order $\leq$, that is $[i, j]=\{k \in \mathbb{Z} \mid i \leq k \leq j\}$ for $i, j \in\{1, 2, \dots, n\}$.

\begin{lem}\label{lem:binary-trees-adapted}
\begin{enumerate}
\item Let $T$ be a binary tree of size $n$. Then $\lhd_T$ is an adapted poset to $\Lambda_n$.
\item If $\lhd$ is an adapted poset to $\Lambda_n$, then there is a binary tree $T$ such that $\lhd \sim \lhd_T$.
\end{enumerate}
\end{lem}
\begin{proof}
We first show (1).
The indecomposable $\Lambda_n$-modules can be identified with usual intervals in $\{1,2,\dots,n\}$. Then the assertion follows from Proposition \ref{char_tree_poset}. 

We next show (2).
Let $\lhd$ be an adapted poset to $\Lambda_n$ and let $\lhd_m$ be the minimal adapted poset equivalent to $\lhd$. Then by Proposition \ref{pro-mini-adapted}, we have $\lhd_m = (\Dec(\lhd)\cup \Inc(\lhd))^{tc}$. 
We may assume that $\lhd=\lhd_m$.
The relation $j\lhd i$ is \emph{decreasing} if $[\Delta(i):S(j)]\neq 0$.
Since $P(i)$ has basis the set of paths starting at $i$, this implies that there is a path from $i$ to $j$. Our choice of orientation implies that $i\leq j$.
In addition for every $i\leq k \leq j$, the simple module $S(k)$ is a composition factor of $\Delta(i)$, so $k\lhd i \in \Dec(\lhd)$.
This proves that the decreasing relations, and by a similar argument the increasing relations, satisfy the assertions $(2)$ of Proposition \ref{char_tree_poset}.
Conversely a relation $j\lhd i$, such that $i\leq j$ and any $k\in [i,j]$ satisfies $k\lhd i$, is decreasing.
Analogously, a relation $j \lhd i$ is increasing if and only if $j \leq i$ and any $k \in [j, i]$ satisfies $k \lhd i$.
It is then easy to see that the relation obtained as transitivity of two decreasing (increasing) relations is also decreasing (increasing).
Moreover if $i \lhd j \in \Inc(\lhd)$ and $j\lhd k \in \Dec(\lhd)$ with $k\neq j$, then we have $k<i$, because if $k\in [i,j]$, then $k\lhd j$ which contradicts the antisymmetry of $\lhd$.
For $x\in [k,i]$ we have $x\lhd k$, so $i\lhd k$ is decreasing.
With a similar argument, the relation obtained as transitivity of a decreasing relation and an increasing one is increasing.
This shows that $\lhd_m = \Dec(\lhd)\cup \Inc(\lhd)$ and that every relation satisfies (2) of Proposition \ref{char_tree_poset}.
Because $\lhd_m$ is adapted it satisfies $(1)$, so there is a binary tree $T$ such that $\lhd_m = \lhd_T$.
\end{proof}

\begin{pro}\label{prop:equiv-adapted-posets-binary-trees}
Let $n\in \mathbb{N}$.
The map sending a binary tree $T$ to the equivalence class of the adapted poset $\lhd_T$ is a bijection between the set of binary trees of size $n$ and the set of equivalence classes of adapted posets for $\Lambda_n$. 
\end{pro}
\begin{proof}
We already know that this map is surjective, we need to see that it is injective. For that we explain how we can recover the tree for the set of standard and costandard modules.

Let $T$ be a binary tree. Then $\lhd_T$ is an adapted poset to $\Lambda_n$ by Lemma \ref{lem:binary-trees-adapted} and $(\Lambda_n,\lhd_T)$ is a quasi-hereditary algebra. It is easy to see that the composition factors of the $\Delta(i)$ are indexed by the elements in the left subtree of $i$ and the composition factors of $\nabla(i)$ are indexed the elements of the right subtree. It follows that two different trees induce two non-equivalent posets. 
\end{proof}
\begin{rem}
The surjectivity of the map in Proposition \ref{prop:equiv-adapted-posets-binary-trees} can be found in the proof of \cite[Proposition 2.44]{CPP19}. For the convenience of the reader we sketch two constructions of the binary tree associated to a \emph{minimal} adapted poset.  It has a greatest element $m$ and each element covers at most two elements. If $x$ is covered by $y$, and $x<y$ for the usual ordering of the integers, then $x$ is a left child of $y$ and it is a right child otherwise. Hence starting with the maximal element and going down in the Hasse diagram of the poset we construct the desired binary seach tree.  

Alternatively, the minimal adapted orders for $\Lambda_n$ are particular cases of \emph{interval posets} in the sense of \cite{CP15} (in fact they are examples of \emph{exceptional} interval posets in the sense of \cite{rognerud20}). Hence, we can use the bijection of \cite[Theorem 2.8]{CP15} which gives a nice algorithm to reconstruct the tree starting only from the increasing (or decreasing) relations of the minimal poset. The construction is purely combinatorial: the Hasse diagram of the poset of increasing relations is a planar forest that we can transform into a binary tree using the so-called Knuth correspondence. We refer to \cite{CP15} for more details. 
\end{rem}
\begin{lem}\label{lem:minimal-adapted-for-An}
Let $ T $ be a binary tree of size $ n $. Then $ \lhd_{T} $ is a minimal adapted order to $ \Lambda_{n} $.
\end{lem}
\begin{proof}
Let $ \lhd' $ be an adapted poset to $ \Lambda_{n} $ such that $ \lhd_{T}\sim \lhd' $. Then there exists a binary tree $ T' $ such that $ \lhd'\sim \lhd_{T'} $, by \cref{lem:binary-trees-adapted}. Thus $ T=T' $ by \cref{prop:equiv-adapted-posets-binary-trees}, and the proof of \cref{lem:binary-trees-adapted} (2) shows that $ \lhd_{T} $ is extended by $ \lhd' $, which shows the claim.
\end{proof}

If $\lhd$ is an adapted order to $\Lambda_n$, the pair $(\Lambda_n,\lhd)$ is a quasi-hereditary algebra so it has a characteristic tilting module $T$ which is characterized by $\add(T) = \mathcal{F}(\Delta)\cap \mathcal{F}(\nabla)$. Since $\Lambda_n$ is a hereditary algebra, the module $T$ is a tilting module. Moreover, the tilting module only depends on the equivalence class of the partial order. So we have a map $\operatorname{char}$ from the set of equivalence classes of adapted partial orders to the set of isomorphism classes of tilting modules for $\Lambda_n$ which sends the equivalence class of $\lhd$ to the characteristic tilting module of $(\Lambda_n,\lhd)$.

\begin{thm}\label{thm_Baptiste_bijection}
We have a commutative diagram of bijections
\[
\xymatrix@C=-1mm@R=1cm{
&\{ \hbox{Binary trees of size n} \}\ar[dl]_{T\mapsto \lhd_T}\ar[dr]^{\phi} & \\
\{\hbox{Adapted partial orders to } \Lambda_n \}/\sim \ar[rr]^{\operatorname{char}}&& \{ \hbox{Tilting modules over } \Lambda_n \}/\cong
}
\]
where $\phi$ is the classical bijection between binary trees and tilting modules for $\Lambda_n$.
\end{thm}
\begin{proof}
In the proof of \cref{prop:equiv-adapted-posets-binary-trees} we determined the set of standard and costandard modules from the binary tree $T'$. We claim that the indecomposable direct summand $T(i)$ of the characteristic tilting module $T$ is the indecomposable module with composition factors indexed by the interval consisting of $i$ and the label of its subtrees (left and right). Since the map $\phi$ sends $T'$ to the module constructed in this way (see \cite[Section~9]{H06}), the proof follows from this claim. 

We denote by $M(i)$ the indecomposable module with composition factors indexed by the interval consisting of $i$ and the label of its subtrees (left and right).
By induction on the size of the subtrees we show that the module $M(i)$ is in $\mathcal{F}(\Delta)\cap \mathcal{F}(\nabla)$. This is clear for the subtrees of size one since in this case $T(i)=S(i)=\Delta(i) = \nabla(i)$.

In the general case, if $i_l$ (resp. $i_r$) denotes the left (resp. right) child of $i$ we have two exact sequences
\[
0 \to \Delta(i) \to M(i) \to M(i_l)\to 0
\]
and 
\[
0\to M(i_r) \to M(i) \to \nabla(i) \to 0.
\]
If $i$ has no left (resp. right) child then we let $M(i_l) = 0$ (resp. $M(i_r) = 0$) and we still have the two exact sequences.
By induction $M(i_l) \in \mathcal{F}(\Delta)$ and $M(i_r)\in \mathcal{F}(\nabla)$, so $M(i)\in \mathcal{F}(\Delta)\cap \mathcal{F}(\nabla)$. The result follows. 
\end{proof}

\begin{cor}\label{cor-An-Tamari}
Let $n\geq 1 $. Then $\qhstr(\Lambda_n)$ and the Tamari lattice of size $n$ are isomorphic as partially ordered sets.
We have that $|\qhstr(\Lambda_n)| = c_n = \frac{1}{n+1} { \binom{2n}{n} }$ is the Catalan number.
\end{cor}
\begin{proof}
As explained in the proof of Lemma \ref{lem:binary-trees-adapted} our notion of decreasing and increasing relations coincide with the one of \cite{CPP19}. The result is then a consequence of \cite[Proposition~41]{CPP19}. Note that $\operatorname{char}$ is a morphism of posets by Lemma \ref{lem-qhstr-poset}. Moreover, it is well known that $\phi$ is also a morphism of posets if the set of binary trees is endowed with the usual partial ordered induced by rotations. 
\end{proof}
\subsection{Path algebras of type \texorpdfstring{$\A$}{TEXT}: general case}
Let $Q$ be a quiver whose underlying graph is of type $\A_n$ and $I=Q_0$.
In this subsection, we classify all the quasi-hereditary structures on $\kk  Q$ and its characteristic tilting modules.

\begin{thm}\label{thm-A-dec}
Let  $Q^1\sqcup Q^2 \sqcup \dots \sqcup Q^{\ell}$ be an iterated deconcatenation of $ Q $ such that each $ Q^i $ is an equioriented quiver of type $ \A_{n_{i}} $ for some $ n_{i}\in \Z_{\geq 1} $. Then there is a bijection
\[
\qhstr(\kk Q) \longrightarrow  \prod_{i=1}^{\ell}\qhstr(\Lambda_{n_{i}})  \]
given by 
$[\lhd] \mapsto \big( [\lhd|_{Q_0^i} ] \big)_{i=1}^{\ell}$. Moreover, if $ \lhd $ is a minimal adapted order, then there exists a binary tree $ T_{i} $ of size $ n_i $  such that $ \lhd|_{Q_0^i}=\lhd_{T_{i}} $, for each $ 1\leq i\leq \ell $.
\end{thm}
\begin{proof}
The bijection follows from \cref{thm_qhstr_divide}. The second assertion is consequence of Lemmas \ref{lem-sink-source-minimal} and  \ref{lem:minimal-adapted-for-An} and \cref{prop:equiv-adapted-posets-binary-trees}.
\end{proof}
Since any indecomposable $\kk Q$-module is determined by its composition factor, the following theorem gives a complete construction of characteristic tilting modules from minimal adapted orders.
\begin{thm}\label{thm-A-ch-tilting}
Let $\lhd$ be a minimal adapted order to $\kk Q$ and $T=\bigoplus_{i\in I}T(i)$ be the characteristic tilting module associated to $\lhd$.
Then for vertices $i,j\in I$, $S(j)$ is a composition factor of $T(i)$ if and only if $ j\lhd i$ holds.
\end{thm}
\begin{proof}
The only if part follows from Proposition \ref{pro-ch-tilt}.
Assume that $Q$ has an iterated deconcatenation with $\ell$ components such that each quiver is an equioriented quiver of type $\A$.
We show the if part by an induction on $\ell$.
If $\ell=1$, then the assertion holds by the proof of Theorem \ref{thm_Baptiste_bijection}.
Assume that $\ell>1$.
We have a deconcatenation $Q^1 \sqcup Q^2$ of $Q$ at a vertex $v\in Q_0$ such that $Q^1$ is an equioriented quiver of type $\A$, and $Q^2$ has an iterated deconcatenation with $\ell-1$ components such that each quiver is an equioriented quiver of type $\A$.
By Lemma \ref{lem-sink-source-minimal}, the adapted orders $\lhd|_{Q_0^1}$ and $\lhd|_{Q_0^2}$ are also minimal.
If $v$ is a source, then apply Lemma \ref{lem-dual}, and we may assume that $v$ is a sink.
For $k=1,2$ and $i\in Q^k_0$, let $T^k(i)$ be the indecomposable direct summand of the characteristic tilting $\kk Q^k$-module associated to an adapted order $\lhd|_{Q^k_0}$.
Let $e^k=\sum_{u\in Q^k_0}e_u$ and $\bar{1}=2$, $\bar{2}=1$.
If $i\in Q_0^k$, then by Proposition \ref{pro-push-out-int} and Theorem \ref{thm-dec-ch-tilt} we have $T(i)e^k \cong  T^{\bar{k}}(v)$, $T(i)e^{\bar{k}} \cong  T^{\bar{k}}(v)$ if $[T(i):S(v)]\neq 0$ and $T(i)e^{\bar{k}}=0$ if $[T(i):S(v)]= 0$.
The assertion follows by induction on $\ell$ and Theorem \ref{thm_Baptiste_bijection}.
\end{proof}
We finish this section by giving a concrete example.
\begin{exa}
 Let $ A=\kk  Q $, where $Q=1 \to 2 \leftarrow 3 \leftarrow 4 \to 5$ as before. Then we have that
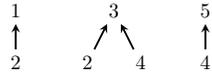
\[\scalebox{.7}{\begin{tikzpicture}
  \node   (a) at (0,0)  {2};
  \node   (b) at (0,1)  {1};
  \draw[->,line width =1pt,>=stealth] (a) -- (b);
\end{tikzpicture} \qquad
\begin{tikzpicture}
  \node   (a) at (0,0)  {2};
  \node   (b) at (.5,1)  {3};
  \node   (c) at (1,0)  {4};
  \draw[->,line width =1pt,>=stealth] (a) -- (b);
  \draw[->,line width =1pt,>=stealth] (c)--(b);
\end{tikzpicture}\qquad
\begin{tikzpicture}
  \node   (a) at (0,0)  {4};
  \node   (b) at (0,1)  {5};
  \draw[->,line width =1pt,>=stealth] (a) -- (b);
\end{tikzpicture} 
}\] are the Hasse diagrams of some minimal adapted posets to $ \kk(1\to 2) $, $\kk( 2\leftarrow 3 \leftarrow 4) $ and $ \kk(4\to 5) $, respectively. Then, by \cref{lem-sink-source-minimal} we have that the concatenation of the last Hasse diagrams \[\scalebox{.7}
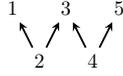
{\begin{tikzpicture}
  \node   (a) at (0,1)  {1};
  \node   (b) at (.5,0)  {2};
  \node   (c) at (1,1)  {3};
  \node   (d) at (1.5,0)  {4};
  \node   (e) at (2,1)  {5};
  
  \draw[->,line width =1pt,>=stealth] (b) -- (a);
  \draw[->,line width =1pt,>=stealth] (b)--(c);
  \draw[->,line width =1pt,>=stealth] (d)--(c);
  \draw[->,line width =1pt,>=stealth] (d)--(e);
\end{tikzpicture}}\] is the corresponding minimal adapted poset to $ A $. In this case $ |\qhstr A|=20 $. In \cref{fig:ex-poset-qhstr-A} we depict the Hasse diagram of $ \qhstr(A) $. The vertices correspond to minimal adapted orders to $ A $ which represent all the quasi-hereditary structures on $ A $. Note that if a total order is a minimal adapted order, then it is the unique element in its equivalence class.
\newcommand\myscalebox{1.2}
\tikzstyle{poset} = [line join=bevel, xscale=.5, yscale=.4, inner sep=.3mm]
\tikzstyle{posetline} = [black,->, >=stealth, line width=.6]

\newcommand{\nueve}{
\scalebox{\myscalebox}{\begin{tikzpicture}[poset]
\node (node_4) at (66.0bp,55.5bp) [draw,draw=none] {$5$};
  \node (node_3) at (51.0bp,6.5bp) [draw,draw=none] {$4$};
  \node (node_2) at (36.0bp,55.5bp) [draw,draw=none] {$3$};
  \node (node_1) at (21.0bp,104.5bp) [draw,draw=none] {$2$};
  \node (node_0) at (6.0bp,55.5bp) [draw,draw=none] {$1$};
  \draw [posetline] (node_2) ..controls (31.877bp,68.969bp) and (28.698bp,79.353bp)  .. (node_1);
  \draw [posetline] (node_3) ..controls (46.877bp,19.969bp) and (43.698bp,30.353bp)  .. (node_2);
  \draw [posetline] (node_0) ..controls (10.123bp,68.969bp) and (13.302bp,79.353bp)  .. (node_1);
  \draw [posetline] (node_3) ..controls (55.123bp,19.969bp) and (58.302bp,30.353bp)  .. (node_4);
\end{tikzpicture}}}

\newcommand{\unopos}{
\scalebox{\myscalebox}{\begin{tikzpicture}[poset]
\node (node_4) at (6.0bp,202.5bp) [draw,draw=none] {$5$};
  \node (node_3) at (6.0bp,153.5bp) [draw,draw=none] {$4$};
  \node (node_2) at (6.0bp,104.5bp) [draw,draw=none] {$3$};
  \node (node_1) at (6.0bp,55.5bp) [draw,draw=none] {$2$};
  \node (node_0) at (6.0bp,6.5bp) [draw,draw=none] {$1$};
  \draw [posetline] (node_3) ..controls (6.0bp,166.82bp) and (6.0bp,176.9bp)  .. (node_4);
  \draw [posetline] (node_0) ..controls (6.0bp,19.822bp) and (6.0bp,29.898bp)  .. (node_1);
  \draw [posetline] (node_2) ..controls (6.0bp,117.82bp) and (6.0bp,127.9bp)  .. (node_3);
  \draw [posetline] (node_1) ..controls (6.0bp,68.822bp) and (6.0bp,78.898bp)  .. (node_2);
\end{tikzpicture}}}

\newcommand{\dos}{
\scalebox{\myscalebox}{\begin{tikzpicture}[poset]
\node (node_4) at (36.0bp,104.5bp) [draw,draw=none] {$5$};
  \node (node_3) at (21.0bp,153.5bp) [draw,draw=none] {$4$};
  \node (node_2) at (6.0bp,104.5bp) [draw,draw=none] {$3$};
  \node (node_1) at (6.0bp,55.5bp) [draw,draw=none] {$2$};
  \node (node_0) at (6.0bp,6.5bp) [draw,draw=none] {$1$};
  \draw [posetline] (node_0) ..controls (6.0bp,19.822bp) and (6.0bp,29.898bp)  .. (node_1);
  \draw [posetline] (node_4) ..controls (31.877bp,117.97bp) and (28.698bp,128.35bp)  .. (node_3);
  \draw [posetline] (node_2) ..controls (10.123bp,117.97bp) and (13.302bp,128.35bp)  .. (node_3);
  \draw [posetline] (node_1) ..controls (6.0bp,68.822bp) and (6.0bp,78.898bp)  .. (node_2);
\end{tikzpicture}}}

\newcommand{\tres}{
\scalebox{\myscalebox}{\begin{tikzpicture}[poset]
\node (node_4) at (36.0bp,104.5bp) [draw,draw=none] {$5$};
  \node (node_3) at (36.0bp,55.5bp) [draw,draw=none] {$4$};
  \node (node_2) at (6.0bp,104.5bp) [draw,draw=none] {$3$};
  \node (node_1) at (6.0bp,55.5bp) [draw,draw=none] {$2$};
  \node (node_0) at (6.0bp,6.5bp) [draw,draw=none] {$1$};
  \draw [posetline] (node_3) ..controls (27.619bp,69.188bp) and (20.975bp,80.042bp)  .. (node_2);
  \draw [posetline] (node_0) ..controls (6.0bp,19.822bp) and (6.0bp,29.898bp)  .. (node_1);
  \draw [posetline] (node_3) ..controls (36.0bp,68.822bp) and (36.0bp,78.898bp)  .. (node_4);
  \draw [posetline] (node_1) ..controls (6.0bp,68.822bp) and (6.0bp,78.898bp)  .. (node_2);
\end{tikzpicture}}}

\newcommand{\cuatro}{
\scalebox{\myscalebox}{\begin{tikzpicture}[poset]
\node (node_4) at (36.0bp,6.5bp) [draw,draw=none] {$5$};
  \node (node_3) at (36.0bp,55.5bp) [draw,draw=none] {$4$};
  \node (node_2) at (21.0bp,104.5bp) [draw,draw=none] {$3$};
  \node (node_1) at (6.0bp,55.5bp) [draw,draw=none] {$2$};
  \node (node_0) at (6.0bp,6.5bp) [draw,draw=none] {$1$};
  \draw [posetline] (node_3) ..controls (31.877bp,68.969bp) and (28.698bp,79.353bp)  .. (node_2);
  \draw [posetline] (node_0) ..controls (6.0bp,19.822bp) and (6.0bp,29.898bp)  .. (node_1);
  \draw [posetline] (node_4) ..controls (36.0bp,19.822bp) and (36.0bp,29.898bp)  .. (node_3);
  \draw [posetline] (node_1) ..controls (10.123bp,68.969bp) and (13.302bp,79.353bp)  .. (node_2);
\end{tikzpicture}}}

\newcommand{\cinco}{
\scalebox{\myscalebox}{\begin{tikzpicture}[poset]
\node (node_4) at (21.0bp,153.5bp) [draw,draw=none] {$5$};
  \node (node_3) at (21.0bp,104.5bp) [draw,draw=none] {$4$};
  \node (node_2) at (36.0bp,6.5bp) [draw,draw=none] {$3$};
  \node (node_1) at (21.0bp,55.5bp) [draw,draw=none] {$2$};
  \node (node_0) at (6.0bp,6.5bp) [draw,draw=none] {$1$};
  \draw [posetline] (node_2) ..controls (31.877bp,19.969bp) and (28.698bp,30.353bp)  .. (node_1);
  \draw [posetline] (node_1) ..controls (21.0bp,68.822bp) and (21.0bp,78.898bp)  .. (node_3);
  \draw [posetline] (node_0) ..controls (10.123bp,19.969bp) and (13.302bp,30.353bp)  .. (node_1);
  \draw [posetline] (node_3) ..controls (21.0bp,117.82bp) and (21.0bp,127.9bp)  .. (node_4);
\end{tikzpicture}}}

\newcommand{\seis}{
\scalebox{\myscalebox}{\begin{tikzpicture}[poset]
\node (node_4) at (51.0bp,55.5bp) [draw,draw=none] {$5$};
  \node (node_3) at (36.0bp,104.5bp) [draw,draw=none] {$4$};
  \node (node_2) at (36.0bp,6.5bp) [draw,draw=none] {$3$};
  \node (node_1) at (21.0bp,55.5bp) [draw,draw=none] {$2$};
  \node (node_0) at (6.0bp,6.5bp) [draw,draw=none] {$1$};
  \draw [posetline] (node_2) ..controls (31.877bp,19.969bp) and (28.698bp,30.353bp)  .. (node_1);
  \draw [posetline] (node_1) ..controls (25.123bp,68.969bp) and (28.302bp,79.353bp)  .. (node_3);
  \draw [posetline] (node_0) ..controls (10.123bp,19.969bp) and (13.302bp,30.353bp)  .. (node_1);
  \draw [posetline] (node_4) ..controls (46.877bp,68.969bp) and (43.698bp,79.353bp)  .. (node_3);
\end{tikzpicture}}}

\newcommand{\siete}{
\scalebox{\myscalebox}{\begin{tikzpicture}[poset]
\node (node_4) at (36.0bp,104.5bp) [draw,draw=none] {$5$};
  \node (node_3) at (36.0bp,55.5bp) [draw,draw=none] {$4$};
  \node (node_2) at (36.0bp,6.5bp) [draw,draw=none] {$3$};
  \node (node_1) at (6.0bp,104.5bp) [draw,draw=none] {$2$};
  \node (node_0) at (6.0bp,55.5bp) [draw,draw=none] {$1$};
  \draw [posetline] (node_0) ..controls (6.0bp,68.822bp) and (6.0bp,78.898bp)  .. (node_1);
  \draw [posetline] (node_2) ..controls (36.0bp,19.822bp) and (36.0bp,29.898bp)  .. (node_3);
  \draw [posetline] (node_3) ..controls (36.0bp,68.822bp) and (36.0bp,78.898bp)  .. (node_4);
  \draw [posetline] (node_3) ..controls (27.619bp,69.188bp) and (20.975bp,80.042bp)  .. (node_1);
\end{tikzpicture}}}

\newcommand{\ocho}{
\scalebox{\myscalebox}{\begin{tikzpicture}[poset]
\node (node_4) at (51.0bp,6.5bp) [draw,draw=none] {$5$};
  \node (node_3) at (36.0bp,55.5bp) [draw,draw=none] {$4$};
  \node (node_2) at (21.0bp,6.5bp) [draw,draw=none] {$3$};
  \node (node_1) at (21.0bp,104.5bp) [draw,draw=none] {$2$};
  \node (node_0) at (6.0bp,55.5bp) [draw,draw=none] {$1$};
  \draw [posetline] (node_0) ..controls (10.123bp,68.969bp) and (13.302bp,79.353bp)  .. (node_1);
  \draw [posetline] (node_4) ..controls (46.877bp,19.969bp) and (43.698bp,30.353bp)  .. (node_3);
  \draw [posetline] (node_2) ..controls (25.123bp,19.969bp) and (28.302bp,30.353bp)  .. (node_3);
  \draw [posetline] (node_3) ..controls (31.877bp,68.969bp) and (28.698bp,79.353bp)  .. (node_1);
\end{tikzpicture}}}

\newcommand{\diez}{
\scalebox{\myscalebox}{\begin{tikzpicture}[poset]
\node (node_4) at (36.0bp,6.5bp) [draw,draw=none] {$5$};
  \node (node_3) at (36.0bp,55.5bp) [draw,draw=none] {$4$};
  \node (node_2) at (36.0bp,104.5bp) [draw,draw=none] {$3$};
  \node (node_1) at (21.0bp,153.5bp) [draw,draw=none] {$2$};
  \node (node_0) at (6.0bp,104.5bp) [draw,draw=none] {$1$};
  \draw [posetline] (node_2) ..controls (31.877bp,117.97bp) and (28.698bp,128.35bp)  .. (node_1);
  \draw [posetline] (node_3) ..controls (36.0bp,68.822bp) and (36.0bp,78.898bp)  .. (node_2);
  \draw [posetline] (node_0) ..controls (10.123bp,117.97bp) and (13.302bp,128.35bp)  .. (node_1);
  \draw [posetline] (node_4) ..controls (36.0bp,19.822bp) and (36.0bp,29.898bp)  .. (node_3);
\end{tikzpicture}}}

\newcommand{\once}{
\scalebox{\myscalebox}{\begin{tikzpicture}[poset]
\node (node_4) at (36.0bp,153.5bp) [draw,draw=none] {$5$};
  \node (node_3) at (36.0bp,104.5bp) [draw,draw=none] {$4$};
  \node (node_2) at (36.0bp,55.5bp) [draw,draw=none] {$3$};
  \node (node_1) at (21.0bp,6.5bp) [draw,draw=none] {$2$};
  \node (node_0) at (6.0bp,55.5bp) [draw,draw=none] {$1$};
  \draw [posetline] (node_2) ..controls (36.0bp,68.822bp) and (36.0bp,78.898bp)  .. (node_3);
  \draw [posetline] (node_1) ..controls (16.877bp,19.969bp) and (13.698bp,30.353bp)  .. (node_0);
  \draw [posetline] (node_3) ..controls (36.0bp,117.82bp) and (36.0bp,127.9bp)  .. (node_4);
  \draw [posetline] (node_1) ..controls (25.123bp,19.969bp) and (28.302bp,30.353bp)  .. (node_2);
\end{tikzpicture}}}

\newcommand{\doce}{
\scalebox{\myscalebox}{\begin{tikzpicture}[poset]
\node (node_4) at (66.0bp,55.5bp) [draw,draw=none] {$5$};
  \node (node_3) at (51.0bp,104.5bp) [draw,draw=none] {$4$};
  \node (node_2) at (36.0bp,55.5bp) [draw,draw=none] {$3$};
  \node (node_1) at (21.0bp,6.5bp) [draw,draw=none] {$2$};
  \node (node_0) at (6.0bp,55.5bp) [draw,draw=none] {$1$};
  \draw [posetline] (node_4) ..controls (61.877bp,68.969bp) and (58.698bp,79.353bp)  .. (node_3);
  \draw [posetline] (node_2) ..controls (40.123bp,68.969bp) and (43.302bp,79.353bp)  .. (node_3);
  \draw [posetline] (node_1) ..controls (16.877bp,19.969bp) and (13.698bp,30.353bp)  .. (node_0);
  \draw [posetline] (node_1) ..controls (25.123bp,19.969bp) and (28.302bp,30.353bp)  .. (node_2);
\end{tikzpicture}}}

\newcommand{\trece}{
\scalebox{\myscalebox}{\begin{tikzpicture}[poset]
\node (node_4) at (66.0bp,55.5bp) [draw,draw=none] {$5$};
  \node (node_3) at (58.0bp,6.5bp) [draw,draw=none] {$4$};
  \node (node_2) at (36.0bp,55.5bp) [draw,draw=none] {$3$};
  \node (node_1) at (13.0bp,6.5bp) [draw,draw=none] {$2$};
  \node (node_0) at (6.0bp,55.5bp) [draw,draw=none] {$1$};
  \draw [posetline] (node_3) ..controls (51.92bp,20.042bp) and (47.188bp,30.582bp)  .. (node_2);
  \draw [posetline] (node_1) ..controls (11.097bp,19.822bp) and (9.6574bp,29.898bp)  .. (node_0);
  \draw [posetline] (node_3) ..controls (60.187bp,19.896bp) and (61.857bp,30.125bp)  .. (node_4);
  \draw [posetline] (node_1) ..controls (19.356bp,20.042bp) and (24.304bp,30.582bp)  .. (node_2);
\end{tikzpicture}}}

\newcommand{\catorce}{
\scalebox{\myscalebox}{\begin{tikzpicture}[poset]
\node (node_4) at (36.0bp,6.5bp) [draw,draw=none] {$5$};
  \node (node_3) at (36.0bp,55.5bp) [draw,draw=none] {$4$};
  \node (node_2) at (36.0bp,104.5bp) [draw,draw=none] {$3$};
  \node (node_1) at (6.0bp,55.5bp) [draw,draw=none] {$2$};
  \node (node_0) at (6.0bp,104.5bp) [draw,draw=none] {$1$};
  \draw [posetline] (node_3) ..controls (36.0bp,68.822bp) and (36.0bp,78.898bp)  .. (node_2);
  \draw [posetline] (node_4) ..controls (36.0bp,19.822bp) and (36.0bp,29.898bp)  .. (node_3);
  \draw [posetline] (node_1) ..controls (6.0bp,68.822bp) and (6.0bp,78.898bp)  .. (node_0);
  \draw [posetline] (node_1) ..controls (14.381bp,69.188bp) and (21.025bp,80.042bp)  .. (node_2);
\end{tikzpicture}}}

\newcommand{\quince}{
\scalebox{\myscalebox}{\begin{tikzpicture}[poset]
\node (node_4) at (36.0bp,153.5bp) [draw,draw=none] {$5$};
  \node (node_3) at (36.0bp,104.5bp) [draw,draw=none] {$4$};
  \node (node_2) at (21.0bp,6.5bp) [draw,draw=none] {$3$};
  \node (node_1) at (21.0bp,55.5bp) [draw,draw=none] {$2$};
  \node (node_0) at (6.0bp,104.5bp) [draw,draw=none] {$1$};
  \draw [posetline] (node_2) ..controls (21.0bp,19.822bp) and (21.0bp,29.898bp)  .. (node_1);
  \draw [posetline] (node_1) ..controls (25.123bp,68.969bp) and (28.302bp,79.353bp)  .. (node_3);
  \draw [posetline] (node_1) ..controls (16.877bp,68.969bp) and (13.698bp,79.353bp)  .. (node_0);
  \draw [posetline] (node_3) ..controls (36.0bp,117.82bp) and (36.0bp,127.9bp)  .. (node_4);
\end{tikzpicture}}}

\newcommand{\dieciseis}{
\scalebox{\myscalebox}{\begin{tikzpicture}[poset]
\node (node_4) at (6.0bp,55.5bp) [draw,draw=none] {$5$};
  \node (node_3) at (6.0bp,104.5bp) [draw,draw=none] {$4$};
  \node (node_2) at (36.0bp,6.5bp) [draw,draw=none] {$3$};
  \node (node_1) at (36.0bp,55.5bp) [draw,draw=none] {$2$};
  \node (node_0) at (36.0bp,104.5bp) [draw,draw=none] {$1$};
  \draw [posetline] (node_2) ..controls (36.0bp,19.822bp) and (36.0bp,29.898bp)  .. (node_1);
  \draw [posetline] (node_1) ..controls (27.619bp,69.188bp) and (20.975bp,80.042bp)  .. (node_3);
  \draw [posetline] (node_4) ..controls (6.0bp,68.822bp) and (6.0bp,78.898bp)  .. (node_3);
  \draw [posetline] (node_1) ..controls (36.0bp,68.822bp) and (36.0bp,78.898bp)  .. (node_0);
\end{tikzpicture}}}

\newcommand{\diecisiete}{
\scalebox{\myscalebox}{\begin{tikzpicture}[poset]
\node (node_4) at (36.0bp,104.5bp) [draw,draw=none] {$5$};
  \node (node_3) at (21.0bp,55.5bp) [draw,draw=none] {$4$};
  \node (node_2) at (21.0bp,6.5bp) [draw,draw=none] {$3$};
  \node (node_1) at (6.0bp,104.5bp) [draw,draw=none] {$2$};
  \node (node_0) at (6.0bp,153.5bp) [draw,draw=none] {$1$};
  \draw [posetline] (node_2) ..controls (21.0bp,19.822bp) and (21.0bp,29.898bp)  .. (node_3);
  \draw [posetline] (node_1) ..controls (6.0bp,117.82bp) and (6.0bp,127.9bp)  .. (node_0);
  \draw [posetline] (node_3) ..controls (16.877bp,68.969bp) and (13.698bp,79.353bp)  .. (node_1);
  \draw [posetline] (node_3) ..controls (25.123bp,68.969bp) and (28.302bp,79.353bp)  .. (node_4);
\end{tikzpicture}}}

\newcommand{\dieciocho}{
\scalebox{\myscalebox}{\begin{tikzpicture}[poset]
\node (node_4) at (36.0bp,6.5bp) [draw,draw=none] {$5$};
  \node (node_3) at (21.0bp,55.5bp) [draw,draw=none] {$4$};
  \node (node_2) at (6.0bp,6.5bp) [draw,draw=none] {$3$};
  \node (node_1) at (21.0bp,104.5bp) [draw,draw=none] {$2$};
  \node (node_0) at (21.0bp,153.5bp) [draw,draw=none] {$1$};
  \draw [posetline] (node_4) ..controls (31.877bp,19.969bp) and (28.698bp,30.353bp)  .. (node_3);
  \draw [posetline] (node_2) ..controls (10.123bp,19.969bp) and (13.302bp,30.353bp)  .. (node_3);
  \draw [posetline] (node_1) ..controls (21.0bp,117.82bp) and (21.0bp,127.9bp)  .. (node_0);
  \draw [posetline] (node_3) ..controls (21.0bp,68.822bp) and (21.0bp,78.898bp)  .. (node_1);
\end{tikzpicture}}}

\newcommand{\diecinueve}{
\scalebox{\myscalebox}{\begin{tikzpicture}[poset]
\node (node_4) at (36.0bp,55.5bp) [draw,draw=none] {$5$};
  \node (node_3) at (21.0bp,6.5bp) [draw,draw=none] {$4$};
  \node (node_2) at (6.0bp,55.5bp) [draw,draw=none] {$3$};
  \node (node_1) at (6.0bp,104.5bp) [draw,draw=none] {$2$};
  \node (node_0) at (6.0bp,153.5bp) [draw,draw=none] {$1$};
  \draw [posetline] (node_2) ..controls (6.0bp,68.822bp) and (6.0bp,78.898bp)  .. (node_1);
  \draw [posetline] (node_3) ..controls (16.877bp,19.969bp) and (13.698bp,30.353bp)  .. (node_2);
  \draw [posetline] (node_1) ..controls (6.0bp,117.82bp) and (6.0bp,127.9bp)  .. (node_0);
  \draw [posetline] (node_3) ..controls (25.123bp,19.969bp) and (28.302bp,30.353bp)  .. (node_4);
\end{tikzpicture}}}

\newcommand{\veinte}{
\scalebox{\myscalebox}{\begin{tikzpicture}[poset]
\node (node_4) at (6.0bp,6.5bp) [draw,draw=none] {$5$};
  \node (node_3) at (6.0bp,55.5bp) [draw,draw=none] {$4$};
  \node (node_2) at (6.0bp,104.5bp) [draw,draw=none] {$3$};
  \node (node_1) at (6.0bp,153.5bp) [draw,draw=none] {$2$};
  \node (node_0) at (6.0bp,202.5bp) [draw,draw=none] {$1$};
  \draw [posetline] (node_2) ..controls (6.0bp,117.82bp) and (6.0bp,127.9bp)  .. (node_1);
  \draw [posetline] (node_3) ..controls (6.0bp,68.822bp) and (6.0bp,78.898bp)  .. (node_2);
  \draw [posetline] (node_4) ..controls (6.0bp,19.822bp) and (6.0bp,29.898bp)  .. (node_3);
  \draw [posetline] (node_1) ..controls (6.0bp,166.82bp) and (6.0bp,176.9bp)  .. (node_0);
\end{tikzpicture}}}
 \begin{figure}[h]
\centering
\scalebox{.7}{
\begin{tikzpicture}[line join=bevel, xscale=2.1,yscale=2.1, inner sep=0.6mm, dib/.style={draw,line width =.5pt}]
  \node (n15) at (218.0bp,157bp) [dib] {\quince};
  \node (n19) at (130.0bp,55.5bp) [dib] {\diecinueve};
  \node (n9) at (110.0bp,6.5bp) [dib] {\nueve}; 
  \node (n10) at (85.0bp,55.5bp) [dib] {\diez};
  \node (n8) at (140.0bp,104.5bp) [dib] {\ocho};
  \node (n14) at (6.0bp,153.5bp) [dib] {\catorce};
  \node (n6) at (170.0bp,153.5bp) [dib] {\seis};
  \node (n5) at (179.0bp,104.5bp) [dib] {\cinco};
  \node (n17) at (218.0bp,104.5bp) [dib] {\diecisiete}; 
  \node (n4) at (6.0bp,104.5bp) [dib] {\cuatro};
  \node (n3) at (42.0bp,55.5bp) [dib] {\tres};
  \node (n12) at (110.0bp,251.5bp) [dib] {\doce};
  \node (n11) at (110.0bp,202.5bp) [dib] {\once};
  \node (n18) at (115.0bp,155bp) [dib] {\dieciocho};
  \node (n20) at (98.0bp,104.5bp) [dib] {\veinte};
  \node (n2) at (55.0bp,202.5bp) [dib] {\dos};
  \node (n7) at (173.0bp,55.5bp) [dib] {\siete};
  \node (n1) at (75.0bp,148.5bp) [dib,fill=white] {\unopos};
  \node (n16) at (170.0bp,202.5bp) [dib] {\dieciseis};
  \node (n13) at (50.0bp,104.5bp) [dib] {\trece};
  
  \begin{scope}[cab/.style={-open triangle 45, 
  							preaction={draw=white,-,line width=4pt},
							line width=1pt,
							shorten >=1mm,shorten <=1mm}]
	\begin{pgfonlayer}{bg} 
	
	\draw [black,cab] (n9) -- (n3.south);
	\draw [black,cab] (n9) -- (n7.south);
	\draw [black,cab] (n9) -- (n19.south);
	\draw [black,cab] (n9) -- (n10.south);
	
	\draw [black,cab] (n3) -- (n4.south);
	\draw [black,cab] (n3.north) -- (n13);
	\draw [black,cab] ([yshift=-2mm]n3.north east) ..controls (75bp,90bp)  .. (n1.south);
	
	\draw [black,cab] ([yshift=-3mm]n10.north west) -- (n4);
	\draw [black,cab] (n10.north) -- ([yshift=-1.5mm]n20.west); 
	\draw [black,cab] ([yshift=-3mm]n10.north east) -- (n8);
	
	\draw [black,cab] (n19.west) -- ([yshift=2mm]n13.south east);
	\draw [black,cab] (n19.north) -- (n20);
	\draw [black,cab] ([yshift=-4mm]n19.north east) -- (n17.south west);
	
	\draw [black,cab] (n7) -- (n8);
	\draw [black,cab] (n7) -- (n5.south);
	\draw [black,cab] (n7) -- (n17.south);
	
	\draw [black,cab] (n4.north east) -- ([xshift=-1mm]n2.south);
	\draw [black,cab] (n4) -- (n14);

	\draw [black,cab] (n13) -- (n14.south east);
    \draw [black,cab] (n13.north) ..controls (60bp,170bp)  .. (n11.west);  
	
	\draw [black,cab] ([yshift=-1mm]n20.west) -- (n14);

	\draw [black,cab] ([xshift=-1mm]n8.north) -- (n18.south east);
	\draw [black,cab] ([xshift=1mm]n8.north) -- ([yshift=4mm]n6.south west);
	
	 \draw [black,cab] ([yshift=4mm]n5.west) -- ([yshift=-3mm]n1.east);
	\draw [black,cab] (n5.north) -- (n6.south);
	\draw [black,cab] (n5) -- (n15);
	
	\draw [black,cab] ([yshift=-2.5mm]n17.north west) -- ([yshift=-2mm]n18.east);
	\draw [black,cab] (n17) -- (n15);	
		
	\draw [black,cab] (n14.north) ..controls (24bp,217bp)  .. (n12.west);
	
	\draw [black,cab] (n1) -- ([xshift=1mm]n2.south);
	\draw [black,cab] (n1) -- (n11);
	
	\draw [black,cab] (n18) -- (n16);
	
	\draw [black,cab] (n6.north west) -- ([yshift=-5mm]n2.east);
	\draw [black,cab] (n6) -- (n16);
	
	\draw [black,cab] (n15) -- (n11);
	\draw [black,cab] (n15.north) -- (n16);
	
	\draw [black,cab] (n2.north) -- (n12);
	\draw [black,cab] (n11) -- (n12);
	\draw [black,cab] (n16.north) -- (n12);
	
	\draw [black,cab] (n20.north) -- ([yshift=-1mm]n18.west);
	\end{pgfonlayer}
	
\end{scope}
\end{tikzpicture}}
\caption[Poset of quasi-hereditary structures on A]{Poset of quasi-hereditary structures on $ A=\kk  (\tikz[baseline=-0.5ex,xscale=.7,inner sep=.7mm,>=stealth,
line width = .5pt]{ \node (1) at (1,0) {1}; \node (2) at (2,0) {2}; \node (3) at (3,0) {3}; \node (4) at (4,0) {4}; \node (5) at (5,0) {5};
\draw[->] (1)--(2); \draw[->] (3)--(2); \draw[->] (4)--(3); \draw[->] (4)--(5); })$.}
\label{fig:ex-poset-qhstr-A}
\end{figure}
\end{exa}

\section{Path algebras of types \texorpdfstring{$\D$}{TEXT} and \texorpdfstring{$\E$}{TEXT}}\label{section-DE-count}
In this section, we count the number of quasi-hereditary structures on $\kk Q$ for a quiver $Q$ of Dynkin type $\D$ and $\E$.
\subsection{Idempotent reduction}
Let $A$ be a finite dimensional algebra and $\{S(i)\}_{i\in I}$ the set of isomorphism classes of simple $A$-modules.
Fix $i\in I$ and a corresponding idempotent $e_i$ in $A$.
For a partial order $\lhd$ on $I\setminus\{i\}$, we denote by $\lhd^{\prime}$ a partial order on $I$ such that $\lhd^{\prime}|_{I\setminus\{i\}}=\lhd$ and $i$ is a unique maximal element.
\begin{lem}\label{lem_idempotent_adapted}
If $\lhd$ is an adapted order to $A/\langle e_i \rangle$, then $\lhd^{\prime}$ is an adapted order to $A$.
\end{lem}
\begin{proof}
Let $M$ be an $A$-module such that $\Top M \cong  S(j)$ and $\soc M \cong  S(k)$ for some $j, k\in I\setminus\{i\}$, and $j, k$ are incomparable with respect to $\lhd^{\prime}$.
If $M$ is an $A/\langle e_i \rangle$-module, then there is nothing to show, since $\lhd$ is an adapted order to $A/\langle e_i \rangle$.
On the other hand, if $Me_i\neq 0$, then $j \lhd^{\prime} i$, $k \lhd^{\prime} i$ implies that $\lhd^{\prime}$ is adapted to $A$.
\end{proof}
\begin{lem}\label{lem_idempotent_qh}
Let $\lhd_1, \lhd_2$ be partial orders on $I\setminus\{i\}$ which define quasi-hereditary structures on $A/\ideal{e_i}$.
Assume that both $\lhd_1^{\prime}, \lhd_2^{\prime}$ define quasi-hereditary structures on $A$.
Then $\lhd_1 \sim \lhd_2$ if and only if $\lhd_1^{\prime} \sim \lhd_2^{\prime}$.
\end{lem}
\begin{proof}
By Lemma \ref{lem-qh-maximal}, it is clear that $\lhd_1^{\prime} \sim \lhd_2^{\prime}$ implies $\lhd_1 \sim \lhd_2$.
Since $i$ is a unique maximal element in $I$, $\Delta_1^{\prime}(i) = P(i) = \Delta_2^{\prime}(i)$ holds.
Thus the converse is also true.
\end{proof}
Then we concentrate on path algebras.
Let $Q$ be a finite acyclic quiver and fix $i\in Q_0$.
By Lemmas \ref{lem_idempotent_adapted} and \ref{lem_idempotent_qh}, we have the following well-defined injective map:
\[
\iota_i \colon  \qhstr(\kk Q/\ideal{e_i}) \longrightarrow \qhstr(\kk Q) \qquad [\lhd] \mapsto [\lhd^{\prime}].
\]
\begin{lem}\label{lem_idempotent_path_algebra}
For a finite acyclic quiver $Q$, we have
$$
\qhstr(\kk Q) = \bigcup_{i\in Q_0}\mathsf{Im}(\iota_i),
$$
and each $\mathsf{Im}(\iota_i)$ bijectively corresponds to $\qhstr(\kk Q/\ideal{e_i})$.
\end{lem}
\begin{proof}
Let $\lhd$ be a partial order on $Q_0$ which defines a quasi-hereditary structure on $\kk Q$.
Then we may assume that $\lhd$ is a total order with a unique maximal element $i$.
Then we have a total order $\lhd|_{Q_0\setminus\{i\}}$ on $Q_0\setminus\{i\}$ and we have $\iota_i([\lhd|_{Q_0\setminus\{i\}}])=[\lhd]$.
\end{proof}
\subsection{Decomposition of the set of quasi-hereditary structures}
Throughout this subsection, let $r, s, t\in\Z_{\geq 1}$ and $Q=Q(r,s,t)$ the following quiver:
\begin{align}\label{quiver-DE}
\begin{tikzpicture}[baseline=0]
\node(al)at(1,0){$a_r$};
\node(al-1)at(2.5,0){$a_{r-1}$};
\node(acdots)at(4,0){$\cdots$};
\node(a1)at(5.5,0){$a_1$};
\node(0)at(7,0){$a_0$};
\node(b1)at(8,1){$b_1$};
\node(bcdots)at(9.5,1){$\cdots$};
\node(bm)at(11,1){$b_s$};
\node(c1)at(8,-1){$c_1$};
\node(ccdots)at(9.5,-1){$\cdots$};
\node(cn)at(11,-1){$c_t$};
\draw[thick, ->] (al)--(al-1);
\draw[thick, ->] (al-1)--(acdots);
\draw[thick, ->] (acdots)--(a1);
\draw[thick, ->] (a1)--(0);
\draw[thick, ->] (0)--(b1);
\draw[thick, ->] (0)--(c1);
\draw[thick, ->] (b1)--(bcdots);
\draw[thick, ->] (c1)--(ccdots);
\draw[thick, ->] (bcdots)--(bm);
\draw[thick, ->] (ccdots)--(cn);
\end{tikzpicture}
\end{align}
Recall that we have $\qhstr(\kk Q) = \bigcup_{i}\mathsf{Im}(\iota_i)$, where $i$ runs over all the vertices of the quiver $Q$.
Each $\mathsf{Im}(\iota_i)$ bijectively corresponds to $\qhstr(\kk Q/\ideal{e_i})$. 
Let $Q_0^b \coloneqq  \{b_k \mid 1 \leq k \leq s \}$ and $Q_0^c \coloneqq \{ c_k \mid 1 \leq k \leq t \}$.
We begin with the following lemma.
\begin{lem}\label{lem-D-int}
Let $i,j$ be vertices of $Q$ with $i\neq j$.
If $\mathsf{Im}(\iota_i) \cap \mathsf{Im}(\iota_j) \neq \emptyset$, then we have $i \in Q_0^b$ and $j \in Q_0^c$, or $j \in Q_0^b$ and $i \in Q_0^c$.
\end{lem}
\begin{proof}
Let $\mfq \in \mathsf{Im}(\iota_i) \cap \mathsf{Im}(\iota_j)$ and denote by $\Delta$ the set of standard $\kk Q$-modules associated to $\mfq$.
For each $\ell=i, j$, by the definition of $\iota_{\ell}$, there exist a partial order $\lhd_{\ell}^{\prime}$ on $Q_0$ such that $[\lhd_{\ell}^{\prime}]=\mfq$ and $\ell$ is a unique maximal element of $\lhd_{\ell}^{\prime}$.
By extending $\lhd_{\ell}^{\prime}$ and Lemma \ref{extension_adapted}, we may assume that $\lhd_{\ell}^{\prime}$ is totally ordered for each $\ell=i, j$.
We have $\Delta(i) = P(i)$ and $\Delta(j) = P(j)$.

Assume that $P(i)$ and $P(j)$ have a common composition factor $S(k)$.
Then we can take $k=i$ or $k=j$ because of the shape of the quiver $Q$.
If $k=i$, then $\Delta(j)=P(j)$ implies that $i \lhd_i^{\prime} j$, which is a contradiction.
Similarly, $k=j$ induces a contradiction.
Thus $P(i)$ and $P(j)$ have no common composition factor.
This implies the assertion.
\end{proof}
Then we observe the intersection $\mathsf{Im}(\iota_i) \cap \mathsf{Im}(\iota_j)$ for $i\in Q_0^b$ and $j\in Q_0^c$.
For each vertex $i\in Q_0$, we have a fully faithful functor $\modu (\kk Q/\ideal{e_i}) \to \modu \kk Q$.
By this functor, we regard $\kk Q/\ideal{e_i}$-modules as $\kk Q$-modules.
\begin{lem}
\label{lem-D-int-rep}
Let $i\in Q_0^b$, $j\in Q_0^c$ and $\mfq_i\in\qhstr(\kk Q/\ideal{e_i})$.
The following statements are equivalent.
\begin{enumerate}
\item[(1)]
There exists a partial order $\lhd$ on $Q_0\setminus\{i\}$ such that $[\lhd]=\mfq_i$ and a vertex $j$ is a unique maximal element by $\lhd$.
\item[($1^{\prime}$)]
There exists a partial order $\lhd$ on $Q_0\setminus\{i\}$ such that $[\lhd]=\mfq_i$ and a vertex $j$ is a maximal element by $\lhd$.
\item[(2)]
$\iota_i(\mfq_i) \in \mathsf{Im}(\iota_j)$.
\end{enumerate}
\end{lem}
\begin{proof}
Clearly, (1) implies ($1^{\prime}$).
For ($1^{\prime}$) to (1), take a total order on $I$ such that it is a refinement of $\lhd$ and $j$  is a unique maximal element.
By Lemma \ref{extension_adapted}, the total order also defines $\mfq_i$.

We show (1) implies (2).
Let $\lhd$ be a partial order as in (1).
Let $\lhd'$ be a partial order on $Q_0$ such that $\lhd'|_{Q_0\setminus\{i\}}=\lhd$ and $k \lhd' i$ holds for any $k\in Q_0$.
We denote by $\tilde{\lhd}$ a partial order on $Q_0$ obtained by switching $i$ and $j$ with respect to $\lhd$, that is, for $k,\ell\in Q_0\setminus\{j\}$, $k\tilde{\lhd} \ell$ holds if and only if $k \lhd' \ell$ holds, and $i \tilde{\lhd} j$ holds, and $k $.
Since $i\in Q_0^b$, $j\in Q_0^c$ and the shape of the quiver $Q$,  $\tilde{\lhd}$ and $\lhd'$ define the same quasi-hereditary structure on $\kk Q$, which is equal to $\iota_i(\mfq_i):=[\lhd']$.
By Lemma \ref{lem-qh-maximal}, $\tilde{\lhd}|_{Q_0\setminus\{j\}}$ defines a qu

asi-hereditary structure on $\kk Q/\ideal{e_j}$, denote the structure by $\mfq'$.
We have $\iota_i(\mfq_i)=[\lhd']=[\tilde{\lhd}] =\iota_j(\mfq')$.

We show (2) implies (1).
Assume that there exists $\mfq_j\in\qhstr(\kk Q/\ideal{e_j})$ with $\iota_i(\mfq_i) = \iota_j(\mfq_j)$.
For $\ell=i,j$, let $\lhd_{\ell}$ be a partial order on $Q_0\setminus\{\ell\}$ such that $[\lhd_{\ell}]=\mfq_{\ell}$.
We have $[\lhd_{\ell}^\p]=\iota_{\ell}(\mfq_{\ell})$.
We denote by $\Delta$ the standard $\kk Q$-modules with respect to $\iota_i(\mfq_i)=\iota_j(\mfq_j)$.
Since $\iota_i(\mfq_i)=\iota_j(\mfq_j)$ is represented by both $\lhd_i^{\prime}$ and $\lhd_j^{\prime}$, for any $k\in Q_0\setminus\{i,j\}$, $\Delta(k)$ does not have composition factors $S(i)$ and $S(j)$, and $\Delta(i) = P(i)$, $\Delta(j) = P(j)$ hold.

We denote by $\lhd^i$ the partial order on $Q_0\setminus\{i\}$ obtained by restricting $\lhd_j^{\prime}$ to $Q_0\setminus\{i\}$.
We show that $[\lhd^i]=\mfq_i$ holds.
Let $\Delta_i$ and $\Delta^i$ the standard $\kk Q/\ideal{e_i}$-modules of $\lhd_i$ and $\lhd^i$, respectively.
We show that $\Delta_i(k) = \Delta(k) = \Delta^i(k)$ holds for any $k\in Q_0\setminus\{i\}$.
Since $\iota_i(\mfq_i)$ is represented by $\lhd_i^{\prime}$, we have $\Delta_i(k) = \Delta(k)$ for any $k\in Q_0\setminus\{i\}$ by Lemma \ref{lem-qh-maximal}.
Let $k \in Q_0\setminus\{i\}$.
Since $j$ is maximal by $\lhd^i$, $i\in Q_0^b$ and $j\in Q_0^c$, $\Delta^i(j) = P(j) = \Delta(j)$ holds.
Assume that $k\neq j$.
Again since $j$ is maximal by $\lhd^i$, $\Delta^i(k)$ does not have a composition factor $S(j)$.
Then we can show that for any $v \in Q_0\setminus\{i, j\}$, $S(v)$ is a composition factor of $\Delta^i(k)$ if and only if $S(v)$ is a composition factor of $\Delta(k)$.
In fact, since $\lhd^i$ is obtained by restricting $\lhd_j^{\prime}$, $\Delta^i(k)$ is a factor of $\Delta(k)$.
Namely, the only if part holds.
On the other hand, if $S(v)$ is a composition factor of $\Delta(k)$, then there exists a (unique) path $p$ in $Q$ from $k$ to $v$ such that each vertex $u$ of the path $p$ satisfies $u \lhd_j^{\prime} k$.
Since $S(i)$ is not a composition factor of $\Delta(k)$, the path $p$ does not factor through a vertex $i$.
Thus the path $p$ is a path in $Q\setminus\{i\}$.
This implies that $S(v)$ is a composition factor of $\Delta^i(k)$.
We have $\Delta(k) = \Delta^i(k)$.
\end{proof}
For $1 \leq k \leq t$, we denote by $\mcC_k$ a subset of $\qhstr(\kk Q/\ideal{e_{c_k}})$ consisting of $\mfq$ such that there exists no partial order $\lhd$ on $Q_0\setminus\{c_k\}$ such that $[\lhd]=\mfq$ and $\lhd$ has one of $b_1, \ldots, b_s$ as a maximal element.
Then by Lemmas \ref{lem_idempotent_path_algebra}, \ref{lem-D-int} and \ref{lem-D-int-rep}, we have the following proposition.
\begin{pro}\label{pro-counting-decomposition}
We have a bijection
$$
\bigsqcup_{i=0}^r\qhstr\left(\frac{\kk Q}{\ideal{e_{a_i}}}\right) \sqcup \bigsqcup_{j=1}^s\qhstr\left(\frac{\kk Q}{\ideal{e_{b_j}}}\right) \sqcup \bigsqcup_{k=1}^t\mcC_k \longrightarrow \qhstr(\kk Q),
$$
where the map is given by $\iota_i$.
\end{pro}
\subsection{Counting quasi-hereditary structures}
In this subsection, we count the number of quasi-hereditary structures on path algebras of Dynkin type $\D$.
Let $Q=Q(r,s,t)$ be the quiver defined in Eq.~\eqref{quiver-DE}.
We consider type $\D_n$ for an integer $n\geq 4$.
By taking the opposite quiver and Theorem \ref{thm_qhstr_divide}, it is enough to study the following two cases.
\begin{align*}
\begin{tikzpicture}[baseline=0]
\node(Q)at(0,0){$Q(n-3,1,1):$};
\node(al-1)at(2.5,0){$a_{n-3}$};
\node(acdots)at(4,0){$\cdots$};
\node(a1)at(5.5,0){$a_1$};
\node(0)at(7,0){$a_0$};
\node(b1)at(8,1){$b_1$};
\node(c1)at(8,-1){$c_1$};
\draw[thick, ->] (al-1)--(acdots);
\draw[thick, ->] (acdots)--(a1);
\draw[thick, ->] (a1)--(0);
\draw[thick, ->] (0)--(b1);
\draw[thick, ->] (0)--(c1);
\end{tikzpicture}
\end{align*}
\begin{align*}
\begin{tikzpicture}[baseline=0]
\node(Q)at(3.5,0){$Q(1,n-3,1):$};
\node(a1)at(5.5,0){$a_1$};
\node(0)at(7,0){$a_0$};
\node(b1)at(8,1){$b_1$};
\node(bcdots)at(9.5,1){$\cdots$};
\node(bm)at(11,1){$b_{n-3}$};
\node(c1)at(8,-1){$c_1$};
\draw[thick, ->] (a1)--(0);
\draw[thick, ->] (0)--(b1);
\draw[thick, ->] (0)--(c1);
\draw[thick, ->] (b1)--(bcdots);
\draw[thick, ->] (bcdots)--(bm);
\end{tikzpicture}
\end{align*}
In both cases, $t=1$, so we study $\mcC_1$ for both cases.
We denote by $c_n$ the $n$-th Catalan number $c_n = \frac{1}{n+1} { \binom{2n}{n} }$.
We begin with the following easy observation.
\begin{lem}\label{lem-An-adapted-maximal}
Let $Q=n \to \dots \to 2 \to 1$ be an equioriented quiver of type $\A_n$.
For $1\leq i \leq n$, let $\mathsf{qh}(i)$ be a subset of $\qhstr(\kk Q)$ consisting of $\mfq$ which is represented by a partial order having $i$ as a maximal element.
Then the following statements hold.
\begin{enumerate}
\item
Any adapted order to $\kk Q$ has a unique maximal element.
\item
The image of the map $\iota_i \colon  \qhstr(\kk Q/\ideal{e_i}) \to \qhstr(\kk Q)$ corresponds to $\mathsf{qh}(i)$, that is, $\mathsf{Im}(\iota_i) = \mathsf{qh}(i)$.
In particular, we have a decomposition
$\qhstr(\kk Q) = \bigsqcup_{i=1}^n\mathsf{Im}(\iota_i).$
\end{enumerate}
\end{lem}
\begin{proof}
(1)
By Theorem \ref{thm_Baptiste_bijection}, any adapted order to $\kk Q$ is an extension of the adapted order induced from a binary tree.
Clearly, each partial order induced from binary trees has a unique maximal element.
(2) directly follows from the definition of $\iota_i$.
\end{proof}
\begin{lem}\label{lem-counting-Q(n-3,1,1)}
Let $Q=Q(n-3,1,1)$.
Then we have the following equalities.
\begin{enumerate}
\item
$|\mcC_1|=c_{n-1}-c_{n-2}$
\item
$|\qhstr(\kk Q)|=2c_n-3c_{n-1}$
\end{enumerate}
\end{lem}
\begin{proof}
(1)
Recall that $\mcC_1$ a subset of $\qhstr(\kk Q/\ideal{e_{c_1}})$ consisting of $\mfq$ such that there exist no partial orders on $Q_0\setminus\{c_1\}$ which represent $\mfq$ and have $b_1$ as a maximal element.
Therefore, by Lemma \ref{lem-An-adapted-maximal}, we have $\mcC_1 = \qhstr(\kk Q/\ideal{e_{c_1}}) \setminus \mathsf{Im}(\iota_{b_1})$.
Therefore, we have $|\mcC_1|=|\qhstr(\kk Q/\ideal{e_{c_1}})|-|\mathsf{Im}(\iota_{b_1})| = |\qhstr(\kk Q/\ideal{e_{c_1}})| - |\qhstr(\kk Q/\ideal{e_{c_1}, e_{b_1}})| = c_{n-1}-c_{n-2}$ by Corollary \ref{cor-An-Tamari}.

(2)
We show the equality by an induction on $n$.
Although we are considering $\D$ type, it is easy to see that the equality holds for $n=3$.
Assume that $n>3$.
We have the following equalities.
\begin{align}
|\qhstr(\kk Q)| & = \left|\bigsqcup_{i=0}^{n-3}\qhstr\left(\frac{\kk Q}{\ideal{e_{a_i}}}\right)\right| + \left|\qhstr\left(\frac{\kk Q}{\ideal{e_{b_1}}}\right)\right| + |\mcC_1| \label{equ-1-Q(n-3,1,1)} \\
& = \sum_{i=0}^{n-3}c_{n-i-3}(2c_{i+2}-3c_{i+1}) + c_{n-1} + c_{n-1} - c_{n-2} \label{equ-2-Q(n-3,1,1)} \\
& = 2(c_n - c_{n-1} -c_{n-2}) -3(c_{n-1} -c_{n-2}) +2c_{n-1} -c_{n-2} \label{equ-3-Q(n-3,1,1)} \\
& = 2c_n-3c_{n-1}.\nonumber
\end{align}
For the first and the second equalities (\ref{equ-1-Q(n-3,1,1)}) and (\ref{equ-2-Q(n-3,1,1)}), we use the statement (1), Theorem \ref{thm_Baptiste_bijection}, Proposition \ref{pro-counting-decomposition}, Lemma \ref{lem-An-adapted-maximal}, and an inductive hypothesis.
For the third equality (\ref{equ-3-Q(n-3,1,1)}), we use the well known equality $c_{n+1}=\sum_{i=0}^nc_{n-i}c_i$.
\end{proof}
\begin{lem}\label{lem-counting-Q(1,n-3,1)}
Let $Q=Q(1,n-3,1)$.
Then the following equalities hold.
\begin{enumerate}
\item
$|\mcC_1| = c_{n-2}+c_{n-3}$
\item
$|\qhstr(\kk Q)|=3c_{n-1}-c_{n-2}$
\end{enumerate}
\end{lem}
\begin{proof}
(1)
Recall that $\mcC_1$ is a subset of $\qhstr(\kk Q/\ideal{e_{c_1}})$ consisting of $\mfq$ such that there exist no partial orders on $Q_0\setminus\{c_1\}$ which represent $\mfq$ and have one of $b_1, \ldots, b_{n-3}$ as a maximal element.
Therefore, by Lemma \ref{lem-An-adapted-maximal} and Corollary \ref{cor-An-Tamari}, we have
$|\mcC_1| = |\qhstr(\kk Q/\ideal{e_{c_1}, e_{a_1}})| + |\qhstr(\kk Q/\ideal{e_{c_1}, e_{a_0}})| = c_{n-2}+c_{n-3}$.

(2)
We show the equality by an induction on $n$.
For $n=3$, the equality holds.
Assume that $n>3$.
We have the following equalities.
\begin{align*}
|\qhstr(\kk Q)| & = \left|\bigsqcup_{i=1}^{n-3}\qhstr\left(\frac{\kk Q}{\ideal{e_{b_i}}}\right)\right| + \left|\qhstr\left(\frac{\kk Q}{\ideal{e_{a_1}}}\right)\right| + \left|\qhstr\left(\frac{\kk Q}{\ideal{e_{a_0}}}\right)\right| + |\mcC_1| \\
& = \sum_{i=1}^{n-3}(3c_{i+1}-c_{i})c_{n-i-3} + 2c_{n-2} + c_{n-3} + c_{n-2} + c_{n-3}\\
& = 3(c_{n-1}-c_{n-3}-c_{n-2})-(c_{n-2}-c_{n-3}) + 3c_{n-2} +2c_{n-3} \\
 & =3c_{n-1} - c_{n-2}.
\end{align*}
All equalities are deduced in a similar way as those of the proof of Lemma \ref{lem-counting-Q(n-3,1,1)}.
\end{proof}
We find the following correspondence to  \href{https://oeis.org/}{OEIS} sequences \cite{OEIS}.

\begin{table}[h!]
\renewcommand{\arraystretch}{1.3}
\centering
\begin{tabular}{@{} c x{2.3cm} x{2.5cm} c  @{}} 
\toprule
$Q$ & & Formula & {OEIS}   \\
\midrule
\multirow{2}{*}{$ Q(n-3,1,1) $} & $|\mcC_{1}|$ & $ {c_{n-1}-c_{n-2}} $ & $\operatorname{\href{https://oeis.org/A000245}{A000245}} (n-2) $      \\
& $  {|\qhstr (\kk Q)|} $ & $ 2c_{n}-3c_{n-1} $ & $\operatorname{\href{https://oeis.org/A070031}{A070031}} (n-2) $  \\
\cmidrule(lr){2-4}
\addlinespace[3pt]
$ Q(1,n-3,1) $& $|\mcC_{1}|$ & $ c_{n-2}+c_{n-3} $ & $\operatorname{\href{https://oeis.org/A005807}{A005807}} (n-3) $   \\
\bottomrule
\end{tabular}
\end{table}
The sequence $(3c_{n-1}-c_{n-2})_{n\geq 2}$ is not listed in \cite{OEIS} until the publication of this article.
\begin{exa}\label{typeE}
Here we give a list of the numbers of quasi-hereditary structures on path algebras of Dynkin types $\E_6$, $\E_7$, and $\E_8$.
By taking the opposite quiver and Theorem \ref{thm_qhstr_divide}, it is enough to calculate the following cases of $Q=Q(r,s,t)$.
\begin{enumerate}
\item
For $\E_6$, $(r,s,t)=(1,2,2)$ or $(2,2,1)$.
\item
For $\E_7$, $(r,s,t)=(1,3,2), (2,3,1)$ or $(3,2,1)$.
\item
For $\E_8$, $(r,s,t)=(1,4,2), (2,4,1)$ or $(4,2,1)$.
\end{enumerate}
\begin{table}[h!]
\renewcommand{\arraystretch}{1.2}
\centering
\begin{tabular}{@{} r c S[table-format = 1]  S[table-format = 1]  S[table-format = 3] @{}} 
\toprule
 & $\phantom{ii} (r,s,t)\phantom{ii} $ & $\phantom{ii}|\mcC_{1}|\phantom{ii}$ & $\phantom{ii}|\mcC_{2}|\phantom{ii}$ & $  {|\qhstr (\kk Q)|} $ \\
\midrule
\multirow{2}{*}{$\E_{6}$} & $ (1,2,2) $ & 7 & 19 & 106    \\
& $ (2,2,1) $ & 23 & 0 & 130 \\
\cmidrule(lr){2-5}
\addlinespace[2pt]
\multirow{3}{*}{$\E_{7}$} & $ (1,3,2) $ & 19 & 52 & 322 \\
& $ (2,3,1) $ & 66 & 0 & 416 \\
& $ (3,2,1) $ & 76 & 0 & 453 \\
\cmidrule(lr){2-5}
\addlinespace[2pt]
\multirow{3}{*}{$\E_{8}$} & $ (1,4,2) $ & 56 & 154 & 1020 \\
& $ (2,4,1) $ & 202 & 0 & 1368 \\
& $ (4,2,1) $ & 255 & 0 & 1584 \\
\bottomrule
\end{tabular}
\end{table}
For example, let $Q=Q(1,3,2)$.
Then $|\mcC_1|$, $|\mcC_2|$ and $|\qhstr(\kk Q)|$ are obtained as follows:
\begin{align*}
|\mcC_1| & = |\qhstr(\kk Q/\ideal{e_{c_1}})| - \sum_{i=1}^3 |\qhstr(\kk Q/\ideal{e_{c_1}, e_{b_i}})| \\
& = c_5 - c_2c_2 - c_3 - c_4 = 19,\\
|\mcC_2| & = |\qhstr(\kk Q/\ideal{e_{c_2}})| - \sum_{i=1}^3 |\qhstr(\kk Q/\ideal{e_{c_2}, e_{b_i}})| \\
& = (3c_5-c_4) - c_3c_2 - (3c_3-c_2) - (3c_4 - c_3) = 52,\\
|\qhstr(\kk Q)| & = \sum_{i=0}^1|\qhstr(\kk Q/\ideal{e_{a_i}})| + \sum_{i=1}^3|\qhstr(\kk Q/\ideal{e_{b_i}})| + |\mcC_1| + |\mcC_2| \\
& = c_3c_2 + c_4c_3 + c_4c_2 + (3c_4-c_3) + |\qhstr(\kk Q(1,2,2))| + 19 + 52 \\
& = 322.
\end{align*}
\end{exa}
\section{Lattice of quasi-hereditary structures}
\label{section-lattice-qhstr}
Throughout this section, all algebras are defined over a fixed field $\mathbf{k}$.
\subsection{Lattice properties and incidence algebras}
We denote by $\tilde{\rm D}_n$ the following quiver,
\begin{align}\label{quiver-D-tilde}
\begin{tikzpicture}[baseline=0]
\node(01+)at(0,0.8){$\circ$};
\node(01-)at(0,-0.8){$\circ$};
\node(al)at(1,0){$\circ$};
\node(al-1)at(2.5,0){$\circ$};
\node(acdots)at(4,0){$\cdots$};
\node(a1)at(5.5,0){$\circ$};
\node(0)at(7,0){$\circ$};
\node(b1)at(8,0.8){$\circ$};
\node(c1)at(8,-0.8){$\circ$};
\draw[thick, ->] (01+)--(al);
\draw[thick, ->] (01-)--(al);
\draw[thick, ->] (al)--(al-1);
\draw[thick, ->] (al-1)--(acdots);
\draw[thick, ->] (acdots)--(a1);
\draw[thick, ->] (a1)--(0);
\draw[thick, ->] (0)--(b1);
\draw[thick, ->] (0)--(c1);
\end{tikzpicture}
\end{align}
where the number of vertices is $n+1$.

Recall that for a finite acyclic quiver $Q$, the underling graph of $Q$ is a non-oriented graph (possibly having multiple edges) obtained by ignoring orientations of arrows of $Q$.
Such underlying graph is called a tree if there is no closed walks with its length greater two.

The main result of this section is the following theorem.
\begin{thm}\label{tree-qh-lattice}
Let $Q$ be a finite acyclic quiver whose underlying graph is a tree.
Then the set $\qhstr(\kk Q)$ of quasi-hereditary structures on $\kk Q$ is a lattice if and only if $Q$ does not have a quiver $\tilde{\rm D}_n$ as a subquiver for any $n\geq 4$.
\end{thm}

From now on we assume that a finite acyclic quiver $Q$ has no multiple arrows.
To show Theorem \ref{tree-qh-lattice}, by Proposition \ref{prop-multiple-free}, it is enough to study under this assumption.

Curiously, this characterization does not seem to have a natural generalization to the setting of general acyclic quivers. However, a tree quiver $Q$ can be naturally seen as a partial order $(Q_0,\leadsto)$ where $x \leadsto y$ if there is a (necessarily unique) path from $x$ to $y$. Then, the path algebra of $Q$ is isomorphic to the incidence algebra of $(Q_0,\leadsto)$. Our experimentation suggests that this may be a good setting for a generalization of Theorem \ref{tree-qh-lattice}. 

Let $n\in \mathbb{N}$ be an even integer and $Z_n$ be the following zigzag orientation of an affine Dynkin diagram of type $\A$.

\begin{align}\label{zigzag}
\begin{tikzpicture}[baseline=-1.3cm, vertex/.style 2 args = {label ={ [label distance=-3mm] #1: {$#2$}},circle} ]
\node(0)at(0,0)[ vertex={120}{1}]{$\circ$}; 
\node(1)at(1,0)[ vertex={50}{2}]{$\circ$}; 
\node(2)at(1.75,-0.87)[vertex={0}{3}]{$\circ$}; 
\node(3)at(1.75,-1.87)[vertex={0}{4}] {$\circ$}; 
\node(nm2)at(0,-2.6)[vertex={200}{n-2}]{$\circ$};  
\node(nm1)at(-0.75,-1.87)[vertex={180}{n-1}]{$\circ$}; 
\node(n)at(-0.75,-0.87)[vertex={180}{n\,}]{$\circ$}; 
\draw[thick, ->] (0)--(1);
\draw[thick, ->] (2)--(1);
\draw[thick, ->] (2)--(3);
\draw[thick, ->] (nm1)--(nm2);
\draw[thick, ->] (nm1)--(n);
\draw[thick, ->] (0)--(n);
\draw[thick,dotted] (nm2) to [out=-10,in=-120](3);
\end{tikzpicture}
\end{align}
We may view $Z_n$ as a quiver, or as the Hasse diagram of a poset. Even for tree quivers, the language of posets is more flexible: for example, it is easy to see that $Z_4$ is not a subquiver of $\widetilde{D}_4$ but it is a (full) subposet of $(\widetilde{D}_4,\leadsto)$.

We remind the reader that $(P,\leq_P)$ is a \emph{subposet} of $(Q,\leq_Q)$ if $ P \subseteq Q $ and the inclusion $P\hookrightarrow Q$ is a morphism of posets. Now we recall the following central definitions for this section.

\begin{dfn}
Let $(P,\leq_P)$ and $(Q,\leq_Q)$ be two posets and $\phi \colon  P\to Q$ a morphism of posets. Then,
\begin{enumerate}
    \item The map $\phi$ is \emph{full} if for any $p_1, p_2\in P$,  $ p_1\leq_{P} p_2$ holds if and only if  $\phi(p_1)\leq_{Q}\phi(p_2)$ holds.
    \item If $\phi$ is full and injective, the image of $\phi$ is called a \emph{full subposet} of $(P,\leq_P)$.
\end{enumerate}
\end{dfn}

We remark that some authors call \emph{weak} subposets to our subposets, and full subposets are also known as \emph{induced} subposets.

Let $(P,\leq)$ be a poset. For $i\leq j$ in $P$, an interval $[i,j]$ is always considered with respect to $\leq$, that is, $[i,j]\coloneqq \{k\in P \mid i \leq k \leq j \}$.

\begin{dfn}\label{def:incidence-alg}
Let $(P,\leq)$ be a finite poset. The \emph{incidence algebra} of $P$ over $\kk$, denoted by $A(P)$ is the $\kk$-vector space with basis the set of intervals $[i,j]$ in the poset $P$ with multiplication induced by $[i,j]\cdot[k,l]=[i,l]$ if $j=k$, and $0$ otherwise.
\end{dfn}

Then, we propose the following generalization of Theorem \ref{tree-qh-lattice}.
\begin{conj}\label{conj-qh-lattice-poset}
Let $(P,\leq)$ be a finite poset. Then, the poset of quasi-hereditary structures on the incidence algebra of $(P,\leq)$ is a lattice if and only if $Z_n$ is not isomorphic to a full subposet of $(P,\leq)$ for any $n\geq 4$.
\end{conj}

Actually, the setting of incidence algebras is not only a good setting for a generalization of our result, it also simplifies our arguments.
So we prove Theorem \ref{tree-qh-lattice} as a corollary of the following theorem. 

\begin{thm}\label{thm-qh-lattice-poset}
Let $(P,\leq)$ be a finite poset. We assume that the incidence algebra $A(P)$ of $(P,\leq)$ is hereditary. Then, the poset of quasi-hereditary structures on $A(P)$ is a lattice if and only if $Z_n$ is not isomorphic to a full subposet of $(P,\leq)$ for any $n\geq 4$.
\end{thm}
\begin{proof}
The only if part follows from Proposition \ref{pro-an-zz} and Theorem \ref{thm-qh-not-lattice}.
The if part follows from Theorem \ref{thm-lattice-tree}.
\end{proof}

We prove Theorem \ref{thm-qh-not-lattice} in Subsection \ref{qh-full-subposet} and show Theorem \ref{thm-lattice-tree} in Subsection \ref{subsection-lattice-tree}. For Theorems \ref{tree-qh-lattice} and \ref{thm-qh-lattice-poset} we need the following Lemmas. The precise comparison between Theorem \ref{tree-qh-lattice} and Theorem \ref{thm-qh-lattice-poset} is given in Remark \ref{comp-thms}. 

\begin{lem}\label{lem-Dn-Zn}
Let $Q$ be a finite quiver whose underlying graph is a tree.
Then the following statements hold.
\begin{enumerate}
\item
$Z_n$ is not a full subposet of $(Q_0, \leadsto)$ for any even $n\geq 6$.
\item
$Z_4$ is not a full subposet of $(Q_0, \leadsto)$ if and only if $\tilde{D}_n$ is not a full subquiver of $Q$ for any $n\geq 4$.
\end{enumerate}
\end{lem}
\begin{proof}
(1)
Assume that there exists a subset $Z=\{1,2,\dots,n\}$ of $Q_0$ such that $(Z, \leadsto|_{Z})$ is isomorphic to $Z_n$.
So we have paths in $Q$ of the form:
\[
n \flowsfroma 1 \leadsto 2 \flowsfroma \cdots \leadsto n-2 \flowsfroma n-1 \leadsto n.
\]

For the rest of this proof, integers are considered modulo $n$.
For each odd $k$, there exists a vertex $k^\p\in Q_0$ and three paths $k \leadsto k^\p$, $k^\p \leadsto (k-1)$ and $k^\p \leadsto (k+1)$ such that $k^\p \leadsto (k-1)$ and $k^\p \leadsto (k+1)$ have no common vertices except $k^\p$.
Similarly, for each even $k$, there exist a vertex $k^\p$ and three paths $(k-1) \leadsto k^\p$, $(k+1) \leadsto k^\p$ and $k^\p \leadsto k$ such that $(k-1) \leadsto k^\p$ and $(k+1) \leadsto k^\p$ have no common vertices except $k^\p$.
In particular, we have the following paths:
\begin{align*}
\begin{tikzpicture}[baseline=-0.5]
\node(j0)at(0,0){$n^\p$};
\node(i0)at(0,-1){$n$};
\node(i1)at(1,0){$1$};
\node(j1)at(1,-1){$1^\p$};
\node(j2)at(2,0){$2^\p$};
\node(i2)at(2,-1){$2$};
\node(i3)at(3,0){$3$};
\node(j3)at(3,-1){$3^\p$};
\node(j4)at(4,0){$\cdots$};
\node(i4)at(4,-1){$\cdots$};
\node(jl-1)at(6,0){$(n-2)^\p$};
\node(il-1)at(6,-1){$(n-2)$};
\node(il)at(8,0){$(n-1)$};
\node(jl)at(8,-1){$(n-1)^\p$};
\node(jl+1)at(10,0){$n^\p$};
\node(il+1)at(10,-1){$n$};
\draw [thick, ->, decorate, decoration={snake,amplitude=.4mm,segment length=2mm,post length=1mm}]
(j0) -- (i0);
\draw [thick, ->, decorate, decoration={snake,amplitude=.4mm,segment length=2mm,post length=1mm}]
(i1) -- (j0);
\draw [thick, ->, decorate, decoration={snake,amplitude=.4mm,segment length=2mm,post length=1mm}]
(i1) -- (j1);
\draw [thick, ->, decorate, decoration={snake,amplitude=.4mm,segment length=2mm,post length=1mm}]
(j1) -- (i0);
\draw [thick, ->, decorate, decoration={snake,amplitude=.4mm,segment length=2mm,post length=1mm}]
(i1) -- (j2);
\draw [thick, ->, decorate, decoration={snake,amplitude=.4mm,segment length=2mm,post length=1mm}]
(j1) -- (i2);
\draw [thick, ->, decorate, decoration={snake,amplitude=.4mm,segment length=2mm,post length=1mm}]
(j2) -- (i2);
\draw [thick, ->, decorate, decoration={snake,amplitude=.4mm,segment length=2mm,post length=1mm}]
(i3) -- (j2);
\draw [thick, ->, decorate, decoration={snake,amplitude=.4mm,segment length=2mm,post length=1mm}]
(j3) -- (i2);
\draw [thick, ->, decorate, decoration={snake,amplitude=.4mm,segment length=2mm,post length=1mm}]
(i3) -- (j3);
\draw [thick, ->, decorate, decoration={snake,amplitude=.4mm,segment length=2mm,post length=1mm}]
(i3) -- (j4);
\draw [thick, ->, decorate, decoration={snake,amplitude=.4mm,segment length=2mm,post length=1mm}]
(j3) -- (i4);
\draw [thick, ->, decorate, decoration={snake,amplitude=.4mm,segment length=2mm,post length=1mm}]
(j4) -- (jl-1);
\draw [thick, ->, decorate, decoration={snake,amplitude=.4mm,segment length=2mm,post length=1mm}]
(i4) -- (il-1);
\draw [thick, ->, decorate, decoration={snake,amplitude=.4mm,segment length=2mm,post length=1mm}]
(jl-1) -- (il-1);
\draw [thick, ->, decorate, decoration={snake,amplitude=.4mm,segment length=2mm,post length=1mm}]
(il) -- (jl-1);
\draw [thick, ->, decorate, decoration={snake,amplitude=.4mm,segment length=2mm,post length=1mm}]
(il) -- (jl);
\draw [thick, ->, decorate, decoration={snake,amplitude=.4mm,segment length=2mm,post length=1mm}]
(il) -- (jl+1);
\draw [thick, ->, decorate, decoration={snake,amplitude=.4mm,segment length=2mm,post length=1mm}]
(jl) -- (il+1);
\draw [thick, ->, decorate, decoration={snake,amplitude=.4mm,segment length=2mm,post length=1mm}]
(jl) -- (il-1);
\draw [thick, ->, decorate, decoration={snake,amplitude=.4mm,segment length=2mm,post length=1mm}]
(jl+1) -- (il+1);
\end{tikzpicture}
\end{align*}

Let $k$ be an odd integer.
Since $Q$ is tree, both $k^\p$ and $(k+1)^\p$ appear in the path $k \leadsto (k+1)$.
If there exists a path from $(k+1)^\p$ to $k^\p$, then there exists a path from $(k+2)$ to $(k-1)$.
This is a contradiction, since $Z_n$ is a full subposet of $(Q_0, \leadsto)$ and $n\geq 6$.
Therefore we assume that there exists a path from $k^\p$ to $(k+1)^\p$.
Similarly, for even $k$, we assume that there exists a path from $(k+1)^\p$ to $k^\p$.
We have paths as follows:
$$
n^\p \flowsfroma 1^\p \leadsto 2^\p \flowsfroma \cdots \leadsto (n-2)^\p \flowsfroma (n-1)^\p \leadsto n^\p.
$$
Because of the construction of $k^\p$, any two adjacent paths have no common vertices.
This is a contradiction, since $Q$ is tree.

(2)
Assume that $\tilde{D}_n$ is a full subquiver of $Q$ for some $n$.
Let $Z=\{1,2,3,4\}$ be the four vertices of $\tilde{D}_n$ such that each of them is a sink or a source.
Then $(Z, \leadsto|_Z)$ is isomorphic to $Z_4$.
Conversely, assume that $Z_4$ is a full subposet of $(Q_0, \leadsto)$.
Then the same argument in the proof of (1) implies that $\tilde{D}_n$ is a full subquiver of $Q$ for some $n$.
\end{proof}

We first observe under which conditions an incidence algebra is hereditary and we collect some basic properties of such an incidence algebra. For that, we recall the next concept.

Let $(P,\leq)$ be a finite poset. We say that $P$ is \emph{diamond-free} if it does not have four elements $a$, $b$, $c$, and $d$ forming a diamond suborder with $a\leq b \leq d$ and $a \leq c \leq d$ and with $b$ and $c$ incomparable.
The Hasse diagram of this type of poset is also called a \emph{multi-tree}. 

\begin{lem}\label{lem-incidence-hereditary}
Let $(P,\leq)$ be a finite poset. Then
\begin{enumerate}
\item The algebra $A(P)$ is hereditary if and only if $P$ is a diamond-free poset. 
\item If $(Q,\leq)$ is a full subposet of $(P,\leq)$ and $A(P)$ is a hereditary algebra, then $A(Q)$ is a hereditary algebra. 
\end{enumerate}
\end{lem}
\begin{proof}
Recall that the incidence algebra of a finite poset is isomorphic to the quotient of the path algebra of its Hasse quiver modulo the relations of total commutativity.
As a consequence, if $(P,\leq)$ is a diamond-free poset, there is no relation to mod out, so the incidence algebra is hereditary.
Conversely, assume that there are $a,b,c$ and $d$ in $P$ such that $a\leq b \leq d$, $a\leq c\leq d$ and $b$ and $c$ are incomparable with respect to $\leq$.
Let $M$ be the submodule of $P_a$ supported by the elements larger or equal than $b$ or $c$. Its top is $S_b$ and $S_c$, so its projective cover is $P_b\oplus P_c$.
However $S_d$ appears twice as a composition factor in this direct sum but only once in $P_a$. 
So $P_a$ has a submodule which is not projective and the algebra is not hereditary. The second statement of the lemma is a direct consequence of the first. 
\end{proof}

\begin{lem}\label{diamond-zz-free}
Let $(P,\leq)$ be a connected finite poset which is diamond-free and has no $Z_n$ as full subposet for any $n\geq 4$. Then, its Hasse diagram is a tree.
\end{lem}
\begin{proof}
We only sketch the proof. Assume that there is a cycle in the unoriented Hasse diagram of $(P,\leq)$. Since $(P,\leq)$ is a poset, it is neither an oriented cycle nor a $3$-cycle. The cycle has at least one source and one sink. If it has exactly one source and one sink, then it is not diamond-free. So, it has at least two sources and two sinks, and the number of sources is equal to the number of sinks. Then, the subset consisting of all the sources and sinks of the cycle induces a full subposet isomorphic to $Z_{n}$ for some $n\geq 4$.
\end{proof}

\begin{rem}\label{comp-thms}
As a consequence of Lemmas \ref{lem-Dn-Zn}, \ref{diamond-zz-free} and Theorem \ref{thm-qh-lattice-poset} we see that:
\begin{enumerate}
    \item An \emph{hereditary} incidence algebra of a finite poset $(P,\leq)$ has a \emph{lattice} of quasi-hereditary structures if and only if the Hasse quiver of $(P,\leq)$ is a tree which does not have $\tilde{D}_{n}$ as subquivers for any $n\geq 4$. 
    \item On the other hand, Theorem \ref{thm-qh-lattice-poset} is slightly stronger than Theorem \ref{tree-qh-lattice} since it allows to treat more general quivers such as (\ref{zigzag}). 
\end{enumerate}
\end{rem}

We list easy properties about the incidence algebra of a diamond-free poset.
All the properties are easy to show and we omit the proof, and we use the following lemma without referring.
\begin{lem}
Let $(P,\leq)$ be a finite diamond-free poset and $i,j\in P$.
The following statements hold.
\begin{enumerate}
\item 
If $i\leq j$, then $([i,j],\leq)$ is a total order. In other words, in the Hasse quiver of $(P,\leq)$ there is a unique path from $i$ to $j$. 
\item
If  $i\leq j$, then there exists an $A(P)$-module
with simple top $S(i)$ and simple socle $S(j)$, and such that a simple module $S(k)$ is a composition factor of it if and only if $i\leq k \leq j$ holds.
\item
Let $L$ be a submodule of an indecomposable projective $A(P)$-module.
Then any non-zero indecomposable direct summand of $L$ has a simple top.
\end{enumerate}
\end{lem}

We end this subsection by showing the following key result.

\begin{pro}\label{pro-an-zz}
For an even integer $n\geq 4$, the poset of quasi-hereditary structures on the path algebra $\kk Z_n$ is not a lattice. 
\end{pro}
\begin{proof}
We assume that the quiver is labeled as in $(\ref{zigzag})$. This means that the sources of the quiver are labeled by odd integers and the sinks by even integers. Moreover we view the labels as elements of $\mathbb{Z}/n\mathbb{Z}$ in order to have a cyclic labeling.

The projective indecomposable $\kk Z_n$-modules indexed by sinks are simple and the projective indecomposable module indexed by a source $i$ has top $S(i)$ and socle $S(i-1)\oplus S(i+1)$.
The injective indecomposable $\kk Z_n$-modules indexed by sources are simple and the projective indecomposable module indexed by a sink $i$ has source $S(i)$ and top $S(i-1)\oplus S(i+1)$.
For vertices $i, j$, we denote by $E^i_{j}$, if it exists, a (unique) uniserial module of length two with top $S(i)$ and socle $S(j)$.

For each vertex $s$, we consider a total order $\lhd_s$ defined by  
\[
s\lhd_s s+1\lhd_s \cdots \lhd_s s-1.
\]
This ordering is total, so this is adapted.
Since the algebra is hereditary, this represents a quasi-hereditary structure of $\kk Z_n$.
We denote by $\Delta_s$ (resp. $\nabla_s$) the corresponding set of standard (resp. costandard) modules.
We show that $\lhd_1$ and $\lhd_3$ does not admit a join.

By direct calculation, we have that for each odd $i$ and each even $i'$:
\[
\Delta_{i}(j)=\left\{
    \begin{array}{ll}
    S(j) & \mbox{$j$ is even, or $j=i$} \\
    E^j_{j-1} & \mbox{$j\neq i$ is odd} 
    \end{array}
\right.
,
\quad
\Delta_{i'}(j)=\left\{
    \begin{array}{ll}
    S(j) & \mbox{$j$ is even} \\
    E^{j}_{j-1} & \mbox{$j\neq i'-1$ is odd} \\
    P(j) & j=i'-1
    \end{array}
\right.
\]
Therefore these total orders represent different quasi-hereditary structures and we have $[\lhd_{i}] \preceq [\lhd_{i'}]$ for each odd $i$ and each even $i'$ by Lemma \ref{lem-surj-Dec}.
Assume that there exists a join of $\lhd_1$ and $\lhd_3$, and we denote it by $\lhd_J$.
Since $\lhd_J$ is a join of $\lhd_1$ and $\lhd_3$, $[\lhd_{\ell}] \preceq [\lhd_J] \preceq [\lhd_{i'}]$ holds for $\ell=1,3$ and each even $i'$.
Again by Lemma \ref{lem-surj-Dec}, for each vertex $j$, there exist surjective morphisms $\Delta_{i'}(j) \to \Delta_J(j) \to \Delta_{\ell}(j)$ for $\ell=1,3$ and each even $i'$.
Since we have surjective morphisms
\[
P(1) = \Delta_2(1) \twoheadrightarrow \Delta_J(1) \twoheadrightarrow \Delta_3(1) = E^1_n,
\]
we have that $\Delta_J(1)=P(1)$ or $\Delta_J(1)=E^1_{n}$.
In both case, $\Delta_J(j)=S(j)$ if $j$ is even and $\Delta_J(j)=E^j_{j-1}$ if $j\neq 1$ is odd, because of $\Delta_2(j) \to \Delta_J(j) \to \Delta_1(j)$. This implies that $j-1\lhd_J j$ for $j\neq 1$ odd. 
Assume that $\Delta_J(1)=E^1_{n}$.
Since for each odd $i$, $P(i)$ is filtered by $\Delta_J(i)$ and $\Delta_J(i+1)$, by Definition \ref{dfn-qh-alg} (3), $i \lhd_J i+1$ holds. 
This implies that $1 \lhd_J 2 \lhd_J \dots \lhd_J n \lhd_J 1$, which is a contradiction.
Thus we have $\Delta_J(1)=P(1)$.
On the other hand, we have surjective morphisms
\[
E^1_n = \Delta_n(1) \twoheadrightarrow \Delta_J(1) \twoheadrightarrow \Delta_3(1) = E^1_n.
\]
This implies that $\Delta_J(1)=E^1_n$, which is a contradiction.
Therefore $\lhd_1$ and $\lhd_3$ do not admit a join.
\end{proof}
\subsection{Quasi-hereditary structures of full subposets}\label{qh-full-subposet}
In this subsection, we construct a morphism between the sets of quasi-hereditary structures on the incidence algebras of two finite posets such that one is a full subposet of the other (Proposition \ref{qh_embedding}).
This map will be used to prove the `only if' part of Theorem \ref{thm-qh-lattice-poset}, see Theorem \ref{thm-qh-not-lattice}.

We fix, for all this subsection, the following setting and notation: Let $(P,\leq)$ be a finite poset and $(Q,\leq)$ be a full subposet. We fix $\lhd$ an adapted poset to the incidence algebra of $Q$. We let $R = P\setminus Q$. The relation $\leq$ induces a poset structure on $R$.

We consider the binary relation $\hl $ on $P$ defined as follows: 

\begin{enumerate}
\item For $r_1,r_2\in R$, $r_1 \hl  r_2$ if and only if $r_1 \leq r_2$.
\item For $q_1,q_2\in Q$, $q_1 \hl  q_2$ if and only if $q_1 \lhd q_2$. 
\end{enumerate}
In other words $\hl $ restricts as $\leq$ on $R$ and $\lhd$ on $Q$. There is \emph{no} relations of the form $q\hl  r$ with $r\in R$ and $q\in Q$. The relations $r\hl q$ are of two possible shapes:
\begin{enumerate}
\item[(3)] If $r\leq q$ then $r \hl  q$ if and only if there exists $q_1\in [r,q]\cap Q$ such that $q_1\lhd q$ and $[r,q_1]\cap Q = \{q_1\}$.
\item[(4)] If $q\leq r$ then $r\hl  q$ if and only if there exists $q_1\in [q,r]\cap Q$ such that $q_1\lhd q$ and $[q_1,r]\cap Q = \{q_1\}$.
\end{enumerate}
These two conditions may look unnatural for now, but we will see in Lemma \ref{standard-lift} that they induce a natural set of standard modules. 
\begin{lem}\label{lem:hl-poset}
The transitive cover of $\hl $ is a partial order on $P$. 
\end{lem}
\begin{proof}
For this proof we use the symbol $\hl $ for the relation described above and $\hl ^{tc}$ for its transitive cover. Let us verify that it is an antisymmetric relation. Let $x,y \in P$ with $x\hl ^{tc} y$. Then, there is $a_1,\dots,a_n \in P$ such that
\[
x \hl  a_1 \hl  \cdots \hl  a_n \hl  y.
\] 
If there is $i\in \{1,\dots ,n\}$ such that $a_i \in Q$, then $a_{i+1},\dots , a_n$ and $y$ are in $Q$ and $a_i \lhd y$ by transitivity of the partial order $\lhd$ on the elements of $Q$. So, if $x\in Q$, then $y\in Q$ and we have $x\lhd y$. If $y\hl ^{tc} x$, then by the same argument $y\lhd x$ and $y=x$ since $\lhd$ is antisymmetric. If $x\in R$ and there exits $i\in \{1,2,\dots ,n\}$ such that $a_i \in Q$, then $y\in Q$ and $y \mathrel{\widehat{\ntriangleleft}}^{tc} x$.

\noindent If $x\in R$ and all the $a_i$'s are in $R$, then $x\leq y$. If $y\hl ^{tc} x$, then there exist $b_1,\dots , b_r$ in $P$ such that $y\hl  b_1 \hl  b_2 \dots  \hl  b_r \hl  x$. Since $x\in R$, all the $b_i$'s are in $R$ and $y\leq x$. So $x=y$ by antisymmetry of $\leq$. 

\noindent This shows that $\hl ^{tc}$ is antisymmetric and it is by construction reflexive and transitive. 
\end{proof}
From now on the symbol $\hl$ is used for the partial order of Lemma \ref{lem:hl-poset}. \begin{lem}
Let $(P,\leq)$ be a finite poset and $(Q,\leq)$ be a full subposet of $(P,\leq)$. Let $\lhd$ be an adapted poset to $A(Q)$. Then $\hl $ is an adapted poset to $A(P)$. 
\end{lem}
\begin{proof}
Let $x,y\in P$ such that $x\leq y$. If $x$ and $y$ are both in $R$, then $x\hl  y$. If $x$ and $y$ are both in $Q$, there exists $q_1 \in [x,y]\cap Q$ such that $x\lhd q_1$ and $y\lhd q_1$. Since $\hl $ coincides with $\lhd$ on the elements of $Q$, we have $x \hl  q_1$ and $y \hl  q_1$. 

We assume that $x =r \in R$ and $y=q\in Q$ and $r$ and $q$ are incomparable with respect to $\hl $. Let $q_1 \in Q$ such that $r \leq q_1 \leq q$ and $q_1$ is minimal for this property. Then $r\hl  q_1$. Now the poset $\lhd$ is adapted so there exists $q_2 \in [q_1,q[$ such that $q_{1}\lhd q_2$ and $q\lhd q_2$. By transitivity, we have $r\hl  q_2$ and $q\hl  q_2$. The case $x \in Q$ and $y\in R$ is similar. 
\end{proof}
Let us describe the increasing and decreasing relations for $\hl $.
\begin{lem}\label{decreasing}
Let $r\in R$ and $q\in Q$. Then 
\begin{enumerate}
\item $r\hl q$ is increasing if and only if $r\leq q$ and 
\[
\forall q_1\in [r,q]\cap Q,\ q_1\lhd q.
\]
\item $r\hl q$ is decreasing if and only if $q\leq r$ and 
\[
\forall q_1\in [q,r]\cap Q,\ q_1\lhd q.
\]
\end{enumerate}
\end{lem}
\begin{proof}
We only prove the second point since the proof of the first is similar. If $r\hl  q$ is decreasing, then we have $q\leq r$ and $\forall x\in [q,r]$, we have $x\hl  q$. This is in particular true for $x\in Q$. Since $\hl $ restricts as $\lhd$ on the elements of $Q$, we have $x\lhd q$. Conversely assume that $q\leq r$ and $\forall q_1\in [q,r]\cap Q, q_1\lhd q$. Let $x\in [q,r]$. Let $q_1\in Q$ such that $q\leq q_1\leq x$ and $q_1$ is maximal for this property. Then, $q_1 \lhd q$ by hypothesis, and $x\hl  q$ by definition of $\hl $.
\end{proof}

\begin{rem}
We see that our rather technical definition of $\hl $ leads to natural increasing and decreasing relations: they are the relations that are increasing or decreasing when forgetting the elements which are not in $Q$.
\end{rem}
To describe precisely the standard modules associated with this partial order we consider the functors in the right part of the classical idempotent recollement associated to the idempotent $e_Q = \sum_{q\in Q}[q,q] \in A(P)$ corresponding to $Q$. 

\begin{lem}\label{lem:centralizer-IncAlg}
The algebra $e_{Q}A(P)e_Q$ is isomorphic to $A(Q)$. 
\end{lem}
\begin{proof}
The incidence algebra of $P$ over $\kk$ has basis the set of intervals $[x,y]$ such that $x\leq_P y$ are two elements of $P$. Multiplying on the right and the left by $e_Q$, we have a basis consisting of intervals $[x,y]$ such that $x,y\in Q$ and $x\leq_P y$. Since $Q$ is a full subposet of $P$, this is exactly the set of intervals $[x,y]$ such that $x,y\in Q$ and $x\leq_Q y$. 
\end{proof}
There is a `restriction' functor $(-)e_Q \colon  \operatorname{mod}A(P) \to \operatorname{mod}A(Q)$ which sends an $A(P)$-module $M$ to $Me_{Q}$. It has a left adjoint $L = -\otimes_{A(Q)} e_Q A(P)$. We recall without proofs the following well-known, and easy to check, properties of these functors: 
\begin{enumerate}
\item The functor $(-)e_Q$ is exact.
\item The functor $L$ is right exact, fully faithful and sends projective modules to projective modules. 
\item $(-)e_Q \circ L \cong \operatorname{Id}_{\operatorname{mod}A(Q)}$.
\end{enumerate}

\begin{lem}\label{standard-lift}
Let $(P,\leq)$ be a finite poset and $(Q,\leq)$ be a full subposet. Let $\lhd$ be an adapted poset to $A(Q)$. We denote by $\Delta$ its set of standard modules. We denote by $\widehat{\Delta}$ the set of standard modules for $A(P)$ corresponding to $\hl $. Then,
\begin{enumerate}
\item If $r\in P\setminus Q$, then $\widehat{\Delta}(r) = S(r)$. 
\item If $q\in Q$, then $\widehat{\Delta}(q) = L(\Delta(q))$. 
\end{enumerate}
\end{lem}
\begin{proof}
Let $r\in R$. If $x\hl  r$, then $x\in R$ and $x\leq r$. Since the composition factors of $P(r)$ are indexed by the elements larger than $r$ in $P$, we have the first point. 

Let $q\in Q$. We know that the composition factors of $\widehat{\Delta}(q)$ are indexed by the elements $x\in P$ such that $x\lhd q$ is a decreasing relation. Lemma \ref{decreasing} gives us the decreasing relations. It remains to compute the composition factors of $L(\Delta(q))$ and check that the two sets coincide. Recall that $\Delta(q)$ is the largest quotient of $P(q)$ having composition factors indexed by elements $q'$ such that $q'\lhd q$. Let $\mathfrak{m}$ be the set of $q'$ such that $q\leq q'$, $q'\centernot{\lhd} q$ and $q'$ is minimal for this property. Then,
\[
\bigoplus_{x\in \mathfrak{m}} P(x) \to P(q) \to \Delta(q) \to 0,
\]
is the beginning of a (minimal) projective presentation of $\Delta(q)$. Applying the functor $L$ we get that
\[
\bigoplus_{x\in \mathfrak{m}} P(x) \to P(q) \to L(\Delta(q)) \to 0,
\]
is the beginning of a projective presentation of $L(\Delta(q))$. Let us also recall that all the non-zero morphisms between projective indecomposable modules are injective. So, to compute the composition factors of $L(\Delta(q))$ we have to compute the composition factors of $P(q)$ which are not in $P(x)$ for $x\in \mathfrak{m}$. If $S(y)$ is a composition factor of $P(x)$ for $x\in \mathfrak{m}$, then $q\leq x\leq y $ and $x\centernot{\lhd} q$. Conversely we assume that there is $q_1\in Q$ such that $q\leq q_1 \leq y$ and $q_1\centernot{\lhd} q$. If $q_1$ is minimal for this property $q_1\in \mathfrak{m}$. If not, let $q_2$ such that $q\leq q_2 \leq q_1$ such that $q_2 \centernot{\lhd} q$ and $q_2$ is minimal for this property. Then $q_2\in \mathfrak{m}$ and $S(y)$ is a composition factor of $P(q_2)$. This shows that the composition factors of $L(\Delta(q))$ are indexed by the elements $x$ such that $q\leq x$ and $\forall q_1\in [q,x] \cap Q$, $q_1\lhd q$. The result now follows from Lemma \ref{decreasing}. 
\end{proof}

\begin{pro}\label{qh_embedding}
Let $(P,\leq)$ be a finite poset such that $A(P)$ is a hereditary algebra. Let $(Q,\leq)$ be a full subposet of $Q$. Then $\lhd \mapsto \hl $ induces a well-defined full embedding of posets from $\qhstr(A(Q))$ to $\qhstr(A(P))$. 
\end{pro}
\begin{proof}
Since the algebra $A(P)$ is hereditary, all the adapted orders give rise to quasi-hereditary structures. 

Lemma \ref{standard-lift} implies that the mapping is compatible with the equivalence relation. Lemma \ref{standard-lift} and $(-)e_Q\circ L \cong \operatorname{Id}_{\operatorname{mod}A(Q)}$ implies that it is injective. Lemma \ref{lem-surj-Dec} and right exactness of $L$ implies that it is a morphism of posets and $(-)e_Q\circ L \cong \operatorname{Id}_{\operatorname{mod}A(Q)}$ implies that it is full.
\end{proof}

\begin{dfn}
Let $\phi \colon  (O_1,\leq_1) \to (O_2,\leq_2)$ be a morphism of posets. Then $\phi$ is \emph{interval-preserving} if for every $x,y\in O_1$ and $z\in O_2$ we have $\phi(x)\leq_2 z\leq_2 \phi(y)$ if and only if there is $z'\in O_1$ such that $x\leq_1 z'\leq_1 y$ and $\phi(z')=z$. 

\end{dfn}
\begin{pro}\label{interval-preserving}
The morphism $[\lhd]\mapsto [\hl ]$ from $\qhstr(A(Q))$ to $\qhstr(A(P))$ of Proposition \ref{qh_embedding} is interval-preserving.
\end{pro}
\begin{proof}
Let $\lhd$ be an adapted order such that $[\lhd_1] \preceq [\lhd] \preceq [\lhd_2]$. We use the following notation: $\Delta_1,\nabla_1, \Delta,\nabla,\Delta_2,\nabla_2$ for the set of standard and costandard modules respectively induced by $[\lhd_1], [\lhd]$ and $[\lhd_2]$. 

Using Lemma \ref{lem-surj-Dec}, for $x\in P$ we have: 
\[
\Delta_2(x)\twoheadrightarrow \Delta(x) \twoheadrightarrow \Delta_1(x),
\]
and 
\[
\nabla_2(x)  \hookrightarrow \nabla(x) \hookrightarrow \nabla_1(x).
\]
In particular for $r\in R \coloneqq  P\setminus Q$, we have $\Delta(r) = S(r)$ and there is no composition factor of $\nabla_{r}$ indexed by an element of $Q$ since the elements of $Q$ are not smaller for $\hl _1$ that the elements of $R$.  

For $q\in Q$, the beginning of a minimal projective presentation of $\Delta(q)$ is 
\begin{equation}\label{propre} 
\bigoplus_{y\in \mathfrak{m}} P(y) \to P(q) \to \Delta(q)\to 0,
\end{equation}
where $\mathfrak{m}$ is the set of $y\in P$ such that $q\leq y$, $y\centernot{\vartriangleleft}q$ and $y$ is minimal (with respect to $\leq)$ for this property. Since $\lhd$ is an adapted order, we have $q\lhd y$ or there is $q < x < y$ such that $q\lhd x$ and $y\lhd x$. The second possibility contradicts the minimality of $y$, so we have $q\lhd y$. Moreover, by minimality of $y$, if $q < x < y$, we have $x\lhd q \lhd y$. This implies that the relation $q\lhd y$ is \emph{increasing}. In other words, the simple module $S(q)$ is a composition factor of $\nabla(y)$. As explained above the composition factors of the costandard modules indexed by the elements of $R$ are in $R$, so we must have $y\in Q$. Then, by the presentation (\ref{propre}) and an isomorphism $L(e_QA(P)e_Q) \cong  e_QA(P)$, we obtain $\Delta(q) = L(\Delta(q)e_Q)$. 
 
In order to finish the proof, we have to show that there is a quasi-hereditary structure on $A(Q)$ which has $\{\Delta(q)e_Q\ ;\ q\in Q\}$ as set of standard modules. The obvious candidate is $\lhd|_{Q}$ the restriction of the poset $\lhd$ to the elements of $Q$. For $q\in Q$, we denote by $\Delta'(q)$ the standard module induced by the poset $\lhd|_{Q}$. Then the beginning of a minimal projective presentation of $\Delta'(q)$ is 
\[
\bigoplus_{q'\in \mathfrak{m}'}P(q') \to P(q) \to \Delta'(q)\to 0,
\]
 where $\mathfrak{m}'$ is the set of $q'\in Q$ such that $q\leq q'$, $q'\centernot{\lhd} q$ and $q'$ is minimal (with respect to $\leq$) for this property. As explained above, this set equals $\mathfrak{m}$. Since the functor $(-)e_Q$ is exact, by applying it to the exact sequence (\ref{propre}), we get $\Delta'(q) = \Delta(q)e_Q$. 
 
 Now we check that $[\lhd|_Q]$ is a quasi-hereditary structure for $A(Q)$. For $q\in Q$, we let 
\[ 0\subset M_0 \subset M_1\subset \cdots \subset M_n = P(q)\] 
be a $\Delta$-filtration. Since the functor $(-)e_Q$ is exact we have:
\[
0 \subset M_0e_Q \subset M_1 e_Q \subset \cdots \subset M_n e_Q = P(q)
\]
For $i\in \{1,\dots , n\}$ we have $(M_{i}e_Q) /(M_{i-1}e_Q) \cong \Delta(x)$ for $x\in P$. If $x\in R$, we have $\Delta(x) = S(x)$ and $\Delta(x)e_Q = 0$. So $M_{i}e_Q = M_{i-1}e_{Q}$. In this case, we remove $M_{i-1}$ from the sequence of submodules. By induction, we end up with
\[
0\subset M_{i_0}e_Q \subset M_{i_1}e_Q \subset \cdots \subset M_{i_r}e_Q = P(q)
\]
$M_{i_{j}}e_Q/M_{i_{j-1}}e_Q \cong \Delta(q)e_Q$ for some $q\in Q$. This proves that $[\lhd|_{Q}]$ is a quasi-hereditary structure for $A(Q)$ and $[\lhd] = ([\widehat{\lhd|_{Q}}])$. 

Finally we have to check that $[\lhd_1] \leq [\lhd|_{Q}] \leq [\lhd_2]$. This is an easy consequence of Lemma  \ref{lem-surj-Dec} and Lemma \ref{lem-incidence-hereditary}
\end{proof}

We state the following basic lemma.
\begin{lem}\label{lem-full-int-lattice}
Let $X, Y$ be posets.
Assume that there exists a morphism of posets $\phi : X \to Y$ which is a full embedding and which is interval-preserving.
If there does not exist a join of $a, b \in X$ and $U=\{c \in X \mid a, b \leq c\}$ is a non empty finite set, then there does not exist a join of $\phi(a), \phi(b) \in Y$.
\end{lem}
\begin{proof}
Since a join of $a, b$ does not exist, $U$ has at least two minimal elements, we denote them by $c, d\in U$.
We have $\phi(i) \leq \phi(j)$ for $i\in \{a, b\}$ and $j\in\{c, d\}$.
Assume that $\phi(a)$ and $\phi(b)$ admits a join in $Y$, we denote the join by $e\in Y$.
we have $\phi(i) \leq e \leq \phi(j)$ for $i\in \{a, b\}$ and $j\in\{c, d\}$.
Since $\phi$ is a full embedding and interval-preserving, there exists $f\in X$ such that $i \leq f \leq j$ for $i\in \{a, b\}$ and $j\in\{c, d\}$.
This contradicts to the minimality of $c, d$ in $U$.
Therefore $Y$ does not admits join.
\end{proof}

\begin{thm}\label{thm-qh-not-lattice}
Let $(P,\leq)$ be a finite poset such that $A(P)$ is a hereditary algebra. Let $(Q,\leq)$ be a full subposet of $(P,\leq)$ such that $\qhstr(A(Q))$ is not a lattice. Then $\qhstr(A(P))$ is not a lattice. 
\end{thm}
\begin{proof}
By Propositions \ref{qh_embedding} and \ref{interval-preserving}, there is a poset morphism $\qhstr(A(Q)) \to \qhstr(A(P))$ which is a full embedding and interval-preserving.
It is easy to check that $\leq^{\rm op}$ induces a unique maximal quasi-hereditary structure of $A(Q)$.
Let $a, b \in \qhstr(A(Q))$ be elements such that there does not exist a join of them in $\qhstr(A(Q))$.
Since $\qhstr(A(Q))$ admits a unique maximal element, the set $U =\{c \in \qhstr(A(Q)) \mid a, b \preceq c\}$ is a non empty finite set.
Therefore by Lemma \ref{lem-full-int-lattice}, $\qhstr(A(P))$ is not a lattice.
\end{proof}

\if() 
\subsection{old writing}
\old{In this section, we consider when the poset $(\qhstr(A), \preceq)$ becomes a lattice for a finite dimensional quasi-hereditary algebra $A$.
We give a complete answer in the case where $A$ is the path algebra of a tree quiver (Theorem \ref{thm-lattice-tree}).
We try to prove the following conjecture.
\begin{conj}
Let $Q$ be a finite quiver such that the underline graph is a tree. Then the poset of quasi-hereditary structures on $\kk Q$ is a lattice if and only if $Q$ does not contains $\widetilde{D_4}$ as a subquiver.
\end{conj}
I will precise in the next lemma what I mean by $\widetilde{D_4}$.
\begin{lem}\label{lem-D4-not-lattice}
Let $Q$ be the following quiver.
\[
\xymatrix@=1em{
1\ar[rd] & & 4\\
&3\ar[ru]\ar[rd]&\\
2\ar[ru] && 5
}
\]
Then the poset of quasi-hereditary structures on $\kk Q$ is not a lattice.
\end{lem}
\begin{proof}
There are $45$ different quasi-hereditary structures on this poset, so we will not display it. In order to prove that it is not a lattice, we will consider two well chosen quasi-hereditary structures and show that they do not have a join. 

Let $\lhd_1$ be the total order defined by $3\lhd_1 1 \lhd_1 4\lhd_1 2 \lhd_1 5$ and $\lhd_2$ the total order defined by $3\lhd_2 2\lhd_2 5 \lhd_2 1\lhd_2 4$. These orders are adapted to $\kk Q$ since they are total orders and they have as respective sets of standards modules
\YKeditcomment{It seems that the definitions of $\lhd_1$ and $\lhd_2$ are typo.}
\BReditcomment{Yes it is correct now. Thank you. We can remove the lemma when we trust the next lemma}
\[
\stackMath
\Delta_1 = \{ \Delta_1(1) = \Shortstack{1 3},\ \Delta_1(2) = \Shortstack{2 3 4},\ \Delta_1(3)=3,\ \Delta_1(4)=4,\ \Delta_1(5)=5 \}
\]
and
\[
\stackMath
\Delta_2 = \{ \Delta_2(1) = \Shortstack{1 3 5},\ \Delta_2(2) = \Shortstack{2 3},\ \Delta_2(3)=3,\ \Delta_2(4)=4,\ \Delta_2(5)=5 \}.
\]
Here we represent the modules by their Loewy shape. This means for example that $\Delta_1(1)$ is the quotient of $P(1)$ by its socle. 

Since these two sets are not equal we see that the quasi-hereditary structures induced by $\lhd_1$ and $\lhd_2$ are not equal. 

Let $\Theta = \{\Theta(1)= \Shortstack{1 3 5},\ \Theta(2)=\Shortstack{2 3 4},\ \Theta(3) = 3, \Theta(4)=4,\ \Theta(5)=5\}.$ Then we claim that \[
\mathcal{F}(\Delta_1)\cap \mathcal{F}(\Delta_2) = \mathcal{F}(\Theta).
\]
The inclusion $\mathcal{F}(\Theta) \subseteq \mathcal{F}(\Delta_1)\cap \mathcal{F}(\Delta_2)$ is clear. We prove the other inclusion by induction on the length of the modules. Let $M\in \mathcal{F}(\Delta_1)\cap \mathcal{F}(\Delta_2)$. If $M$ is simple, it has to be $S(3)$, $S(4)$ or $S(5)$ and it is in $\mathcal{F}(\Theta)$. If $M$ has length $k\geq 1$. Since $M$ is in $\mathcal{F}(\Delta_1)$, there is a filtration 
\[
0 = M_0 \subset M_1 \subset \cdots \subset M_{n-1}\subset M_{n} = M
\]
such that the quotients are in $\Delta_1$. If $M/M_{n-1}\neq \Delta_1(1)$, then $M_{n}$ and $M_{n}/M_{n-1}$ are in $\mathcal{F}(\Delta_2)$ and since this category is stable by kernel of epimorphism, we see that $M_{n-1}\in \mathcal{F}(\Delta_1)\cap \mathcal{F}(\Delta_2)$. By induction we have $M_{n-1}\in \mathcal{F}(\Theta)$ and $M_{n}/M_{n-1}\in \mathcal{F}(\Theta)$. This shows that $M_n\in \mathcal{F}(\Theta)$. Now we assume that $M_n/M_{n-1} = \Delta_1(1)$. This implies in particular that the simple module $S(1)$ is at the top of $M$. Since $M\in \mathcal{F}(\Delta_2)$ and $S(1)$ is at the top of $M$, the module $\Delta_2(1)$ appears as a quotient of $M$ in a $\Delta_2$-filtration. More precisely, there exists $N\in\mathcal{F}(\Delta_2)$ such that $M/N\cong \Delta_2(1)$. We have $\Delta_2(1)\in \mathcal{F}(\Delta_1)$ and $M\in \mathcal{F}(\Delta_1)$ and this category is stable under kernel of epimorphism, so we see that $N\in \mathcal{F}(\Delta_1)\cap \mathcal{F}(\Delta_2)$. By induction we have $N\in \mathcal{F}(\Theta)$ and it follows that $M\in \mathcal{F}(\Theta)$. 

If $\lhd_1$ and $\lhd_2$ have a join $\widetilde{\lhd}$ with set of standard modules $\widetilde{\Delta}$, then $\lhd_1 \leq \widetilde{\lhd}$ and $\lhd_2 \leq \widetilde{\lhd}$ implies that 

\begin{equation}\label{eq-utile}
    \mathcal{F}(\widetilde{\Delta})\subseteq \mathcal{F}(\Delta_1)\cap \mathcal{F}(\Delta_2)=\mathcal{F}(\Theta).
\end{equation}

Let us consider two other quasi-hereditary structures on $\kk Q$. Let $\lhd_3$ be the total order defined by $3\lhd_3 4 \lhd_3 2 \lhd_3 5 \lhd_3 1 $ and $\lhd_4$ be the total order defined by $3\lhd_4 5\lhd_4 1\lhd_4 4\lhd_4 2 $. They have as respective sets of standards modules
\[
\stackMath
\Delta_3 = \{ \Delta_3(1) = \Shortstack{1 3 {45}},\ \Delta_3(2) = \Shortstack{2 3 4},\ \Delta_3(3)=3,\ \Delta_3(4)=4,\ \Delta_3(5)=5 \}
\]
and
\[
\stackMath
\Delta_4 = \{ \Delta_4(1) = \Shortstack{1 3 5},\ \Delta_4(2) = \Shortstack{2 3 {45}},\ \Delta_4(3)=3,\ \Delta_4(4)=4,\ \Delta_4(5)=5 \}.
\]
It is easy to check that $\lhd_1 \leq \lhd_3$, $\lhd_1 \leq \lhd_4$, $\lhd_2 \leq \lhd_3$ and $\lhd_2\leq \lhd_4$. 

Assume now that $\lhd_1$ and $\lhd_2$ have a join $\widetilde{\lhd}$. Then $\widetilde{\lhd}\leq \lhd_3$ and $\widetilde{\lhd}\leq \lhd_4$. This means that $\mathcal{F}(\Delta_3)\subseteq \mathcal{F}(\widetilde{\Delta})$ and $\mathcal{F}(\Delta_4)\subseteq \mathcal{F}(\widetilde{\Delta})$. This implies that the modules $\Shortstack{1 3 5}$ and $\Shortstack{2 3 4}$ are in $\mathcal{F}(\widetilde{\Delta})$, and we see that $\mathcal{F}(\Theta)\subseteq \mathcal{F}(\widetilde{\Delta})$. Together with $(\ref{eq-utile})$, this implies that $\mathcal{F}(\widetilde{\Delta}) = \mathcal{F}(\Theta)$. We prove that this implies that $\widetilde{\Delta} = \Theta$. The module $\widetilde{\Delta}(i)$ is $\Theta$-filtered, so it has a module $\Theta(j)$ has a quotient. Since $\widetilde{\Delta}(i)$ has simple top $S(i)$, we see that $\Theta(i)$ is a quotient of $\widetilde{\Delta}(i)$. Similarly the module $\theta(i)$ has simple top $S(i)$ and is $\widetilde{\Delta}$-filtered, so $\widetilde{\Delta}(i)$ is a quotient of $\theta(i)$ and this implies that $\widetilde{\Delta}(i)=\Theta(i)$.

We finish the proof by showing that $\Theta$ cannot be the set of standard modules for a partial order adapted to $\kk Q$. Let $\lhd$ be an adapted partial order with $\Theta$ as set of standard modules. The shape of $\Theta(3)$ implies that $4 \centernot{\lhd} 3$ and $5 \centernot{\lhd} 3$. The poset $\lhd$ is adapted, so we have $3\lhd 4$ and $3\lhd 5$. The shape of $\Theta(1)$ implies that $3\lhd 1$, $5\lhd 1$ and $4 \centernot{\lhd} 1$. Since the poset is adapted either $1\lhd 4$ or $1\lhd 3$ and $4\lhd 3$. The second possibility is not compatible with the fact that $3\lhd 4$, so we have $1\lhd 4$. By transitivity we have $5\lhd 4$. The shape of $\Theta(2)$ implies that $4\lhd 2$ and $5 \centernot{\lhd} 2$. But $4\lhd 2$ implies that $5\lhd 2$ contradicting the existence of such an adapted partial order. 
\end{proof}
\begin{lem}\label{lem-not-lattice-Dn}
Let $n\geq 4$ and $\widetilde{\rm D}_n^{(2,2)}$ be the quiver (\ref{quiver-D-tilde}) where the vertices are labeled from $1$ to $n+1$. The top left vertex has label $1$ the bottom left has label $2$ the top right $n$ and the bottom right has label $n+1$. Then the poset of quasi-hereditary structures on $\kk Q$ is not a lattice.
\end{lem}
\begin{proof}
Let us consider the following two total orders and let us show that they do not have a join in the poset of quasi-hereditary structures.

The first total order is defined by 
\[ 3\lhd_1 4 \lhd_1 \cdots \lhd_1 n-1 \lhd_1 1 \lhd_1 n \lhd_1 2\lhd_1 n+1.\]
The second total order is defined by
\[3\lhd_2 4 \lhd_2 \cdots \lhd_2 n-1 \lhd_2 2 \lhd_2 n+1 \lhd_2 1\lhd_2 n. \]
\BReditcomment{This is a straightforward generalization of the case $n=4$.}
The two sets of standard modules are easily computed. The only standard modules that are not simple are indexed by $1$ and $2$ and we have \[\Delta_1(1) = \Shortstack{1 3 {\vdots} n-1},\ \Delta_1(2)=\Shortstack{2 3 {\vdots} {n-1} n}, \Delta_2(1) = \Shortstack{1 3 {\vdots} {n-1} {n+1}},\ \Delta_2(2)=\Shortstack{2 3 {\vdots} n-1}.\]

ince these two sets are not equal we see that the quasi-hereditary structures induced by $\lhd_1$ and $\lhd_2$ are not equal. 

Let us define $\Theta(i)=S(i)$ for $i=3,\dots  n+1$, $\Theta(1)= \Delta_2(1)$ and $\Theta(2)= \Delta_1(2)$. We let $\Theta$ be the set consisting of the modules $\Theta(i)$ for $i=1,2,\dots , n+1$. 

Then we claim that \[
\mathcal{F}(\Delta_1)\cap \mathcal{F}(\Delta_2) = \mathcal{F}(\Theta).
\]
The inclusion $\mathcal{F}(\Theta) \subseteq \mathcal{F}(\Delta_1)\cap \mathcal{F}(\Delta_2)$ is clear. We prove the other inclusion by induction on the length of the modules. Let $M\in \mathcal{F}(\Delta_1)\cap \mathcal{F}(\Delta_2)$. If $M$ is simple, it has to be $S(3),S(4),
\dots , S(n)$ or $S(n+1)$ and it is in $\mathcal{F}(\Theta)$. If $M$ has length $k\geq 1$. Since $M$ is in $\mathcal{F}(\Delta_1)$, there is a filtration 
\[
0 = M_0 \subset M_1 \subset \cdots \subset M_{n-1}\subset M_{n} = M
\]
such that the quotients are in $\Delta_1$. If $M/M_{n-1}\neq \Delta_1(1)$, then $M_{n}$ and $M_{n}/M_{n-1}$ are in $\mathcal{F}(\Delta_2)$ and since this category is stable by kernel of epimorphism, we see that $M_{n-1}\in \mathcal{F}(\Delta_1)\cap \mathcal{F}(\Delta_2)$. By induction we have $M_{n-1}\in \mathcal{F}(\Theta)$ and $M_{n}/M_{n-1}\in \mathcal{F}(\Theta)$. This shows that $M_n\in \mathcal{F}(\Theta)$. Now we assume that $M_n/M_{n-1} = \Delta_1(1)$. This implies in particular that the simple module $S(1)$ is at the top of $M$. Since $M\in \mathcal{F}(\Delta_2)$ and $S(1)$ is at the top of $M$, the module $\Delta_2(1)$ appears as a quotient of $M$ in a $\Delta_2$-filtration. More precisely, there exists $N\in\mathcal{F}(\Delta_2)$ such that $M/N\cong \Delta_2(1)$. We have $\Delta_2(1)\in \mathcal{F}(\Delta_1)$ and $M\in \mathcal{F}(\Delta_1)$ and this category is stable under kernel of epimorphism, so we see that $N\in \mathcal{F}(\Delta_1)\cap \mathcal{F}(\Delta_2)$. By induction we have $N\in \mathcal{F}(\Theta)$ and it follows that $M\in \mathcal{F}(\Theta)$. 

If $\lhd_1$ and $\lhd_2$ have a join $\widetilde{\lhd}$ with set of standard modules $\widetilde{\Delta}$, then $\lhd_1 \leq \widetilde{\lhd}$ and $\lhd_2 \leq \widetilde{\lhd}$ implies that 

\begin{equation}\label{eq-utile}
    \mathcal{F}(\widetilde{\Delta})\subseteq \mathcal{F}(\Delta_1)\cap \mathcal{F}(\Delta_2)=\mathcal{F}(\Theta).
\end{equation}

Let us consider two other quasi-hereditary structures for $\kk Q$. Let $\lhd_3$ and $\lhd_4$ be the two total orders defined by 
\[ 3\lhd_3 \cdots \lhd_3 n-1 \lhd_3 n \lhd_3 2 \lhd_3 n+1 \lhd_3 1 \]
and 
\[ 3\lhd_4 \cdots \lhd_4 n-1 \lhd_4 n+1 \lhd_4 1\lhd_4 n \lhd_4 2 .\] 
As before the only standard modules which are not simple are indexed by $1$ and $2$ and we have $\Delta_3(1)=P(1)$, $\Delta_3(2)=\Theta(2)$, $\Delta_4(1) = \Theta(1)$ and $\Delta_4(2) = P(2)$. 

It is easy to check that $\lhd_1 \leq \lhd_3$, $\lhd_1 \leq \lhd_4$, $\lhd_2 \leq \lhd_3$ and $\lhd_2\leq \lhd_4$. 

Assume now that $\lhd_1$ and $\lhd_2$ have a join $\widetilde{\lhd}$. Then $\widetilde{\lhd}\leq \lhd_3$ and $\widetilde{\lhd}\leq \lhd_4$. This means that $\mathcal{F}(\Delta_3)\subseteq \mathcal{F}(\widetilde{\Delta})$ and $\mathcal{F}(\Delta_4)\subseteq \mathcal{F}(\widetilde{\Delta})$. This implies that the modules $\Theta(1)$ and $\Theta(2)$ are in $\mathcal{F}(\widetilde{\Delta})$, and we see that $\mathcal{F}(\Theta)\subseteq \mathcal{F}(\widetilde{\Delta})$. Together with $(\ref{eq-utile})$, this implies that $\mathcal{F}(\widetilde{\Delta}) = \mathcal{F}(\Theta)$. We have to check that this implies $\widetilde{\Delta} = \Theta$. The module $\widetilde{\Delta}(i)$ is $\Theta$-filtered, so it has a module $\Theta(j)$ has a quotient. Since $\widetilde{\Delta}(i)$ has simple top $S(i)$, we see that $\Theta(i)$ is a quotient of $\widetilde{\Delta}(i)$. Similarly the module $\Theta(i)$ has simple top $S(i)$ and is $\widetilde{\Delta}$-filtered, so $\widetilde{\Delta}(i)$ is a quotient of $\theta(i)$ and this implies $\widetilde{\Delta}(i)=\Theta(i)$.

We finish the proof by showing that $\Theta$ cannot be the set of standard modules for a partial order adapted to $\kk Q$. Let $\lhd$ be an adapted partial order with $\Theta$ as set of standard modules. The shape of $\Theta(1)$ implies that $i \lhd 1$ for $i \in \{3,\dots ,n-1,n+1\} $ and $n \centernot{\lhd} 1$. Since the partial order is adapted to $KQ$, we see that either $1\lhd n$ or there is $k\in \{3,\cdots,n-1\}$ such that $1\lhd k$ and $n\lhd k$. The second possibility contradicts $k\lhd 1$, so we have $1\lhd n$ and by transitivity $n+1\lhd n$. 

Let us look at the shape of $\Theta(2)$ we see that $3,4,\cdots n-1, n$ are smaller than $2$ and $n+1$ is not smaller than $2$ for $\lhd$. Since we have $n\lhd 2$, by transitivity we also have $n+1\lhd 2$ contradicting the existence of such an adapted partial order. 
\end{proof}
Let $n\in \mathbb{N}$ be an even integer and $Z_n$ be the following zigzag orientation of an affine Dynkin diagram of type $\A$.
\begin{align}\label{zigzag}
\begin{tikzpicture}[baseline=0,
                    vertex/.style args = {#1 #2}{
                                                label=#1:#2}]
\node(0)at(0,0){$\underset{\circ}{1}$};
\node(1)at(1,0){$\underset{\circ}{2}$};
\node(2)at(1.75,-0.87){$\underset{\circ}{3}$};
\node(3)at(1.75,-1.87)[label={[xshift=0.3cm, yshift=-0.5cm]$4$}] {$\circ$};
\node(nm2)[label={[xshift=-0.6cm, yshift=-0.5cm]$n-2$}]at(0,-2.6){$\circ$};
\node(nm1)[label={[xshift=-0.6cm, yshift=-0.5cm]$n-1$}]at(-0.75,-1.87){$\circ$};
\node(n)at(-0.75,-0.87){$\underset{\circ}{n}$};
\draw[thick, ->] (0)--(1);
\draw[thick, ->] (2)--(1);
\draw[thick, ->] (2)--(3);
\draw[thick, ->] (nm1)--(nm2);
\draw[thick, ->] (nm1)--(n);
\draw[thick, ->] (0)--(n);
\draw[thick,dotted] (3) to [out=250,in=325](nm2);
\end{tikzpicture}
\end{align}
\begin{pro}
Let $n\geq 4$ be an even integer. Let $K$ be a field and $Z_n$ the zigzag orientation of the affine Dynkin diagram defined in $(\ref{zigzag})$. Then the poset of quasi-hereditary structures on $KZ_n$ is not a lattice. 
\end{pro}
\begin{proof}
We assume that the quiver is labeled as in $(\ref{zigzag})$. This means that the sources of the quiver are labeled by even integers and the sinks by even integers. Moreover we view the labels as elements of $\mathbb{Z}/n\mathbb{Z}$ in order to have a cyclic labeling.

The projective indecomposable $KZ_n$-modules indexed by sinks are simple and the projective indecomposable module indexed by a source $i$ has top $S(i)$ and socle $S(i-1)\oplus S(i+1)$. 

Let $x$ and $y$ be the labels of two distinct sources of $Z_n$. For $s\in \{x,y\}$ we consider the total order $\lhd_s$ defined by  
\[
s\lhd_s s+1\lhd_s \cdots \lhd_s s-1
\]
These two orderings are total, so they are adapted and since the algebra is hereditary they represent two quasi-hereditary structures. For $s\in\{x,y\}$, we denote by $\Delta_s$ the corresponding set of standard modules. The standard modules indexed by sinks are simple, the standard module indexed by $s$ is also simple and the standard modules indexed by the other sources $v$ are uniserial with top $S(v)$ and socle $S(v-1)$. Since $x$ and $y$ are distinct, we see that the orderings $\lhd_x$ and $\lhd_y$ represent two different quasi-hereditary structures.

For a sink $i$ we let $\theta(i) = S(i)$ and for a source $i$ we let $\theta(i)$ be the uniserial module with top $S(i)$ and socle $S(i-1)$ and we denote by $\Theta$ the set consisting of the modules $\theta(i)$ for $i\in \{1,2,\cdots,n\}$. 

For $s\in \{x,y\}$, we have $\mathcal{F}(\Theta)\subseteq \mathcal{F}(\Delta)$. Indeed, for $i\neq s$ we have $\Delta_s(i) = \theta(i)$ and $\theta(s)$ clearly has a $\Delta_s$-filtration. Conversely, we show by induction on the length of $M$ that if $M\in \mathcal{F}(\Delta_x)\cap \mathcal{F}(\Delta_y)$, then $M\in \mathcal{F}(\Theta)$. This is clearly true if $M$ is simple. Let 
\[
0 = M_0\subset M_1\subset \cdots \subset M_{n-1} \subset M_n = M
\]
be a $\Delta_x$-filtration of $M$. Then we have $M/M_{n-1} \cong \Delta_x(i)$ for some $i \in \{1,2,\cdots,n\}$. If $i\neq x$ then $\Delta_x(i) \in \mathcal{F}(\Delta_y)$. Since $M$ is also in $\mathcal{F}(\Delta_y)$ and this category is closed under kernel of epimorphism, we obtain that $M_{n-1} \in \mathcal{F}(\Delta_x)\cap \mathcal{F}(\Delta_y)$. By induction, the module $M_{n-1}$ is in $\mathcal{F}(\Theta)$. Since the category $\mathcal{F}(\Theta)$ is closed under extension and $\Delta_x(i)\in\mathcal{F}(\Theta)$, we conclude that $M\in \mathcal{F}(\Theta)$. If $i = x$, the $M$ has the simple module $S(x)$ as direct summand of its top. So there is a $\Delta_y$-filtration of $M$ where the standard module $\Delta_y(x)$ appears as first factor. That is, there is $N\in \mathcal{F}(\Delta_y)$ such that $M/N\cong \Delta_y(x)$. Since $x\neq y$, the module $\Delta_y(x)$ is $\Delta_x$-filtered and we use the same argument as above to conclude that $M\in \mathcal{F}(\Theta)$. 

Let us remark that there is no adapted order that have $\Theta$ as set of standard modules. Indeed, let us assume that there is such an ordering $\lhd$. The shape of the modules implies that if $i$ is an odd integer, we have $i-1\lhd i$ and $i+1\centernot{\lhd} i$. Since $\lhd$ is adapted, we must have $i\lhd i+1$. It follows that $n\lhd 1$ and $1\lhd 2 \lhd 3\lhd \cdots\lhd n$ contradicting the antisymmetry of $\lhd$. 

Let us finish the proof by showing that if the join $\lhd_J$ of $\lhd_x$ and $\lhd_y$ in the poset of quasi-hereditary structures exists, it must have $\Theta$ as set of standard modules, contradicting the fact that no adapted order have this set as standard modules. 

For this, we consider $z = x+1$ and $w = y+1$ the sinks that are respectively after $x$ and $y$ in the cyclic ordering of $Z_n$. As before the cyclic ordering gives two total orderings $\lhd_z$ and $\lhd_t$ where the minimal elements are respectively $z$ and $t$ and the maximal elements are respectively $x$ and $y$. For $s\in \{z,t\}$ the standard modules $\Delta_s$ are as follows: the standard modules indexed by sinks are simple, the standard module indexed by the source $s-1$ (which is respectively $x$ and $y$) are projective and the standard modules indexed another source $i$ are uniserial with top $S(i)$ and socle $S(i-1)$.  It is straightforward to check that $\mathcal{F}(\Delta_s)\subseteq \mathcal{F}(\Delta_x)\cap \mathcal{F}(\Delta_y)$ for $s\in\{z,t\}$.

Let us now assume that $\lhd_x$ and $\lhd_y$ have a join $\lhd_J$ in the poset of quasi-hereditary structures. Then we have
\[
 \mathcal{F}(\Delta_z)\cup\mathcal{F}(\Delta_t)\subseteq \mathcal{F}(\Delta_J) \subset \mathcal{F}(\Delta_x)\cap\mathcal{F}(\Delta_y) =\mathcal{F}(\Theta).
\]
Since every module $\theta(i)$ appears in $\Delta_z$ or $\Delta_t$, we have $\mathcal{F}(\Theta) = \mathcal{F}(\Delta_J)$. For $i\in\{1,2,\cdots, n\}$, we see $\theta(i)\in \mathcal{F}(\Delta_J)$ implies that $\Delta_J(i)$ is a quotient of $\theta(i)$ since the only module in $\Theta$ with top $S(i)$ is $\theta(i)$. Similarly $\Delta_J(i) \in \mathcal{F}(\Theta)$ implies that $\theta(i)$ is a quotient of $\Delta_J(i)$. So, $\Theta = \Delta_J$ and if the join of $\lhd_x$ and $\lhd_y$ exists, it has $\Theta$ as set of standard modules.  
\end{proof}
} 
\fi 

\subsection{Quasi-hereditary structures of diamond-free posets}\label{subsection-lattice-tree}
In this subsection we prove the reciprocity of Theorem \ref{thm-qh-lattice-poset}. 
\begin{thm}\label{thm-lattice-tree}
Let $(P,\leq)$ be a finite poset such that $A(P)$ is a hereditary algebra.
If $Z_{n}$ is not a full subposet of $(P,\leq)$ for any $n\geq 4$, then $\qhstr(A(P))$ is a lattice and the join and the meet of two quasi-hereditary structures $[\lhd_1]$ and $[\lhd_2]$ are represented by the following partial orders
\begin{align*}
\lhd_1 \wedge \lhd_2 = \Bigl(\bigl(\Dec(\lhd_1) \cap \Dec(\lhd_2)\bigr) \cup \Inc(\lhd_1) \cup \Inc(\lhd_2)\Bigr)^{\sf tc},\\
\lhd_1 \vee \lhd_2 = \Bigl(\Dec(\lhd_1) \cup \Dec(\lhd_2) \cup \bigl(\Inc(\lhd_1) \cap \Inc(\lhd_2)\bigr)\Bigr)^{\sf tc}.
\end{align*}
\end{thm}
Our strategy for the proof is as follows: first we prove that the candidate for the meet is a partial order (this is the technical part, Proposition \ref{pro-deltap-antisymmetric}) and then we prove that it induces a quasi-hereditary structure and check that it gives a meet of $[\lhd_1]$ and $[\lhd_2]$ in the poset of quasi-hereditary structures (Proposition \ref{pro-deltap-meet}).
Since $\qhstr(A(Q))$ is a finite poset with a greatest element, this is enough to prove that it is a lattice. 

For the rest of this subsection, we fix a finite poset $(P,\leq)$ and fix two adapted orders $\lhd_1$ and $\lhd_2$ on $P$.
Let
\[
\lhdp = \Bigl(\bigl(\Dec(\lhd_1) \cap \Dec(\lhd_2)\bigr) \cup \Inc(\lhd_1) \cup \Inc(\lhd_2)\Bigr)^{\sf tc}.
\]
We denote by $\Delta_1$, $\Delta_2$ the set of standard $A(P)$-modules associated to $\lhd_1$ and $\lhd_2$, respectively.
For $i, j \in P$, we write $i \lhd^D j$ if $(i,j)\in\Dec(\lhd_1) \cap \Dec(\lhd_2)$ and write $i\lhd^{I_m}j$ if $(i,j)\in\Inc(\lhd_m)$ for $m=1,2$.

For $i,j\in P$ we have $j\lhd^D i$ if and only if $i\leq j$ and every $k\in [i,j]$ satisfies $k\lhd_1 i$ and $k\lhd_2 i$.
Similarly, for $m=1,2$, $j\lhd^{I_m}i$ holds if and only if $i\leq j$ and every $k\in [i,j]$ satisfies $k\lhd_m i$.

The following lemma is obvious, but it is important to realize  that without the diamond-free hypothesis it is false.
\begin{lem}\label{dec-transitive}
Let $(P,\leq)$ be a diamond-free poset. Then $\lhd^{D}$ is a transitive relation.
\end{lem}
\begin{proof}
Assume that $k\lhd^Dj\lhd^Di$ for $i\leq j \leq k \in P$. 
Let $x\in [i,k]$.
Since $P$ is diamond-free, $x$ is either smaller than $j$ or larger than $j$. 
If $x$ is smaller than $j$, then $x \lhd_m i$ holds for $m=1,2$.
If $x$ is larger than $j$, then $x \lhd_m j \lhd_m i$ holds for $m=1,2$.
This implies $k \lhd^D i$.
\end{proof}

\begin{lem}\label{lem-deltad-deltai}
Let $i,j \in P$.
Assume that $i\leq j$ and any $k\in [i,j]\setminus\{j\}$ satisfies $k\lhd^Di$.
Then one of $j\lhd^Di$, $i\lhd^{I_1}j$ or $i\lhd^{I_2}j$ holds.
\end{lem}
\begin{proof}
We used the argument that we proved in the proof of Proposition \ref{interval-preserving}: if $j$ is minimal for $\leq$ such that $j\centernot{\lhd_s}i$, then $i\lhd_1 j \in \Inc(\lhd_s)$ for $s = 1,2$. So if $j\centernot\lhd^Di$, then either $j\centernot{\lhd_1}i$ or $j\centernot{\lhd_2}i$ holds. In the first case we have $i\lhd^{I_1}j$ holds and in the second case $i\lhd^{I_2} j$ holds. 
\end{proof}
\begin{lem}\label{lem-I-sequence}
Assume that there exists a sequence of elements of $P$ such that 
$$
i_0 \lhd^{I_{x(0)}} i_1 \lhd^{I_{x(1)}} \cdots \lhd^{I_{x(\ell-1)}} i_{\ell},
$$
where $x(k)\in\{1, 2\}$ for $k=0,\ldots,\ell-1$.
If $i_{\ell}\lhd^D i_0$ holds, then $i_0=i_1=\cdots=i_{\ell}$ holds.
\end{lem}
\begin{proof}
Since $i_{\ell}\lhd^{D} i_0$ holds, we have $i_{0}\leq i_{\ell}$.
On the other hand, we have $i_{0} \leq i_1\leq \cdots \leq i_{\ell}$. Since $i_1 \in [i_0,i_{\ell}]$ we have $i_1\lhd_1 i_0$ and $i_{1}\lhd_2 i_0$. By antisymmetry of $\lhd_1$ and $\lhd_2$, we have $i_1 = i_0$. 
By induction on $\ell$, we have the assertion.
\end{proof}
In Lemma \ref{lem-path-exist} and Proposition \ref{pro-deltap-antisymmetric}, we show that $\lhd^\p$ is a partial order on $P$ if $Z_{n}$ is not a full subposet of $(P,\leq)$ for any $n\geq 4$.
\begin{lem}\label{lem-path-exist}
Let $(P,\leq)$ be a finite poset and $\ell$ be an odd integer with $\ell\geq 3$.
Assume that $(P,\leq)$ does not have $Z_n$ as a full subposet for any even $n\geq 4$.
If there exist relations in $P$ of the form:
$$
i_0 \geq i_1 \leq i_2 \geq \cdots \leq i_{\ell-1} \geq i_{\ell} \leq i_0,
$$
then there exists $k\in \{0,1,\dots,\ell\}$ such that we have at least one of the following $i_k\leq i_{k+2}$, $i_{k+2}\leq i_{k}$ where $k$ considered modulo $\ell+1$.
\end{lem}
\begin{proof}
In this proof the labels are considered modulo $\ell +1$. We prove the result by induction on $\ell\geq 3$. 

We need a precise inductive hypothesis: let $Z = \{a_{0},\dots, a_{n}\}$ be a subposet consisting of $n+1$ consecutive elements of $ \{i_0,\dots,i_{\ell} \}$ and which induces a subposet isomorphic to $Z_{n+1}$. Then, there is $k\in \{0,1,\dots n+1\}$ such that $a_{k}\leq a_{k+2}$ or $a_{k+2} \leq a_{k}$ and $\{a_k,a_{k+1},a_{k+2}\}$ is a set of three consecutive elements of $\{i_0,\dots,i_{\ell} \}$.

If $n =3$, then we have a subposet isomorphic to $Z_{4}$ which by hypothesis is not full. So, we must have a relation between $a_{0}$ and $a_{2}$ or a relation between $a_{1}$ and $a_{3}$. 

We assume now that $n \geq 5$ and that we have a subposet $Z = \{a_{0},\dots, a_{n}\}$ isomorphic to $ Z_{n +1}$. Without loss of generality we can assume that $a_{i} \leq a_{i+1}$ when $i$ is odd. Since it is not full, there are two indices $k$ and $j$ with $k\notin \{j-1,j,j+1\}$ and a relation $a_{k} \leq a_{j}$ in $(P,\leq)$.   If $k$ is even, then $a_{k+1} \leq a_{k} \leq a_{j}$ and $a_{k-1} \leq a_{k} \leq a_{j}$. Moreover at least one of $a_{k+1}$ and $a_{k-1}$ is not in the set $\{a_{j-1},a_{j}, a_{j+1}\}$. Replacing $a_{k}$ by this element, we may assume $k$ to be odd. Similarly, if $j$ is even, then we have $a_{k}\leq a_{j} \leq a_{j+1}$ and $a_{k}\leq a_{j} \leq a_{j-1}$ and for at least one $j' \in \{j-1,j+1\}$, we have $k\notin \{j'-1,j',j'+1\}$. Changing $a_{j}$ by $a_{j'}$ we may assume that $j'$ is even. 

The relation $a_{k} \leq a_{j}$ splits $Z$ into two subposets $Z_1$ and $Z_2$ which are isomorphic to $Z_{m_1}$ and $Z_{m_2}$ for some integers $m_1$ and $m_2$. By assumption $k$ is odd and $j$ is even, so we have $4\leq m_1 < n$ and $4\leq m_2 < n$. Precisely, $Z_1$ is the set consisting of the elements after or equal to $a_{k}$ and before or equal to $a_{j}$ in the cyclic ordering and $Z_2$ consists of the elements after or equal to $a_{j}$ and before or equal $a_{k}$ in the cyclic ordering.  One of $Z_{1}$ or $Z_{2}$ satisfies the hypothesis of consecutive elements and by induction we have the result. 
\end{proof}
\begin{pro}\label{pro-deltap-antisymmetric}
Let $(P,\leq)$ be a diamond-free poset. Assume that $(P,\leq)$ does not have $Z_{n}$ as a full subposet.
Then for any two adapted orders $\lhd_1$ and $\lhd_2$ on $P$, a binary relation
\[
\lhdp = \Bigl(\bigl(\Dec(\lhd_1) \cap \Dec(\lhd_2)\bigr) \cup \Inc(\lhd_1) \cup \Inc(\lhd_2)\Bigr)^{\sf tc}.
\]
is a partial order on $P$.
\end{pro}
\begin{proof}
By the definition, $\lhdp$ is reflexive and transitive.
We show that $\lhdp$ is antisymmetric.
It is enough to show the following claim:
if there a sequence of vertices of $P$ such that
\begin{align}\label{sequence-original}
i_0 \lhd^{X_0} i_1 \lhd^{X_1} \cdots \lhd^{X_{\ell-1}} i_{\ell} = i_0,
\end{align}
where $X_k\in\{D, I_1, I_2\}$ for $k=0,\ldots,\ell-1$, then we have $i_0=i_1=\cdots=i_{\ell-1}$.

By Lemma \ref{dec-transitive} $\lhd^D$ is transitive, so we may assume that there is no $k$ with $X_k = X_{k+1}=D$.
Let $d$ be the number of $k=0,\ldots,\ell-1$ such that $X_k = D$.
We prove the claim by an induction on $d$.
Assume that $d=0$, then we have $i_0 \leq i_1\leq \cdots \leq i_{\ell-1} \leq i_0$, so $i_0=i_1=\cdots=i_{\ell-1}$.
Assume that $d=1$.
Without loss of generality, we may assume that $X_0=D$.
By applying Lemma \ref{lem-I-sequence} to $i_1 \lhd^{X_1} \cdots \lhd^{X_{\ell-1}} i_{\ell} = i_0$, we have $i_0=i_1=\cdots=i_{\ell-1}$.

Assume that $d>1$.
Without loss of generality, we may assume that $X_0=D$.
There exists a function $\delta \colon  \{0,1,\ldots,d-1\} \to \{0,\ldots,\ell-1\}$ such that $X_{\delta(k)}=D$, $\delta$ is strictly increasing and $\delta(0)=0$.
Then we have a sequence
$$
i_{0} \lhd^D i_{1} \lhd^{X_1} \cdots \lhd^{X_{\delta(k)-1}} i_{\delta(k)} \lhd^D i_{\delta(k)+1} \lhd^{X_{\delta(k)+1}} \cdots \lhd^{X_{\ell-1}} i_{\ell} = i_0,
$$
and this sequence induces the following relations in $(P,\leq)$.
$$
i_{0} \geq i_{1} \leq i_{\delta(1)} \geq i_{\delta(1)+1} \leq \cdots \geq i_{\delta(d-1)+1} \leq i_{\ell} =i_0.
$$
Since there is no $k$ with $X_k=X_{k+1}=D$, and $d>1$, we have $\delta(d-1)+1\geq 3$.
By applying Lemma \ref{lem-path-exist}, without loss of generality, there is at least one of the following relation:
\[
i_0 \leq i_{\delta(1)}, \quad i_{\delta(1)} \leq i_0, \quad i_1 \leq i_{\delta(1)+1} , \quad i_{\delta(1)+1} \leq i_1.
\]

If $i_0 \leq i_{\delta(1)}$ holds, then $i_0 \in [i_1, i_{\delta(1)}]$ holds since $(P,\leq)$ is a diamond-free poset.
Then there exists an integer $ m \in \{1,\cdots, \delta(1)-1\}$ such that $i_0 \in [i_m, i_{m+1}]$.
Since $i_m\lhd^{X_m}i_{m+1}$ and $X_m\in\{I_1,I_2\}$, $i_0 \lhd^{X_m}i_{m+1}$ holds.
By removing vertices $i_1,i_2,\ldots,i_m$ from (\ref{sequence-original}), we have a sequence
$$
i_0 \lhd^{X_m}i_{m+1}\lhd^{X_{m+1}}\cdots\lhd^{X_{\ell-1}} i_\ell =i_0.
$$
In this sequence, the number of $k$ such that $X_k=D$ is $d-1$.
Thus by inductive hypothesis, we have $i_0=i_{m+1}=\cdots = i_{\ell-1}$.
Applying Lemma \ref{lem-I-sequence} to the sequence $i_1 \lhd^{X_1}i_2\lhd^{X_2}\cdots \lhd^{X_m} i_{m+1} =i_0$, we have the claim.

If we have the relation $i_{\delta(1)} \leq i_0$, then the elements $i_2, i_3,\ldots,i_{\delta(1)}$ appear in $[i_1,i_0]$.
In particular, $i_{\delta(1)} \lhd^D i_1$ holds.
By applying Lemma \ref{lem-I-sequence} to a sequence $i_1 \lhd^{X_1} i_2 \lhd^{X_2} \cdots \lhd^{X_{\delta(1)-1}} i_{\delta(1)}$, we have $i_1 = i_2 = \cdots = i_{\delta(1)}$.
By identifying $i_1 = i_2 = \cdots = i_{\delta(1)}$ in the sequence (\ref{sequence-original}), since $\lhd^D$ is transitive and by an inductive hypothesis, we have the claim.

If we have the relation $i_1 \leq i_{\delta(1)+1}$, then there exists an integer $m \in \{1,\dots,\delta(1)-1\}$ such that $i_{\delta(1)+1} \in [i_m,i_{m+1}]$.
We have $i_{\delta(1)+1}\lhd^{X_m}i_{m+1}$.
By applying Lemma \ref{lem-I-sequence} to a sequence
$$
i_{\delta(1)+1} \lhd^{X_m} i_{m+1} \lhd^{X_{m+1}} \cdots \lhd^{X_{\delta(1)-1}} i_{\delta(1)},
$$
we have $i_{m+1}=\cdots=i_{\delta(1)}=i_{\delta(1)+1}$.
By an inductive hypothesis, we have the claim.

If we have the relation $i_{\delta(1)+1}\leq i_1$, then $i_1 \in [i_{\delta(1)+1},i_{\delta(1)}]$.
In particular, $i_1\lhd^Di_{\delta(1)+1}$ holds.
Then by an inductive hypothesis, we have the claim.

We proved the claim and it induces that $\lhdp$ is antisymmetric.
\end{proof}

We show that the partial order $\lhd^\p$ gives a meet of $[\lhd_1]$ and $[\lhd_2]$.
\begin{lem}\label{lem-comp-fac-deltap}
Let $(P,\leq)$ be a diamond-free poset and $i,j\in P$.
Assume that $\lhdp$ is a partial order on $P$.
Let $\Deltap$ be the set of standard $A(P)$-modules associated to $\lhdp$.
Then $S(j)$ is a composition factor of $\Deltap(i)$ if and only if $S(j)$ is a composition factor of both $\Delta_1(i)$ and $\Delta_2(i)$.
\end{lem}
\begin{proof}
The if part is easy.
If $j=i$, then the assertion is clear.
So let $j\neq i$.
Assume that $S(j)$ is a composition factor of $\Deltap(i)$.
Then $i\leq j$ and any $k\in[i,j]$ satisfies $k\lhdp i$.
Let $m$ be the maximal element $k\in[i,j]$ such that $i_{k^{\prime}}\lhd^D i$ holds for any $i \leq k^{\prime} \leq k$.
If $k=j$, then $S(j)$ is a composition factor of both $\Delta_1(i)$ and $\Delta_2(i)$.

Suppose that $k\neq j$.
By Lemma \ref{lem-deltad-deltai}, the minimal element $l$ in $[k,j]$ satisfies either $i\lhd^{I_1}l$ or $i\lhd^{I_2}l$.
This implies $i=l$, since $\lhdp$ is antisymmetric.
This is a contradiction, since $i$ never equals to $l$.
Therefore we have $k=j$ and the assertion holds.
\end{proof}
\begin{lem}\label{lem-delta12-filtered}
Assume that $(P,\leq)$ is a diamond-free poset and that $\lhdp$ is a partial order on $P$.
Let $\Deltap$ be the set of standard $A(P)$-modules associated to $\lhdp$.
Then $\mcF(\Delta_1) \cup \mcF(\Delta_2) \subset \mcF(\Deltap)$ holds.
\end{lem}
\begin{proof}
We show that $\Delta_1(i)$ belongs to $\mcF(\Deltap)$ for any $i\in P$.
It is enough to show that for a submodule $K$ of $\Delta_1(i)$, if $\Delta_1(i)/K$ belongs to $\mcF(\Deltap)$, then there exists a surjective morphism from $K$ to a product of some $\Deltap(j)$'s.
Assume that $L$ is an indecomposable direct summand of $K$.
Since $A(P)$ is hereditary, $L$ has a simple top $S(j)$ for some $j\in P$.
Let $k\in P$.
We show that $S(k)$ is a composition factor of $L$ if $S(k)$ is a composition factor of $\Deltap(j)$.
Since $A(P)$ is hereditary, we have that $S(k)$ is a composition factor of $L$ if and only if $S(k)$ is a composition factor of $\Delta_1(i)$ and $j\leq k$ holds. 

Since $L$ is a submodule of $\Delta_1(i)$, $i\leq j$ holds and any $u\in[i,j]$ satisfies $u\lhd_1 i$.
Assume that $S(k)$ is a composition factor of $\Deltap(j)$.
Then $j\leq k$ holds and any $v\in[j,k]$ satisfies $v\lhdp j$.
By Lemma \ref{lem-comp-fac-deltap}, such $v$ satisfies $v\lhd_1 j$.
Thus for the interval $[i,k]$, we have that any $u\in[i,k]$ satisfies $u\lhd_1 i$.
Thus $S(k)$ is a composition factor of $\Delta_1(i)$, and therefore $S(k)$ is a composition factor of $L$.
\end{proof}
\begin{pro}\label{pro-deltap-meet}
Let $(P,\leq)$ be a diamond-free poset.
For given two quasi-hereditary structures $[\lhd_1]$ and $[\lhd_2]$ on $A(P)$, assume that $\lhdp$ is a partial order on $P$.
Let $\Deltap$ be the set of standard $A(P)$-modules associated to $\lhdp$.
Then we have the following statements.
\begin{enumerate}
\item
The partial order $\lhdp$ induces a quasi-hereditary structure on $A(P)$.
\item
The meet of $[\lhd_1]$ and $[\lhd_2]$ in $\qhstr(A(P))$ exists and is represented by $\lhdp$.
\end{enumerate}
\end{pro}
\begin{proof}
(1)
We show that $(A(P),(P, \lhdp))$ satisfies (1), (2) and (3) of Definition \ref{dfn-qh-algebra}.
By the definition and Lemma \ref{lem-delta12-filtered}, (1) and (2) hold.
Clearly $(P(i):\Deltap(i))=1$ holds for any $i\in P$.
Assume that $(P(i):\Deltap(j))\neq 0$ for $j\leq i\in P$.
We show $i\lhd^\p j$.
Since $P(i)$ belongs to $\mcF(\Deltap)$, there exist the following sequence of submodules of $P(i)$:
\begin{align*}
M_{\ell+1} \subset M_{\ell} \subset \dots \subset M_0 = P(i)
\end{align*}
such that $M_k/M_{k+1} \cong  \Deltap(i_k)$ for $k=0,\dots,\ell$, $i_0=i$ and $i_{\ell}=j$.
Let $i^\p_k$ be the maximal element in $[i_{k-1}, i_k]\setminus\{i_k\}$ for $k=1,\dots,\ell$.
We have a disjoint union of intervals
\[
[i,j]=[i_0,i^\p_1] \sqcup [i_1,i^\p_2] \sqcup \dots \sqcup [i_{\ell-1}, i_{\ell}^\p] \sqcup \{i_{\ell}\}.
\]
For each $k=0,\dots,\ell-1$ and any $m \in [i_k, i_{k+1}^\p]$, $S(m)$ is a composition factor of $\Deltap(i_k)$.
Therefore by Lemma \ref{lem-comp-fac-deltap}, we can apply Lemma \ref{lem-deltad-deltai} to each interval $[i_{k-1}, i_k]=[i_{k-1}, i_k^\p]\sqcup \{i_k\}$.
We have that $i_{k-1}\lhd^{I_{\ast}}i_k$ for each $k=1,\dots,\ell$ and some $\ast\in\{1,2\}$.
Namely, we have $i\lhdp j$.

(2)
Let $\lhd_3$ be an adapted order on $P$ with $[\lhd_3]\preceq [\lhd_1]$ and $[\lhd_3]\preceq[\lhd_2]$.
Let $\Delta_3$ be the standard modules associated to $\lhd_3$.
By Lemma \ref{lem-qhstr-poset}, any composition factor of $\Delta_3(i)$ is a composition factor of both $\Delta_1(i)$ and $\Delta_2(i)$.
By Lemma \ref{lem-comp-fac-deltap}, there exists a surjection from $\Deltap(i)$ to $\Delta_3(i)$.
Therefore, we have $[\lhd_3]\preceq[\lhdp]$.
\end{proof}
Then we complete the proof of Theorem \ref{thm-lattice-tree}.
\begin{proof}[Proof of Theorem \ref{thm-lattice-tree}]
Let $(P,\leq)$ be a diamond-free poset such that $Z_{n}$ is not a full subposet for any $n\geq 4$.
For two quasi-hereditary structures $[\lhd_1]$ and $[\lhd_2]$, consider the following two binary relations on $P$:
\begin{align*}
\lhd^\p = \Bigl(\bigl(\Dec(\lhd_1) \cap \Dec(\lhd_2)\bigr) \cup \Inc(\lhd_1) \cup \Inc(\lhd_2)\Bigr)^{\sf tc}, \\
\lhd^{\p\p} = \Bigl(\Dec(\lhd_1) \cup \Dec(\lhd_2) \cup \bigl(\Inc(\lhd_1) \cap \Inc(\lhd_2)\bigr)\Bigr)^{\sf tc}.
\end{align*}
By Propositions \ref{pro-deltap-antisymmetric} and \ref{pro-deltap-meet}, $\lhd^\p$ defines a quasi-hereditary structure on $A(P)$ and $[\lhd^\p]$ is a meet of $[\lhd_1]$ and $[\lhd_2]$.

Let $B$ be the incidence algebra of the opposite poset $(P, \leq^{\op})$ of $(P,\leq)$.
It is easy to see that the opposite algebra of $A(P)$ is isomorphic to $B$.
Let $\Dec_B$ and $\Inc_B$ be a decreasing and a increasing over $B$, respectively.
By Lemma \ref{lem-dual} (1), for a partial order $\lhd$ on $P$ and $i,j\in P$, we have
\begin{itemize}
\item
$(i,j)\in \Dec(\lhd)$ if and only if $(j,i)\in \Inc_B(\lhd)$.
\item
$(i,j)\in \Inc(\lhd)$ if and only if $(j,i)\in \Dec_B(\lhd)$.
\end{itemize}
Therefore, we have the following equality of binary relations on $P$:
\[
\lhd^{\p\p} = \Bigl(\bigl(\Dec_B(\lhd_1) \cap \Dec_B(\lhd_2)\bigr) \cup \Inc_B(\lhd_1) \cup \Inc_B(\lhd_2)\Bigr)^{\sf tc}.
\]
By applying Propositions \ref{pro-deltap-antisymmetric} to $B$, $\lhd^{\p\p}$ is a partial order on $P$.
By Lemma \ref{lem-dual}, we have an anti-isomorphism of posets between $\qhstr(A(P))$ and $\qhstr(B)$ which is induced from an identity map of the set of partial orders on $P$.
Thus $\lhd_m$ defines a quasi-hereditary structure $[\lhd_m]_B$ on $B$ for $m=1,2$.
By Proposition \ref{pro-deltap-meet}, $\lhd^{\p\p}$ represents a meet of $[\lhd_1]_B$ and $[\lhd_2]_B$ in $\qhstr(B)$.
This implies that $\lhd^{\p\p}$ represents a join of $[\lhd_1]$ and $[\lhd_2]$ in $\qhstr(A(P))$.
\end{proof}
\section*{Acknowledgments}
We are grateful to the anonymous referees for their careful reading and for the suggestions which clearly improved the article. The first author is indebted to his advisor, Henning Krause, for his guidance and for suggesting the study of quasi-hereditary structures as PhD research project, resulting the classification of quasi-hereditary structures on path algebras of type $\A$ as part of his doctoral thesis. The last author wants to dedicate this article to Gabrielle. 
\printbibliography

@Book{coxeter_combinatorics,
 Author = {Anders {Bj\"orner} and Francesco {Brenti}},
 Title = {{``Combinatorics of Coxeter groups".}},
 FJournal = {{Graduate Texts in Mathematics}},
 Journal = {{Grad. Texts Math.}},
 ISSN = {0072-5285},
 Volume = {231},
 Pages = {xiv + 363},
 Year = {2005},
 Publisher = {New York, NY: Springer},
 MSC2010 = {05-01 05E15 20F55 68R15},
 Zbl = {1110.05001}
}

@Article{cline_parshall_scott,
 Author = {E. T. {Cline} and B. J. {Parshall} and L. L. {Scott}},
 Title = {{Finite dimensional algebras and highest weight categories.}},
 FJournal = {{Journal f\"ur die Reine und Angewandte Mathematik}},
 Journal = {{J. Reine Angew. Math.}},
 Volume = {391},
 Pages = {85--99},
 Year = {1988},
 Publisher = {De Gruyter, Berlin},
 MSC2010 = {18E30 17B10 20G05 16Gxx 16P10},
 Zbl = {0657.18005}
}

@incollection {dlab-ringel,
    AUTHOR = {Dlab, Vlastimil and Ringel, Claus Michael},
     TITLE = {The module theoretical approach to quasi-hereditary algebras},
 BOOKTITLE = {Representations of algebras and related topics ({K}yoto,
              1990)},
    SERIES = {London Math. Soc. Lecture Note Ser.},
    VOLUME = {168},
     PAGES = {200--224},
 PUBLISHER = {Cambridge Univ. Press, Cambridge},
      YEAR = {1992},
   MRCLASS = {16E60 (16D90 16G20)},
  MRNUMBER = {1211481},
MRREVIEWER = {Otto Kerner},
}

@article {Ringel91,
    AUTHOR = {Ringel, Claus Michael},
     TITLE = {The category of modules with good filtrations over a
              quasi-he\-re\-di\-ta\-ry algebra has almost split sequences},
   JOURNAL = {Math. Z.},
  FJOURNAL = {Mathematische Zeitschrift},
    VOLUME = {208},
      YEAR = {1991},
    NUMBER = {2},
     PAGES = {209--223},
   MRCLASS = {16G10 (16D90)},
  MRNUMBER = {1128706},
MRREVIEWER = {J. Antonio de la Pe\~{n}a},
}

@article {Ringel10,
    AUTHOR = {Ringel, Claus Michael},
     TITLE = {Iyama's finiteness theorem via strongly quasi-hereditary
              algebras},
   JOURNAL = {J. Pure Appl. Algebra},
  FJOURNAL = {Journal of Pure and Applied Algebra},
    VOLUME = {214},
      YEAR = {2010},
    NUMBER = {9},
     PAGES = {1687--1692},
   MRCLASS = {16G10 (16E10 16G60 16S50)},
  MRNUMBER = {2593693},
MRREVIEWER = {Frauke M. Bleher},
}

@article {CPP19,
    AUTHOR = {Ch\^{a}tel, Gr\'{e}gory and Pilaud, Vincent and Pons, Viviane},
     TITLE = {The weak order on integer posets},
   JOURNAL = {Algebr. Comb.},
  FJOURNAL = {Algebraic Combinatorics},
    VOLUME = {2},
      YEAR = {2019},
    NUMBER = {1},
     PAGES = {1--48},
   MRCLASS = {06A07 (05E10)},
  MRNUMBER = {3912167},
}

@article {Dlab-Ringel-89,
    AUTHOR = {Dlab, Vlastimil and Ringel, Claus Michael},
     TITLE = {Quasi-hereditary algebras},
   JOURNAL = {Illinois J. Math.},
  FJOURNAL = {Illinois Journal of Mathematics},
    VOLUME = {33},
      YEAR = {1989},
    NUMBER = {2},
     PAGES = {280--291},
   MRCLASS = {16A35 (16A48 16A60)},
  MRNUMBER = {987824},
MRREVIEWER = {Alfred G. Wiedemann},
       
}

@article {HU,
    AUTHOR = {Happel, Dieter and Unger, Luise},
     TITLE = {On a partial order of tilting modules},
   JOURNAL = {Algebr. Represent. Theory},
  FJOURNAL = {Algebras and Representation Theory},
    VOLUME = {8},
      YEAR = {2005},
    NUMBER = {2},
     PAGES = {147--156},
   MRCLASS = {16G10 (16E10 16G70)},
  MRNUMBER = {2162278},
MRREVIEWER = {Apostolos D. Beligiannis},
}

@article {RS,
    AUTHOR = {Riedtmann, Christine and Schofield, Aidan},
     TITLE = {On a simplicial complex associated with tilting modules},
   JOURNAL = {Comment. Math. Helv.},
  FJOURNAL = {Commentarii Mathematici Helvetici},
    VOLUME = {66},
      YEAR = {1991},
    NUMBER = {1},
     PAGES = {70--78},
   MRCLASS = {16G20 (18G99 55U99)},
  MRNUMBER = {1090165},
MRREVIEWER = {Luise Unger},
}

@phdthesis{C16,
    title    = {On certain strongly quasihereditary algebras},
    school   = {University of Oxford},
    author   = {Conde, Teresa},
    publisher = {University of Oxford},
    year     = {2016},
}

@article {G81,
    AUTHOR = {Gabriel, Peter},
     TITLE = {Un jeu? Les nombres de Catalan},
  JOURNAL = {Uni Z\"{u}rich, Mitteilungsblatt des Rektorats},
    VOLUME = {6},
      YEAR = {1981},
     PAGES = {4--5},
}

@ARTICLE{C19,
       author = {{Coulembier}, Kevin},
        title = "{The classification of blocks in BGG category $\mathcal{O}$}",
      journal = {arXiv e-prints},
         year = "2019",
        month = "03",
archivePrefix = {arXiv},
       eprint = {1903.02717},
}

@incollection {scott,
    AUTHOR = {Scott, Leonard L.},
     TITLE = {Simulating algebraic geometry with algebra. {I}. {T}he
              algebraic theory of derived categories},
 BOOKTITLE = {The {A}rcata {C}onference on {R}epresentations of {F}inite
              {G}roups ({A}rcata, {C}alif., 1986)},
    SERIES = {Proc. Sympos. Pure Math.},
    VOLUME = {47},
     PAGES = {271--281},
 PUBLISHER = {Amer. Math. Soc., Providence, RI},
      YEAR = {1987},
   MRCLASS = {20G05 (16A64 18E30)},
  MRNUMBER = {933417},
MRREVIEWER = {H. H. Andersen},
}

@manual{sage,
  shorthand = {Sag19},
  Key          = {SageMath},
  Author       = {{The Sage Developers}},
  Title        = {{S}ageMath, the {S}age {M}athematics {S}oftware {S}ystem ({V}ersion 8.8)},
  url         = { https://www.sagemath.org},
  Year = {2019}
}

@manual{qpa,
  shorthand = {QPA19},
  Key          = {QPA},
  Author       = {{The QPA-team}},
  Title        = { {QPA} - {Q}uivers, path algebras and representations - a {GAP} package, Version 1.30},
  url         = { https://folk.ntnu.no/oyvinso/QPA/},
  Year         = {2019}
}

@online{code,
  author = {Flores, Manuel and Rognerud, Baptiste},
  title = {Quasi-hereditary structures, 2020, Research code ({V}ersion 1.0)},
  Year         = {2020},
  url = {https://webusers.imj-prg.fr/~baptiste.rognerud/en/quasi_hereditary_structures.html},
}

@online {OEIS,
    shorthand = {OEI20},
    author = {{OEIS Foundation Inc. (2020)}},
    title = {The On-Line Encyclopedia of Integer Sequences.},
    URL = {https://oeis.org/},
}

@article {H06,
    AUTHOR = {Hille, L.},
     TITLE = {On the volume of a tilting module},
   JOURNAL = {Abh. Math. Sem. Univ. Hamburg},
  FJOURNAL = {Abhandlungen aus dem Mathematischen Seminar der Universit\"{a}t
              Hamburg},
    VOLUME = {76},
      YEAR = {2006},
     PAGES = {261--277},
   MRCLASS = {16G30 (18G35)},
  MRNUMBER = {2293445},
MRREVIEWER = {David Smith},
}

@Article{rognerud20,
 Author = {Baptiste {Rognerud}},
 Title = {{Exceptional and modern intervals of the Tamari lattice}},
 FJournal = {{S\'eminaire Lotharingien de Combinatoire}},
 Journal = {{S\'emin. Lothar. Comb.}},
 ISSN = {1286-4889},
 Volume = {79},
 Pages = {b79d, 23},
 Year = {2020},
 Publisher = {Universit\"at Wien, Fakult\"at f\"ur Mathematik, Wien},
 Language = {English},
 MSC2010 = {06A07 05C05 05E16},
}

@article{CP15,
title = {Counting smaller elements in the Tamari and m-Tamari lattices},
journal = {Journal of Combinatorial Theory, Series A},
volume = {134},
pages = {58-97},
year = {2015},
author = {Gr\'egory Ch\^atel and Viviane Pons}
}

\end{document}